\documentclass[onefignum,onetabnum]{siamart251216}

\usepackage{amsmath,amssymb,mathtools}
\usepackage{bm,enumitem,multirow,siunitx,xpatch}
\usepackage{appendix}
\usepackage{amsopn}
\usepackage{lipsum}
\usepackage{amsfonts}
\usepackage{epstopdf}
\usepackage{algorithm,algpseudocode}
\usepackage{enumitem}
\usepackage{blkarray,nicematrix}

\usepackage{graphicx,subcaption}
\graphicspath{{fig/}}

\setlist[enumerate]{label={(\arabic*)}, leftmargin=*}
\ifpdf
  \DeclareGraphicsExtensions{.eps,.pdf,.png,.jpg}
\else
  \DeclareGraphicsExtensions{.eps}
\fi

\usepackage{blkarray,nicematrix}

\numberwithin{equation}{section}

\newcommand{\macheps}{\boldsymbol{u}}

\newcommand{\wt}{\tilde}
\newcommand{\wh}{\hat}
\newcommand{\tightcolon}{\mathord{:}}
\newcommand{\sffl}{\textup{fl}}
\newcommand{\sfgivens}{\textup{rot}}
\newcommand{\mygiv}[2]{{\sfgivens}(#1,#2)} 
\newcommand{\sfvec}{\textup{vec}}
\newcommand{\tsqrt}[1]{\textstyle{\sqrt{#1}}}
\newcommand{\mymat}[4][]{{#2}_{#3,#4}^{#1}}
\newcommand{\trans}{^\top}
\newcommand{\dspace}{\qquad}
\newcommand{\calK}{{K}}
\newcommand{\tol}{\lVert\mymat{\check{A}}{2}{1}\rVert_{\max}}
\newcommand{\forbackerr}[3]{e(#1,#2,#3)}

\DeclareMathOperator{\diag}{diag}

\DeclareMathOperator{\trace}{trace}

\algnewcommand{\bfInput}{\item[\textbf{Input:}]}
\algnewcommand{\bfOutput}{\item[\textbf{Output:}]}

\newsiamremark{remark}{Remark}
\newsiamremark{hypothesis}{Hypothesis}
\crefname{hypothesis}{Hypothesis}{Hypotheses}
\newsiamthm{claim}{Claim}

\headers{Givens-Based Eigenvalue Swapping}{M. Shao and R. Xu}

\title{A Givens-Based Direct Eigenvalue Swapping Algorithm with Rounding Error
Analysis%
\thanks{Submitted to the editors DATE
}}

\author{Meiyue Shao\thanks{School of Data Science and Shanghai Key Laboratory for Contemporary Applied Mathematics, Fudan University, Shanghai 200433, China (\email{myshao@fudan.edu.cn}).}
\and Rentao Xu\thanks{School of Mathematical Sciences, Fudan University, Shanghai 200433, China (\email{rtxu24@m.fudan.edu.cn}).}}

\makeatletter
\newcommand*{\addFileDependency}[1]{
  \typeout{(#1)}
  \@addtofilelist{#1}
  \IfFileExists{#1}{}{\typeout{No file #1.}}
}
\makeatother

\usepackage{hyperref}
\begin{document}

\maketitle

\begin{abstract} 
For a given matrix \( A \) in the real Schur form, the eigenvalue swapping algorithm is used to swap two adjacent diagonal blocks.
In LAPACK, this operation is performed by the subroutine~\texttt{xLAEXC}, which is based on a direct swapping algorithm proposed by Bai and Demmel in 1993. 
In this work, we propose a modified direct swapping algorithm for two \( 2 \times 2 \) blocks, replacing Householder QR with a pivoted Givens QR factorization. 
A thorough error analysis shows that the modified algorithm can achieve an error bound of the form \( O( \macheps \lVert A \rVert ) \) under mild assumptions, where \( \macheps \) is the unit roundoff. 
We present empirical evidence to illustrate the superior stability of our algorithm over existing approaches.
\end{abstract}



\section{Introduction}
Given a matrix \( A \) in the real Schur form:
\[
A = \begin{bmatrix}
    \mymat{A}{1}{1} & \mymat{A}{1}{2} & \dots & \mymat{A}{1}{k} \\
     & \mymat{A}{2}{2} & \dots & \mymat{A}{2}{k} \\
     &  & \ddots & \vdots \\
     &  &  & \mymat{A}{k}{k}
\end{bmatrix} ,
\]
where each diagonal block \( \mymat{A}{i}{i} \) is either \( 1 \times 1 \) or
\( 2 \times 2 \).
In particular, for each \( 2 \times 2 \) diagonal block, its diagonal entries
are equal and the off-diagonal entries have opposite signs.
Reordering eigenvalues in \( A \) means updating \( A \) by an orthogonal
matrix \( Q \) such that \( Q\trans AQ \) maintains the real Schur form while
the diagonal blocks are rearranged in a desired order.

Eigenvalue reordering serves as a vital ingredient in several algorithms.
For example, the aggressive early deflation technique~\cite{Braman02} performs
reordering to detect deflatable eigenvalues in the multishift QR algorithm.
Similarly, in the Krylov--Schur method~\cite{Stewart02_2,Stewart02},
reordering is used to move unwanted Ritz values to the trailing block.
In several variants of the QZ algorithm~\cite{Fokkema98,Kragstrom07,Steel21},
eigenvalue reordering in the generalized real Schur form also arises.
Another notable application of eigenvalue reordering is the Schur method for
solving algebraic Riccati equations~\cite{Laub79}.

In general, eigenvalue reordering is achieved by a sequence of adjacent swaps.
Thus, it suffices to consider the swapping problem in \( A \) with only two
diagonal blocks.
Early studies on this topic can be found in \cite{Stewart76, Dooren81,
Dooren82}.
In these works, the shift technique from the QR algorithm is introduced to
address the swapping problem.
Although the swapping is typically accomplished within a few iterations, these
algorithms may still fail for some ill-conditioned problems.

As a simpler and numerically more stable alternative, a class of swapping
methods based on the computation of invariant subspaces, known as \emph{direct
methods}, has been developed in~\cite{Bai1993,Camps19,Cao81,Parlett88,Sai2022}.
Stable algorithms for eigenvalue swapping in the real or generalized real
Schur form that involve \( 1 \times 1 \) blocks can be found in~\cite{Sai2022,
Dooren81}.
In particular, Sai~\cite{Sai2022} proves her method, which combines explicitly
shifted QR and Givens rotations, is stable when either of the diagonal blocks
is \( 1 \times 1 \).
However, the numerical stability of these methods in~\cite{Cao81,Parlett88,Sai2022} is not satisfactory when applied to some
ill-conditioned \( 2 \times 2 \) blocks.
In 1993, Bai and Demmel~\cite{Bai1993} developed an improved direct swapping
algorithm.
Their method is motivated by the following theorem.

\begin{theorem}[Ng and Parlett~\cite{Parlett88}] \label{theo: iff condition for reordering}
    Suppose \( A \in \mathcal{T}_{4} \). Let \( X \) be a solution to the \textup{Sylvester} equation \( \mymat{A}{1}{1}X - X\mymat{A}{2}{2} = \mymat{A}{1}{2} \). Then an orthogonal matrix \( Q \in \mathbb{R}^{4 \times 4} \) swaps \( A \) if and only if 
    \[
    Q\trans
    \begin{bmatrix}
        -X  \\
        I_{2}
    \end{bmatrix} = 
    \begin{bmatrix}
        Z \\
        0
    \end{bmatrix} 
    \]
    for some nonsingular matrix \( Z \). 
\end{theorem}

Bai and Demmel's direct swapping method, which appears as Algorithm~\ref{alg: direct swapping method}, consists of two main steps: solving the Sylvester equation \( \mymat{A}{1}{1}X - X\mymat{A}{2}{2} = \gamma \mymat{A}{1}{2} \) and computing the QR factorization of \( [-X\trans, \gamma I_{2}]\trans \). Note that when there is no risk of overflow, \( \gamma \) is simply taken as~\( 1 \).
Neglecting the error in the QR factorization, Bai and Demmel provide an error analysis for their algorithm.  
\begin{theorem}[Bai and Demmel~\cite{Bai1993}] \label{theo: Bai and Demmel error bound}
    Suppose \( A \in \mathcal{S}_{4} \). Let \( \wh{X} \) be the computed~\( X \) in Algorithm~\ref{alg: direct swapping method} and \( R = \mymat{A}{1}{1}\wh{X} - \wh{X}\mymat{A}{2}{2} - \mymat{A}{1}{2} \) be the residual. Assume the computation of \textup{QR} factorization is exact, which results in an orthogonal matrix~\( \wh{Q} \). Then 
    \[
    \check{A} =  
    \wh{Q}\trans A\wh{Q} = \begin{bmatrix}
        \mymat{\check{A}}{1}{1} & \mymat{\check{A}}{1}{2} \\
        \mymat{\check{A}}{2}{1} & \mymat{\check{A}}{2}{2}
    \end{bmatrix} , \dspace \mymat{\check{A}}{i}{j} \in \mathbb{R}^{2 \times 2} ,
    \]
    where 
    \[
    \lVert \mymat{\check{A}}{2}{1} \rVert \le \frac{ \lVert R \rVert_{\mathrm{F}} }{ 1 + ( \sigma_{2}( \wh{X} ) )^{2} } .
    \]
    Here \( \lVert \mymat{\check{A}}{2}{1} \rVert \) is the \( 2 \)-norm of \( \mymat{\check{A}}{2}{1} \), \( \lVert R \rVert_{\mathrm{F}} \) is the Frobenius norm of \( R \) and \( \sigma_{2}(\wh{X}) \) is the second largest singular value of \( \wh{X} \). 
\end{theorem}

Algorithm~\ref{alg: direct swapping method} has been adopted as the subroutine
for eigenvalue swapping in LAPACK, namely, \texttt{xLAEXC}.
The \( \mymat{\check{A}}{2}{1} \) block is used for the stability test---when
the largest entry of \( \mymat{\check{A}}{2}{1} \) exceeds a certain size, the
swapping is regarded as unstable, and is rejected.
Fortunately, failure of swapping in \texttt{xLAEXC} is quite rare in practice.
Recently, Camps et~al.~\cite{Camps19} proposed a refinement of
Algorithm~\ref{alg: direct swapping method} that further reduces
\( \lVert\mymat{\check{A}}{2}{1}\rVert \) and enhances the stability through
iterative refinement.
Algorithm~\ref{alg: direct swapping method} can also be adapted to the
generalized real Schur form; see~\cite{Camps19,Kagstrom93,Kagstrom96}.

Despite the practical reliability observed for these algorithms, a complete
rounding error analysis is largely absent from the literature.
It is unclear whether it is safe to ignore rounding errors in the QR
factorization as in Theorem~\ref{theo: Bai and Demmel error bound}.
And the bound of \( \lVert \mymat{\check{A}}{2}{1} \rVert \) in
Theorem~\ref{theo: Bai and Demmel error bound}, which is proportional to \(
\lVert R \rVert_{\mathrm{F}} \), is not sharp enough in some ill-conditioned
cases. As a result, \texttt{xLAEXC} still retains failure reporting for the
numerically unstable case. Higher-level algorithms that rely on eigenvalue
reordering must likewise retain the corresponding failure-handling branches. 

In this paper, we restrict our attention to the swapping of two \( 2 \times 2 \) blocks in the real Schur form. The primary contributions of this paper are summarized as follows.
\begin{itemize}
    \item A mixed forward--backward error model is introduced for eigenvalue
swapping.
Our error model overcomes the limitation that a pure backward error analysis
is not suitable for eigenvalue swapping.
See Section~\ref{sec: main result} for details.

    \item We propose a strategy that uses pivoted Givens QR to replace the
Householder QR employed in Algorithm~\ref{alg: direct swapping method}.
A comprehensive error analysis shows that this modified algorithm has provable
stability under mild assumptions.

    \item Concrete examples are provided to show that our modified algorithm
can be more stable than existing eigenvalue swapping algorithms
in~\cite{Bai1993, Camps19}.
\end{itemize}

The rest of this paper is organized as follows.
Section~\ref{sec: preliminaries} introduces the notation and preliminary
results.
Section~\ref{sec: pivoted Givens QR} presents the new Givens-based strategy
for the QR factorization, which results in
Algorithm~\ref{alg: modified direct swapping method}, a modification of
Algorithm~\ref{alg: direct swapping method}.
The error analysis of Algorithm~\ref{alg: modified direct swapping method} is
given in Section~\ref{sec: main result}.
Section~\ref{sec: numerical experiments} presents numerical experiments
comparing the stability of Algorithm~\ref{alg: modified direct swapping method}
with the eigenvalue swapping algorithms in~\cite{Bai1993,Camps19}.
Section~\ref{sec: strategy for the error analysis} outlines the proof of the
error analysis by decomposing the overall error into contributions from
solving a linear system and from the QR factorization.
Sections~\ref{sec: contribution to error of solving linear system}
and~\ref{sec: contribution to error from QR factorization} bound these two
contributions, respectively.
Section~\ref{sec: conclusion} concludes the paper and suggests future work.
Detailed proofs of some propositions and lemmas are given in
Appendix~\ref{appendix}.

\begin{algorithm}[tb]
    \caption{Direct swapping method}
    \label{alg: direct swapping method}
    \begin{algorithmic}[1] 
    \bfInput A matrix \( A \in \mathcal{S}_{4} \). 
    \State Let \( T = I_{2} \otimes \mymat{A}{1}{1} - \mymat[\top]{A}{2}{2} \otimes I_{2} \), \( b = \sfvec( \mymat{A}{1}{2} ) \). Solve the linear system \( Tx = \gamma b \) by GECP. Take \( X = \sfvec^{-1}( x ) \).
    \State Compute the QR factorization of \( H = [-X\trans, \gamma I_{2}]\trans \) by Householder reflections such that 
    \[
    \begin{bmatrix}
        -X \\
        \gamma I_{2}
    \end{bmatrix} = Q
    \begin{bmatrix}
        Z \\
        0
    \end{bmatrix} .
    \]
    \State Stability test. If the swap is accepted, compute \( Q\trans AQ \) and reduce it to the real Schur form. 
    \bfOutput \( Q\trans AQ \in \mathcal{S}_{4} \).
    \end{algorithmic}
\end{algorithm}

\section{Preliminaries} \label{sec: preliminaries}
In this section, we first introduce some notation for this paper. Next we present several lemmas that will be used in the subsequent analysis.

\subsection{Notation}
Let \( \macheps \) be the unit roundoff. We write \( \gamma_{k} \) to denote the quantity which satisfies
\[
\gamma_{k} =  ( 1 - \macheps )^{-k} - 1 .
\]
In particular, we always assume that \( \macheps \) is small enough to ensure \( \gamma_{k} \ll 1 \) holds for constants \( k > 0 \) that appear in this paper. This assumption is not critical but avoids some unnecessary arguments. The notation \( \sffl( \cdot ) \) is used to identify the computation in floating point arithmetic. Throughout this paper, we follow the standard model of floating point arithmetic:
\[
\sffl( x~\textup{op}~y ) 
= ( x~\textup{op}~y ) ( 1 + \delta ), \dspace 
\lvert \delta \rvert 
\le \macheps, \dspace 
\textup{op} \in \lbrace+, -, *, /\rbrace.
\]
For two real numbers \( a \) and \( b \), denote 
\[
a \wedge b = \min\{ a, b \}, \dspace   
a \vee b = \max \{ a, b \} . 
\]

The operation \( \sfvec( \cdot ) \) maps a matrix into a column vector by concatenating its columns from left to right while \( \sfvec^{-1}( \cdot ) \) maps a vector of dimension \( 4 \) back into a \( 2 \times 2 \) matrix. For a vector \( v = [v_{1}, v_{2}]\trans \in \mathbb{R}^{2} \), the notation \( G = \mygiv{v_{1}}{v_{2}} \) refers to the Givens rotation such that 
\begin{equation} \label{structure of givens rotation}
G
= \begin{bmatrix}
      c & s \\
      -s & c
  \end{bmatrix}, \dspace 
c = \frac{ v_{1} }{ \sqrt{ v_{1}^{2} + v_{2}^{2} } }, \dspace 
s = \frac{ -v_{2} }{ \sqrt{ v_{1}^{2} + v_{2}^{2} } } . 
\end{equation} 
Then \( G\trans v \) zeros the second entry of \( v \). 

The \( 2 \)-norm and the \( \infty \)-norm of a vector or matrix are denoted by \( \lVert \cdot \rVert \) and \( \lVert \cdot \rVert_{\infty} \), respectively. For a matrix \( M \), we write \( \lVert M \rVert_{\max} \) for \( \lVert \sfvec( M ) \rVert_{\infty} \). The Frobenius norm of \( M \) is \( \lVert M \rVert_{\mathrm{F}} \). The condition number \( \kappa( M ) = \lVert M \rVert \lVert M^{-1} \rVert \) with the convention \( \kappa( M ) = \infty \) if \( M \) is singular. The \( i \)th largest singular value of \( M \) is denoted by \( \sigma_{i}( M ) \). Moreover, for \( M \in \mathbb{R}^{2 \times 2} \), we introduce a quantity 
\[
\calK( M ) = \frac{\sigma_{1}( M )}{\sqrt{ 1 + 2\sigma_{2}^{2}( M ) }} .
\]
It is evident that \( \calK( M ) \le \sigma_{1}( M ) \wedge \kappa( M ) \). This quantity will be used frequently in the subsequent error analysis. 

The notation \( \mathcal{T}_{4} \) is used to denote the set of block upper triangular matrices, partitioned as 
\[
\begin{bmatrix}
    \mymat{A}{1}{1} & \mymat{A}{1}{2} \\
                    & \mymat{A}{2}{2}
\end{bmatrix}, \dspace \mymat{A}{i}{j} \in \mathbb{R}^{2 \times 2} .
\]
Moreover, \( \mathcal{S}_{4} \) denotes the subset of \( \mathcal{T}_{4} \) consisting of matrices in the real Schur form. For \( A \in \mathcal{T}_{4} \) or \( \mathcal{S}_{4} \), the shorthand ``\( Q \) swaps \( A \)'' indicates that \( Q\trans A Q \) remains block upper triangular (not necessarily in the real Schur form), with the eigenvalues of \(  \mymat{A}{1}{1} \) and \( \mymat{A}{2}{2} \) swapped. 

\subsection{Auxiliary lemmas}
We begin with a classical result~\cite[Theorem 9.4]{hig02} on the componentwise backward stability of Gaussian elimination with complete pivoting (GECP). Note that the original theorem in~\cite{hig02} is stated without pivoting, but as observed in~\cite[pp.~163]{hig02}, it extends naturally to complete pivoting.
For a matrix \( T \in \mathbb{R}^{n \times n} \), let \( PTQ\trans = LU \) be its LU factorization with complete pivoting, where \( L \), \( U \) are LU factors and \( P \), \( Q \) are permutations.

\begin{lemma}[{\cite[Theorem 9.4]{hig02}}] \label{lemma: stability of GECP}
    Suppose \( T \in \mathbb{R}^{n \times n} \) and \textup{GECP} is used to solve the linear system \( Tx = b \). Then the computed \( \wh{L} \), \( \wh{U} \), \( \wh{P} \), \( \wh{Q} \), and the computed solution \( \wh{x} \) satisfy 
    \[
    ( T + \Delta T ) \wh{x} 
    = b, \dspace 
    \lvert \Delta T \rvert 
    \le \gamma_{3n} \wh{P}\trans \lvert \wh{L} \rvert \lvert \wh{U} \rvert \wh{Q} .
    \]
\end{lemma}

Let \( r = T\wh{x} - b = -\Delta T \cdot \wh{x} \) be the residual. Lemma~\ref{lemma: stability of GECP} leads to helpful consequences~\cite[Theorem 9.5]{hig02}
\begin{align} \label{componentwise bound for residual} 
    \textup{(componentwise)} \quad &\lvert r \rvert 
    \le \gamma_{3n} \wh{P}\trans \lvert \wh{L} \rvert \lvert \wh{U} \rvert \wh{Q} \lvert \wh{x} \rvert, \\ 
    \label{infinite norm bound for residual}
    \textup{(normwise)} \quad &\lVert r \rVert_{\infty} 
    \le n^{2} \rho_{n} \gamma_{3 n} \lVert T \rVert_{\infty} \lVert \wh{x} \rVert_{\infty} ,
\end{align} 
where \( \rho_{n} \) is the growth factor in Gaussian elimination. In particular, it is known that \( \rho_{4} = 4 \) for complete pivoting~\cite{Cryer1968,Wilkinson61}. 

The next lemma provides the norm conversions needed in later analysis. It is a straightforward application of matrix norm properties~\cite[Section 2.3.2]{Golub13}. 
\begin{lemma} \label{lemma: norm inequalities}
    Suppose \( A \in \mathcal{S}_{4} \). Let \( T = I_{2} \otimes \mymat{A}{1}{1} - \mymat[\top]{A}{2}{2} \otimes I_{2} \). Then
    \[ 
    \lVert T \rVert_{\infty} 
    \le 2\sqrt{2} \, \lVert A \rVert, \dspace
    \lVert T \rVert  
    \le 2 \lVert A \rVert, \dspace
    \lVert T \rVert 
    \le \sqrt{2} \, \lVert A \rVert_{\mathrm{F}}, 
    \] 
\end{lemma} 

The following lemma describes how the computed Givens rotations may be affected by perturbations. Since~\cite{Sai2022} is written in Chinese and its setting differs slightly from ours, we provide a brief proof in Appendix~\ref{app: Proof of Lemma perturbations of GR}.

\begin{lemma}[{\cite[Section~2]{Sai2022}}] \label{lemma: perturbations of GR} 
    Suppose \( v = [v_{1}, v_{2}]\trans \) and \( \wt{v} = [\wt{v}_{1}, \wt{v}_{2}]\trans \in \mathbb{R}^{2} \) satisfy \( \lvert \wt{v} - v \rvert \le \gamma_{k} \lvert v \rvert \). Let \( c \), \( s \) be constructed exactly from \( v \) according to~(\ref{structure of givens rotation}), and \( \wh{c} \), \( \wh{s} \) be the computed values from \( \wt{v} \) using the same construction. Then 
    \[
    \lvert \wh{c} - c \rvert \le \gamma_{2k+4}  \lvert c \rvert , \dspace 
    \lvert \wh{s} - s \rvert \le \gamma_{2k+4}  \lvert s \rvert .
    \]
    Moreover, if we set \( g = \sqrt{ v_{1}^{2} + v_{2}^{2} } \) and \( \wh{g} = \sffl \big( \sqrt{ \wt{v}_{1}^{2} + \wt{v}_{2}^{2} } \big) \), then 
    \[
    \lvert \wh{g} - g \rvert \le \gamma_{k+3} \cdot g  .
    \]
\end{lemma}

In addition, we present a componentwise backward error result on Givens rotations, whose proof can be found in Appendix~\ref{app: proof of Lemma componentwise backward error result on GR}.

\begin{lemma} \label{lemma: componentwise backward error result on GR}
    Suppose \( X \in \mathbb{R}^{2 \times 2} \setminus \{0\} \) and the largest entry of \( \lvert X \rvert \) is in its first column. Let \( G = \mygiv{\mymat{x}{1}{1}}{\mymat{x}{2}{1}} \) and \( \wh{G} \) be the computed \( G \). Then   
    \[
    \sffl \Big(
    \wh{G}\trans 
    \begin{bmatrix}
        \mymat{x}{1}{2} \\
        \mymat{x}{2}{2}
    \end{bmatrix} \Big) = G\trans \begin{bmatrix}
        \mymat{x}{1}{2} + \Delta \mymat{x}{1}{2} \\
        \mymat{x}{2}{2} + \Delta \mymat{x}{2}{2} 
    \end{bmatrix} , 
    \]
    where 
    \[
    \lvert \Delta \mymat{x}{1}{2} \rvert 
    \le \gamma_{6} \big( 5 \lvert \mymat{x}{1}{2} \rvert + 8 \sigma_2( X ) \big) , \dspace 
    \lvert \Delta \mymat{x}{2}{2} \rvert 
    \le \gamma_{6} \big( 5 \lvert \mymat{x}{2}{2} \rvert + 8 \sigma_2( X ) \big) .
    \]
\end{lemma}

The final lemma provides a componentwise perturbation result for the SVD of a \( 2 \times 2 \) matrix. See Appendix~\ref{app: proof of stability of SVD} for its proof. 
\begin{lemma} \label{lemma: stability of SVD}
    Let \( c \) be a positive number with \( c < 1 \). For a matrix \( X \in \mathbb{R}^{2 \times 2} \), suppose the \textup{SVD} of \( X \) is \( X = U \Sigma V\trans \) and \( c + \kappa^{-1}( X ) < 1 \). Let \( \Delta X \) be a perturbation matrix satisfying \( \lVert \Delta X \rVert \le c \lVert X \rVert \). Suppose the \textup{SVD} of \( \wt{X} = X + \Delta X \) is \( \wt{X} = \wt{U}\wt{\Sigma}\wt{V}\trans \). Then for \( 1 \le i, j \le 2 \), 
    \[
    \lvert \mymat{\wt{u}}{i}{j} \rvert \le \lvert \mymat{u}{i}{j} \rvert + d + d^{2}, \dspace 
    \lvert \mymat{\wt{v}}{i}{j} \rvert \le \lvert \mymat{v}{i}{j} \rvert + d + d^{2} , 
    \] 
    where \( d = d( c, \kappa( X ) ) = c / ( 1 - c - \kappa^{-1}( X ) ) \).
\end{lemma}

\section{Pivoted Givens QR} \label{sec: pivoted Givens QR}
Suppose a matrix \( H \in \mathbb{R}^{4 \times 2} \) takes the form  
\[
\begin{bmatrix}
          -X \\
           I_{2}
\end{bmatrix} =
\begin{bmatrix} 
    \times & \times \\
    \times & \times \\
    1      & 0      \\
    0      & 1
\end{bmatrix} .
\]
Our task is to find a \( 4 \times 4 \) orthogonal matrix \( Q \) such that 
\[
Q\trans H = 
\begin{bmatrix}
   Z \\
   0
\end{bmatrix} , 
\]  
where \( Z \) is a \( 2 \times 2 \) nonsingular matrix. Again, we emphasize that \( Z \) is not necessarily upper triangular.

Unlike the Householder QR employed in~\cite{Bai1993}, we introduce a strategy based on Givens rotations to address this issue. When the largest entry of \( \lvert X \rvert \) is in its first column, \( Q \) is accumulated via the following procedure:
\begin{equation} \label{order of rotations if col1>col2}
    \begin{bmatrix}
        \pmb{\times} & \times \\
        \pmb{\times} & \times \\
        1      & 0      \\
        0      & 1
    \end{bmatrix} 
    \xrightarrow[(2, 1)]{(1, 1)}
    \begin{bmatrix}
        \pmb{\times} & \times \\
        0      & \times \\
        \pmb{1}      & 0      \\
        0      & 1
    \end{bmatrix} 
    \xrightarrow[(3, 1)]{(1, 1)} 
    \begin{bmatrix}
        \times & \times \\
        0      & \pmb{\times} \\
        0      & \times \\
        0      & \pmb{1}
    \end{bmatrix} 
    \xrightarrow[(4, 2)]{(2, 2)} 
    \begin{bmatrix}
        \times & \times \\
        0      & \pmb{\times} \\
        0      & \pmb{\times}  \\
        0      & 0
    \end{bmatrix} 
    \xrightarrow[(3, 2)]{(2, 2)}
    \begin{bmatrix}
        \times & \times \\
        0      & \times \\
        0      & 0      \\
        0      & 0
    \end{bmatrix} .
\end{equation}
When the largest entry is in the second column, the accumulation proceeds as follows:
\begin{equation} \label{order of rotations if col1<=col2}
    \begin{bmatrix}
        \times & \pmb{\times} \\
        \times & \pmb{\times} \\
        1      & 0      \\
        0      & 1
    \end{bmatrix}
    \xrightarrow[(2, 2)]{( 1, 2 )}
    \begin{bmatrix}
        \times & \pmb{\times} \\
        \times & 0     \\
        1      & 0      \\
           0      & \pmb{1}
    \end{bmatrix} 
     \xrightarrow[(4, 2)]{(1, 2)}
      \begin{bmatrix}
         \times & \times \\
         \pmb{\times} & 0     \\
         \pmb{1}      & 0      \\
         \times & 0
     \end{bmatrix} 
     \xrightarrow[(3, 1)]{(2, 1)}
     \begin{bmatrix}
        \times & \times \\
        \pmb{\times} & 0     \\
        0      & 0      \\
        \pmb{\times} & 0
    \end{bmatrix} 
    \xrightarrow[(4, 1)]{(2, 1)}
    \begin{bmatrix}
        \times & \times \\
        \times & 0     \\
        0      & 0      \\
        0      & 0
    \end{bmatrix} .
\end{equation} 
Let \( G_{i} \in \mathbb{R}^{4 \times 4} \) denote the \( i \)th Givens rotation in the reduction. Then we have \( Q = G_{1}G_{2}G_{3}G_{4} \). We refer to this QR factorization strategy as the pivoted Givens QR. The corresponding modified direct swapping algorithm is described in Algorithm~\ref{alg: modified direct swapping method}. 

\begin{algorithm}[tb]
    \caption{Modified direct swapping method}
    \label{alg: modified direct swapping method}
    \begin{algorithmic}[1] 
    \bfInput A matrix \( A \in \mathcal{S}_{4} \).
    \State Let \( T = I_{2} \otimes \mymat{A}{1}{1} - \mymat[\top]{A}{2}{2} \otimes I_{2} \), \( b = \sfvec( \mymat{A}{1}{2} ) \). Solve the linear system \( Tx = \gamma b \) by GECP. Take \( X = \sfvec^{-1}( x ) \). 
    \State Compute the QR factorization of \( H = [-X\trans, \gamma I_{2}]\trans \) by the pivoted Givens QR such that 
    \[
    \begin{bmatrix}
        -X \\
        \gamma I_{2}
    \end{bmatrix} = Q
    \begin{bmatrix}
        Z \\
        0
    \end{bmatrix} .
    \]
    \State Stability test. If the swap is accepted, compute \( Q\trans AQ \) and reduce it to the real Schur form. 
    \bfOutput \( Q\trans AQ \in \mathcal{S}_{4} \).
    \end{algorithmic}
\end{algorithm}  

\section{Main result} \label{sec: main result}
Before presenting the main theoretical result of this work, we first clarify the meaning of error in eigenvalue swapping. Denote the compact set of orthogonal matrices that swap \( A \in \mathcal{T}_{4} \) by \( \mathcal{P} ( A ) \). The distance between a matrix \( M \in \mathbb{R}^{4 \times 4} \) and \( \mathcal{P}( A ) \) is denoted by 
\[
\textup{dist}( M, \mathcal{P}( A ) ) = \min_{Q \in \mathcal{P}( A )} \lVert M - Q \rVert . 
\]
The classical backward error perturbs the input so that the computed solution becomes exact for the perturbed problem. However, eigenvalue swapping demands an orthogonal matrix. Thus, a computed matrix that is only approximately orthogonal cannot be an exact solution for a swapping problem. To this end, we introduce the mixed forward--backward error~\( \forbackerr{A}{\wh{Q}}{\theta} \) for eigenvalue swapping: \( \forbackerr{A}{\wh{Q}}{\theta} \) is the smallest perturbation~\( \Delta A \) (in the spectral norm) such that the approximate solution~\( \wh{Q} \) is within distance~\( \theta \) of the set of exact solutions of the perturbed problem~\( A + \Delta A \). Clearly, \( \forbackerr{A}{\wh{Q}}{\theta} \) is monotonically decreasing in~\( \theta \).

\begin{definition} \label{def: error in eigenvalue reordering}
Suppose \( A \in \mathcal{T}_{4} \) and \( \wh{Q} \in \mathbb{R}^{4 \times 4} \) is approximately orthogonal. Given a nonnegative \( \theta \) satisfying 
\[
\theta \ge \min_{ Q\trans Q = I_{4} } \lVert \wh{Q} - Q \rVert . 
\]
The error associated with \( A \), \( \wh{Q} \) and \( \theta \) in eigenvalue swapping is defined as
\[
\begin{aligned}
    \forbackerr{A}{\wh{Q}}{\theta} = \inf \big\{ \epsilon \ge 0 \colon 
    &\text{ there exists } \Delta A \in \mathcal{T}_{4}, 
    \text{ s.t. } \lVert \Delta A \rVert \le \epsilon \lVert A \rVert, \\
    &\; \textup{dist}( \wh{Q}, \mathcal{P}( A + \Delta A ) ) \le \theta \big\} .
\end{aligned}
\] 
\end{definition} 

From Definition~\ref{def: error in eigenvalue reordering}, for all \( \epsilon > \forbackerr{A}{\wh{Q}}{\theta} \), there exist \( \Delta A \in \mathcal{T}_{4} \) and orthogonal \( \wt{Q} \in \mathbb{R}^{4 \times 4} \) such that \( \wt{Q} \) swaps \( A + \Delta A \) and
\[
\lVert \Delta A \rVert 
\le \epsilon \lVert A \rVert, \dspace \lVert \wt{Q} - \wh{Q} \rVert \le \theta . 
\]
Then one can bound \( \lVert \mymat{\check{A}}{2}{1} \rVert \) as 
\begin{align*}
    \lVert \mymat{\check{A}}{2}{1} \rVert 
    &\le \lVert \wh{Q}\trans A\wh{Q} - \wt{Q}\trans ( A + \Delta A ) \wt{Q} \rVert \\ 
    &\le \lVert \wh{Q}\trans A\wh{Q} - \wh{Q}\trans ( A + \Delta A ) \wt{Q} \rVert + \lVert \wh{Q}\trans ( A + \Delta A ) \wt{Q} - \wt{Q}\trans ( A + \Delta A ) \wt{Q} \rVert \\
    &\le \lVert \wh{Q}\trans A \rVert \cdot \lVert \wh{Q} - \wt{Q} \rVert + \lVert \wh{Q}\trans \Delta A \rVert + \lVert \wh{Q} - \wt{Q} \rVert \cdot \lVert A + \Delta A \rVert \\
    &\le ( 2\theta + \epsilon ) \lVert A \rVert + \theta( \theta + 2\epsilon ) \lVert A \rVert . 
\end{align*} 
Thus, we have 
\begin{equation} \label{bound for checkA21}
    \lVert \mymat{\check{A}}{2}{1} \rVert 
    \le \big( 2\theta + \forbackerr{A}{\wh{Q}}{\theta} \big) \lVert A \rVert + \theta \big( \theta + 2\forbackerr{A}{\wh{Q}}{\theta} \big) \lVert A \rVert . 
\end{equation}
To obtain a tight bound for \( \lVert \mymat{\check{A}}{2}{1} \rVert \), we hope to find a \( \theta \) as small as possible while \( \forbackerr{A}{\wh{Q}}{\theta} \) remains small.

Suppose the SVD of \( \wh{X} \) is \( \wh{X} = \wh{U}\wh{\Sigma}\wh{V}\trans \), where \( \wh{\Sigma} = \diag( \wh{\sigma}_{1}, \wh{\sigma}_{2} ) \), \( \wh{\sigma}_{1} \ge \wh{\sigma}_{2} \ge 0 \). Now we state the main theoretical result of this paper. 
\begin{theorem} \label{theo: main result}
    Suppose \( A \in \mathcal{S}_{4} \) and 
    \begin{align} \label{ratio assumption}
        \max \Big\{ \frac{ \lvert \mymat{a}{1}{2} \rvert }{ \lvert \mymat{a}{2}{1} \rvert }, \frac{ \lvert \mymat{a}{2}{1} \rvert }{ \lvert \mymat{a}{1}{2} \rvert }, \frac{ \lvert \mymat{a}{3}{4} \rvert }{ \lvert \mymat{a}{4}{3} \rvert }, \frac{ \lvert \mymat{a}{4}{3} \rvert }{ \lvert \mymat{a}{3}{4} \rvert } \Big\} 
        \le \macheps^{-1} . 
    \end{align}
    Let \( \wh{X} \), \( \wh{Q} \) be the computed \( X \), \( Q \) in Algorithm~\ref{alg: modified direct swapping method} and \( R = \mymat{A}{1}{1}\wh{X} - \wh{X}\mymat{A}{2}{2} - \mymat{A}{1}{2} \) be the residual. Denote \( \epsilon_{r} = \lVert R \rVert / \lVert A \rVert \) and 
    \begin{align*}
        \epsilon_{s}
        = \begin{dcases}
            \frac{ \wh{\calK} \macheps }{ ( 24 + 30\sqrt{2} )^{-1} - \wh{\calK} \macheps } \wedge ( \sqrt{ 1 + 2\wh{\sigma}_{2}^{2} } - 1 ), & \text{if~} \wh{\calK} \macheps < ( 24 + 30\sqrt{2} )^{-1} , \\[0.8ex]
            \sqrt{ 1 + 2\wh{\sigma}_{2}^{2} } - 1, & \text{if~} \wh{\calK} \macheps \ge ( 24 + 30\sqrt{2} )^{-1} ,
        \end{dcases} 
    \end{align*}
    where \( \wh{\calK} = \calK( \wh{X} ) \). Then there exists \( \theta \le ( 64 + 45\sqrt{2} ) \macheps \) such that the error \( \forbackerr{A}{\wh{Q}}{\theta} \) satisfies
    \[
    \forbackerr{A}{\wh{Q}}{\theta}
    \le c( \epsilon_{r}, \epsilon_{s} ) \macheps ,
    \]
    where 
    \[
    c( \epsilon_{r}, \epsilon_{s} )
    = 78456 + 39936\sqrt{2} + 78396\epsilon_{s} + 78204\epsilon_{r}( 1 + \epsilon_{s} ) + 38304\sqrt{2}( \epsilon_{r} + \epsilon_{s} + \epsilon_{r}\epsilon_{s} ) . 
    \]
\end{theorem}

We mention that both \( \epsilon_{r} \) and \( \epsilon_{s} \) are bounded by \( 1 \) provided 
\[
\xi( A ) 
= \frac{ \lVert A \rVert_{\mathrm{F}} }{ \textup{sep}( \mymat{A}{1}{1}, \mymat{A}{2}{2} ) } 
\]
is safely bounded by \( \macheps^{-1} \), where \( \textup{sep}( \mymat{A}{1}{1}, \mymat{A}{2}{2} ) = \sigma_{4}( T ) \) is the separation between \( \mymat{A}{1}{1} \) and \( \mymat{A}{2}{2} \). Indeed, we have the following corollary. 
\begin{corollary} \label{cor: stability}
    Suppose \( A \in \mathcal{S}_{4} \) satisfies the assumption~(\ref{ratio assumption}). Let \( \epsilon_{0} = 64\gamma_{12} \). If
    \[ 
    \xi( A ) \le ( 48 + 60\sqrt{2} )^{-1} \epsilon_{0}^{-1} ,
    \] 
    then \( \epsilon_{r}, \epsilon_{s} < 1 \), and
    \[
    \lVert \mymat{\check{A}}{2}{1} \rVert 
    \le \big( c( 1,1 ) + 128 + 90\sqrt{2} \big) \macheps \lVert A \rVert + O( \macheps^{2} \lVert A \rVert ) . 
    \]
\end{corollary}

\begin{proof}
    Consider the perturbed linear system \( ( T + \Delta T ) \wh{x} = b \) for the computed solution \( \wh{x} \), with \( T \) and \( b \) as in Algorithm~\ref{alg: direct swapping method}. Assume \( b \ne 0 \); otherwise \( \epsilon_{r} = \epsilon_{s} = 0 \). By Lemma~\ref{lemma: stability of GECP}, it is nor hard to show that \( \lVert \Delta T \rVert \le \epsilon_{0} \lVert T \rVert \). Notice that \( \kappa( T ) \le \sqrt{2} \, \xi( A ) \) by Lemma~\ref{lemma: norm inequalities}. If \( \sqrt{2} \, \epsilon_{0} \cdot \xi( A ) < 1 \), then by~\cite[Lemma~2.6.1]{Golub13} it can be shown that 
    \[
    \lVert \wh{x} \rVert 
    \le \frac{ 1 + \epsilon_{0} \cdot \kappa( T ) }{ 1 - \epsilon_{0} \cdot \kappa( T ) } \lVert x \rVert 
    \le \frac{ 1 + \sqrt{2} \, \epsilon_{0} \cdot \xi( A ) }{ 1 - \sqrt{2} \, \epsilon_{0} \cdot \xi( A ) } \lVert x \rVert . 
    \]
    Suppose \( x \) is the exact solution to \( Tx = b \). Using \( \wh{\calK} \le \lVert \wh{x} \rVert \) and \( x = T^{-1} b \), we have  
    \[
    \wh{\calK} \macheps 
    \le \frac{ 1 + \sqrt{2} \, \epsilon_{0} \cdot \xi( A ) }{ 1 - \sqrt{2} \, \epsilon_{0} \cdot \xi( A ) } \macheps \cdot \frac{ \lVert \mymat{A}{1}{2} \rVert_{\mathrm{F}} }{ \textup{sep}( \mymat{A}{1}{1}, \mymat{A}{2}{2} ) } 
    \le \frac{ 1 + \sqrt{2} \, \epsilon_{0} \cdot \xi( A ) }{ 1 - \sqrt{2} \, \epsilon_{0} \cdot \xi( A ) } \macheps \cdot \xi( A ) .  
    \] 
    For the residual \( r = -\Delta T \cdot \wh{x} \), by Lemma~\ref{lemma: norm inequalities} we have 
    \[
    \frac{ \lVert r \rVert }{ \lVert A \rVert } 
    \le \epsilon_{0} \lVert \wh{x} \rVert \cdot \frac{ \lVert T \rVert }{ \lVert A \rVert  } 
    \le \frac{ 2 \big( 1 + \sqrt{2} \,  \epsilon_{0} \cdot \xi( A ) \big) }{ 1 - \sqrt{2} \, \epsilon_{0} \cdot \xi( A ) } \epsilon_{0} \cdot \xi( A ) .
    \]
    Thus, a loose choice of \( \xi( A ) < ( 48 + 60\sqrt{2} )^{-1} \epsilon_{0}^{-1}  \) guarantees that \( \epsilon_{r}, \epsilon_{s} < 1 \). Our claim follows immediately from Theorem~\ref{theo: main result} and~(\ref{bound for checkA21}).
\end{proof}

Corollary~\ref{cor: stability} indicates that Algorithm~\ref{alg: modified direct swapping method} is stable in terms of the norm of \( \mymat{\check{A}}{2}{1} \) if \( A \) satisfies the assumption~(\ref{ratio assumption}) and \( \xi( A ) \) is bounded away from \( \macheps^{-1} \).

\begin{remark}
    The assumption~(\ref{ratio assumption}) is a numerical genericity condition: if it fails, some block has an off-diagonal entry below \( \macheps \lVert A \rVert \). Thus, dropping the smaller off-diagonal entry of the \( 2 \times 2 \) block is an \( O(\macheps \lVert A \rVert) \) perturbation, and the resulting block is defective with two real eigenvalues.
\end{remark}

\begin{remark}
    One may be concerned that the constant factors in the bounds for \( \theta \) and \(\forbackerr{A}{\wh{Q}}{\theta} \) are too large to match the threshold used in the practical algorithm. However, since our primary goal is to establish bounds of the form \( O(\macheps) \) for \( \theta \) and \( \forbackerr{A}{\wh{Q}}{\theta} \), the constants appearing in the argument are chosen mainly to facilitate the proof and may not be sharp. In practice, as seen in our experiments, the actual norm of \( \mymat{\check{A}}{2}{1} \) is always much smaller than the bound we provide. 
\end{remark}

\section{Numerical experiments} \label{sec: numerical experiments}
In this section, we present numerical experiments to show the stability of our algorithm. All experiments are conducted in double precision, that is, the unit roundoff \( \macheps \approx 1.1\times 10^{-16} \). 

The algorithms used for comparison are listed as follows. 
\begin{itemize}
    \item \texttt{DLAEXC}: the LAPACK subroutine developed to implement the direct swapping algorithm in~\cite{Bai1993} for the eigenvalue swapping in double precision. 

    \item \textbf{MDS}: our modified direct swapping algorithm replacing the Householder QR in \texttt{DLAEXC} with the pivoted Givens QR.   

    \item \textbf{RDS}: a refined direct swapping algorithm in~\cite{Camps19}. More specifically, if the norm of \( \mymat{\check{A}}{2}{1} \) exceeds the threshold after the first attempt, refinement steps will be applied until it is small enough.  
\end{itemize} 
To ensure a fair and accurate comparison, the reference algorithms are tested in their native programming languages, i.e., \texttt{DLAEXC} in Fortran~77 and RDS in MATLAB. Accordingly, we provide implementations of our algorithm in both languages. Our test algorithms and matrices are available at \url{https://github.com/xurentao/mds-code}. 

Key features of the three algorithms are summarized in Table~\ref{tab: comparison of three algorithms}. Notice that for a fair comparison, the stability tests of the three algorithms are set to be the same as that in \texttt{DLAEXC}, i.e.,
\[
\tol \le 10 \macheps \lVert A \rVert_{\max} .
\]  

Let \( \check{A} \) be the matrix \( A \) after swapping. For convenience, we denote the complex eigenvalue of \( \mymat{A}{1}{1} \) with nonnegative imaginary part by \( \lambda_{1} \), and that of \( \mymat{\check{A}}{2}{2} \) by \( \check{\lambda}_{1} \). Similarly, \( \lambda_{2} \) and \( \check{\lambda}_{2} \) denote the corresponding eigenvalues with nonnegative imaginary part of \( \mymat{A}{2}{2} \) and \( \mymat{\check{A}}{1}{1} \). 

\begin{table}[tbp]
\centering
\caption{Comparison of three algorithms.}
\renewcommand{\arraystretch}{1.2}
\begin{tabular}{lcc}
\hline
Algorithm & Solve \( Tx = b \) &  \( QR = [ -X\trans, I_{2} ]\trans \) \\ \hline 
\texttt{DLAEXC} & GECP & Householder QR \\ \hline
MDS & GECP & pivoted Givens QR \\ \hline
RDS & GEPP & SVD of \( X \) and iterative refinement \\ \hline
\end{tabular}
\label{tab: comparison of three algorithms}
\end{table}

\subsection{Test matrices} \label{subsec: test matrices}
Three specific test matrices are listed below, followed by a family of generated matrices. For reasons of space, all test matrices are shown to four decimal places with rounding to zero. 

\textbf{Test matrix 1:} 
Given 
\[
A = 
\begin{bmatrix}
    5.2705 \times 10^{-4} & 2.4180 \times 10^{-12} & 1.9644 & 1.9639 \\
    -7.7816 & 5.2705 \times 10^{-4} & 1.0212 &  0.9426 \\
     &  & 1.7847 \times 10^{-8} & 3.6409 \times 10^{-8} \\
     &  & -3.6409 \times 10^{-8} & 1.7847 \times 10^{-8} 
\end{bmatrix} .
\]
This matrix is specifically crafted, featuring \( \xi( A ) \approx 2.4 \times 10^{8} \) and a large ratio \( \lvert \mymat{a}{2}{1}/\mymat{a}{1}{2} \rvert \approx 3.2 \times 10^{12} \). 

\textbf{Test matrix 2:}
Given 
\[
A = 
\begin{bmatrix}
    5.2705 \times 10^{-4} & -7.7816 & 0.9426 & 1.0212 \\
    2.4180 \times 10^{-12} & 5.2705 \times 10^{-4} & 1.9639 & 1.9644 \\  
     &  &  1.7847 \times 10^{-8} & -3.6409 \times 10^{-8} \\
     &  &  3.6409 \times 10^{-8} & 1.7847 \times 10^{-8}
\end{bmatrix} .
\]
Clearly, this test matrix is obtained by a similarity transformation \( P \) on test matrix~\( 1 \), where 
\[
P = I_{2} \otimes P_{2} , \dspace P_{2} = \begin{bmatrix}
    0 & 1 \\
    1 & 0
\end{bmatrix} .
\]  

\textbf{Test matrix 3:} Given 
\[
A = 
\begin{bmatrix}
    0.4367 & 8.7951 \times 10^{-13} & 0.2059 & -2.1198 \\
    -0.5795 & 0.4367 & 0.7728 & -0.5159 \\
     &  & 0.6205 & -0.3141 \\
     &  & 0.3141 & 0.6205
\end{bmatrix} .
\]
This test matrix is designed with~\( \xi( A )~\approx~15.2 \), and with a clear separation between \( \lambda_{1} \) and \( \lambda_{2} \). However, the imaginary part of \( \lambda_{1} \) is on the order of \( 10^{-7} \), significantly less than \( \lVert A \rVert \). 

\textbf{The generated family.} The following matrix is used in~\cite{Camps19} to demonstrate the stability of RDS: 
\begin{equation} \label{generated A}
    A = \begin{bmatrix}
        a    & bk & \mymat{a}{1}{3} & \mymat{a}{1}{4} \\
        -b/k & a  & \mymat{a}{2}{3} & \mymat{a}{2}{4} \\
        0    & 0  & a + r_{1}g      & (b + r_{2}g)k   \\
        0    & 0  & -(b + r_{2}g)/k & a + r_{1}g
\end{bmatrix} ,
\end{equation}
where parameters \( k \) and \( g \) control the gap between eigenvalues, and the remaining parameters are drawn from the standard normal distribution. 

\subsection{Numerical results} \label{subsec: numerical results}
All numerical results are reported with \( 2 \) significant digits for simplicity, as shown in Tables~\ref{tab: results in Fortran 77} and~\ref{tab: results in MATLAB}. The variations in the results of MDS between Tables~\ref{tab: results in Fortran 77} and~\ref{tab: results in MATLAB} could be attributed to differences in our implementation details between the two programming environments.

\begin{table}[tbp]
\centering
\caption{Comparison between \texttt{DLAEXC} and MDS in Fortran 77.} 
\renewcommand{\arraystretch}{1.2}
\begin{tabular}{c|c|ccc}
\hline
Algorithm & Test matrix & \( \tol / \lVert A \rVert_{\max} \) & \( \lvert \lambda_{1} - \check{\lambda}_{1} \rvert / \lvert \lambda_{1} \rvert \) & \( \lvert \lambda_{2} - \check{\lambda}_{2} \rvert / \lvert \lambda_{2} \rvert \) \\
\hline
\multirow{3}{*}{\texttt{DLAEXC}} 
    & 1 & \( 2.2 \times 10^{-13} \) & \( 1.4 \times 10^{-7} \) & \( 5.1 \times 10^{-5} \) \\
    & 2 & \( 9.8 \times 10^{-18} \) & \( 9.4 \times 10^{-8} \) & \( 4.2 \times 10^{-9} \) \\
    & 3 & \( 1.6 \times 10^{-17} \) & \( 3.6 \times 10^{-11} \) & \( 3.2 \times 10^{-16} \) \\
\hline
\multirow{3}{*}{MDS} 
    & 1 & \( 5.7 \times 10^{-17} \) & \( 1.1 \times 10^{-8} \) & \( 4.5 \times 10^{-9} \) \\
    & 2 & \( 7.1 \times 10^{-18} \) & \( 4.1 \times 10^{-8} \) & \( 6.2 \times 10^{-10} \) \\
    & 3 & \( 1.3 \times 10^{-17} \) & \( 1.4 \times 10^{-11} \) & \( 1.8 \times 10^{-16} \) \\
\hline
\end{tabular}
\label{tab: results in Fortran 77}
\end{table}

As shown in Table~\ref{tab: results in Fortran 77}, \texttt{DLAEXC} fails on test matrix \( 1 \) as \( \tol \) exceeds the threshold \( 10 \macheps \lVert A \rVert_{\max} \) by a factor of approximately \( 10^{2} \). In contrast, \( \tol \) from MDS is safely below the threshold. Indeed, a MATLAB implementation (used here for its flexibility in inspecting intermediate quantities) of MDS shows that for test matrices \( 1 \) and \( 2 \), the computed \( \wh{X} \) and residual \( R \) satisfy \( \lVert R \rVert \approx 4.3 \times 10^{-12} \) and 
\begin{equation} \label{phi(A) for test matrix 1 and 2} 
    \frac{ \lvert \wh{u}_{2}\trans R\wh{v}_{2} \rvert }{ \tsqrt{ 1 + 2\wh{\sigma}_{2}^{2} } } 
    \approx 1.5 \times 10^{-16} .
\end{equation}  
By the analysis in Section~\ref{sec: strategy for the error analysis}, (\ref{phi(A) for test matrix 1 and 2}) implies that the contribution of solving the linear system to the bound for \( \lVert \mymat{\check{A}}{2}{1} \rVert_{\max} \) is at most of order \( 10^{-16} \). As a consequence, the contrasting performance of the two algorithms comes from the QR factorization. Specifically, the pivoted Givens QR exhibits superior stability over the Householder QR on test matrix~\( 1 \). On the other hand, both algorithms swap successfully on test matrix~\( 2 \). Note that the corresponding \( [ -\wh{X}\trans, I_{2} ]\trans \) of test matrices~\( 1 \) and~\( 2 \) only differ by a row permutation. These results indicate the QR factorization approach has a subtle but crucial effect on the error of eigenvalue swapping algorithms. Moreover, MDS produces smaller relative errors of \( \lambda_{1} \) and \( \lambda_{2} \) than \texttt{DLAEXC} on both test matrices. 

For test matrix \( 3 \), the performances of MDS and \texttt{DLAEXC} are comparable in terms of both \( \tol \) and the relative error of eigenvalues. In particular, for both algorithms the relative errors of \( \lambda_{1} \) are significantly larger than \( \macheps \) on all test matrices. This loss of precision exposes that the relative error of eigenvalues may still be large even when \( \tol \) is on the order of \( \macheps \lVert A \rVert_{\max} \).    

\begin{table}[tbp]
\centering
\caption{Comparison between RDS (after the first attempt) and MDS in MATLAB.} 
\renewcommand{\arraystretch}{1.2}
\begin{tabular}{c|c|ccc}
\hline
Algorithm & Test matrix & \( \tol / \lVert A \rVert_{\max} \) & \( \lvert \lambda_{1} - \check{\lambda}_{1} \rvert / \lvert \lambda_{1} \rvert \) & \( \lvert \lambda_{2} - \check{\lambda}_{2} \rvert / \lvert \lambda_{2} \rvert \) \\
\hline
\multirow{3}{*}{RDS} 
    & 1 & \( 2.0 \times 10^{-13} \) & \( 4.8 \times10^{-10} \) & \( 2.9 \times 10^{-1} \) \\
    & 2 & \( 2.7 \times 10^{-13} \) & \( 1.9 \times 10^{-7} \) & \( 3.3 \times 10^{-2} \) \\
    & 3 & \( 2.6 \times 10^{-17} \) & \( 8.8 \times 10^{-11} \) & \( 5.1 \times 10^{-16} \) \\
\hline
\multirow{3}{*}{MDS} 
    & 1 & \( 7.0 \times 10^{-18} \) & \( 3.7 \times10^{-8} \) & \( 6.2 \times 10^{-10} \) \\
    & 2 & \( 5.6 \times 10^{-17} \) & \( 2.0 \times 10^{-8} \) & \( 4.3 \times 10^{-9} \) \\
    & 3 & \( 5.9 \times 10^{-18} \) & \( 1.4 \times 10^{-11} \) & \( 8.0 \times 10^{-17} \) \\
\hline
\end{tabular}
\label{tab: results in MATLAB}
\end{table}

Table~\ref{tab: results in MATLAB} shows that on test matrices \( 1 \) and \( 2 \), RDS fails on the first attempt, thus requiring additional refinement steps. A numerical examination reveals that on test matrix~\( 1 \), the residual \( R_{\mathrm{GEPP}} \) associated with solving the linear system by Gaussian elimination with partial pivoting (GEPP) satisfies \( \lVert R_{\mathrm{GEPP}} \rVert \approx 5.9 \times 10^{-12} \) and 
\[
\frac{ \lvert \wh{u}_{2}\trans R_{\mathrm{GEPP}}\wh{v}_{2} \rvert }{ \tsqrt{ 1 + 2\wh{\sigma}_{2}^{2} } } \approx 2.0 \times 10^{-12} .
\]
And on test matrix~\( 2 \), these quantities are on the same order. These results justify the use of GECP in the direct swapping method. Additionally, our test shows that one refinement step is sufficient to reduce \( \tol \) below the threshold. Nevertheless, the final relative error \( \lvert \lambda_{2} - \check{\lambda}_{2} \rvert / \lvert \lambda_{2} \rvert \) implies \( \check{\lambda}_{2} \) suffers a severe loss of precision. Hence, regarding the stability of eigenvalues, the refinement step appears to be less effective on test matrices \( 1 \) and \( 2 \). 

Figure~\ref{fig: tol} shows the error \( \tol / \lVert A \rVert_{\max} \) for \( A \) generated by~(\ref{generated A}) using RDS and MDS. The dark (larger-error) regions are slightly fewer for MDS than for RDS, but both lie below the threshold and are overall comparable. 
\begin{figure}[tbp]
  \centering
  \begin{minipage}[t]{0.45\textwidth}
    \centering
    \includegraphics[width=\linewidth]{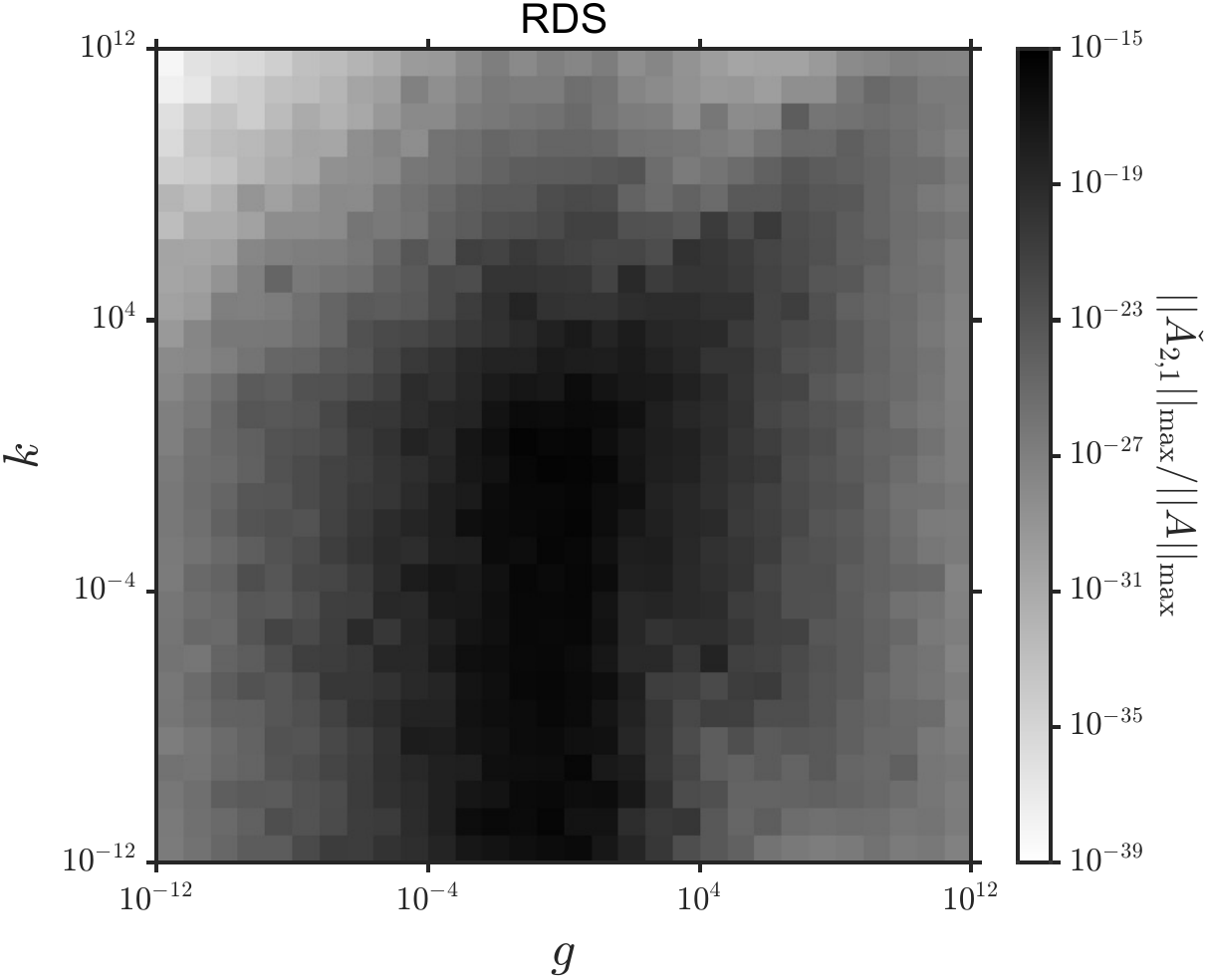}
  \end{minipage}
  \hfill
  \begin{minipage}[t]{0.45\textwidth}
    \centering
    \includegraphics[width=\linewidth]{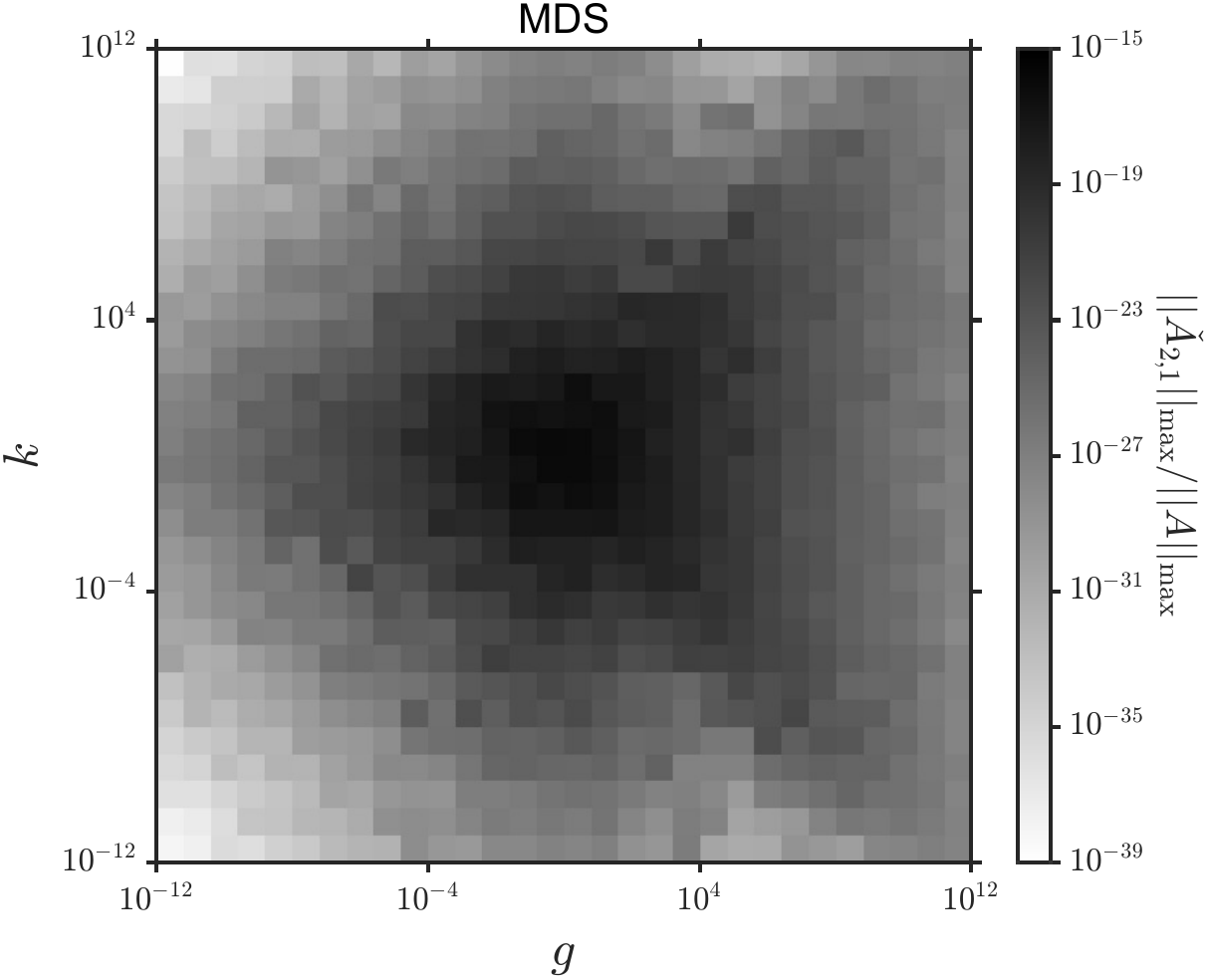}
  \end{minipage}
  \caption{Error \( \tol / \lVert A \rVert_{\max} \) in swapping \(A\) by RDS and MDS. For each parameter pair \( (k, g) \), \( 20 \) matrices are generated following (\ref{generated A}). The maximum error among them is plotted at \( (k, g) \) in the figure. For both algorithms, no error exceeds \( 10^{-15} \).}
  \label{fig: tol}
\end{figure}
Figure~\ref{fig:eig} shows the relative errors in the computed eigenvalues for the two algorithms. For \( \lambda_{1} \), MDS outperforms RDS in the region  \( g < 1 \). For \( \lambda_{2} \), this advantage appears in the region \( k < 1 \) and \( g < 1 \). In the remaining regions, the two algorithms behave similarly. 
\begin{figure}[tbp]
  \centering
  \begin{subfigure}[t]{0.45\textwidth}
    \centering
    \includegraphics[width=\linewidth]{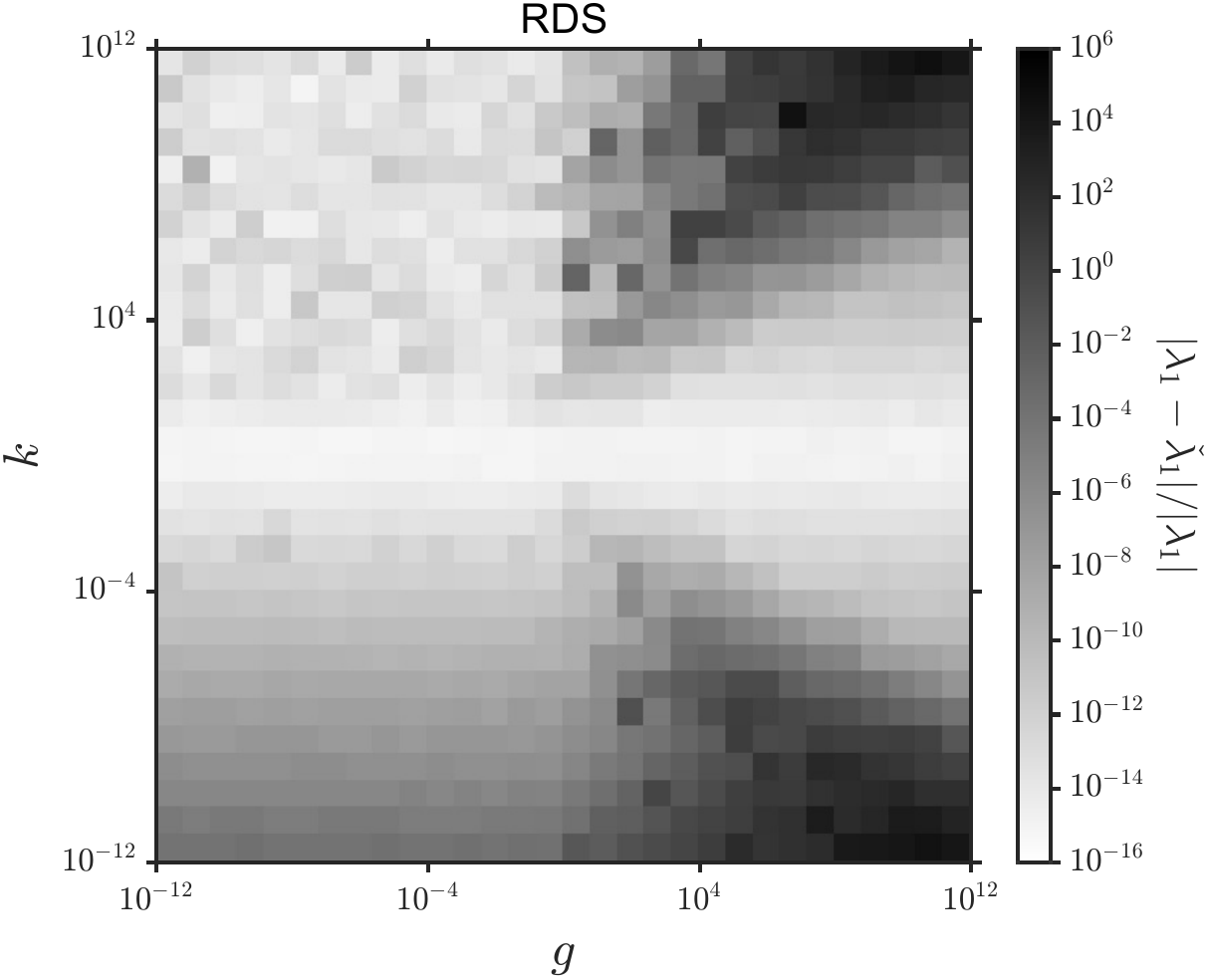}
  \end{subfigure}\hfill
  \begin{subfigure}[t]{0.45\textwidth}
    \centering
    \includegraphics[width=\linewidth]{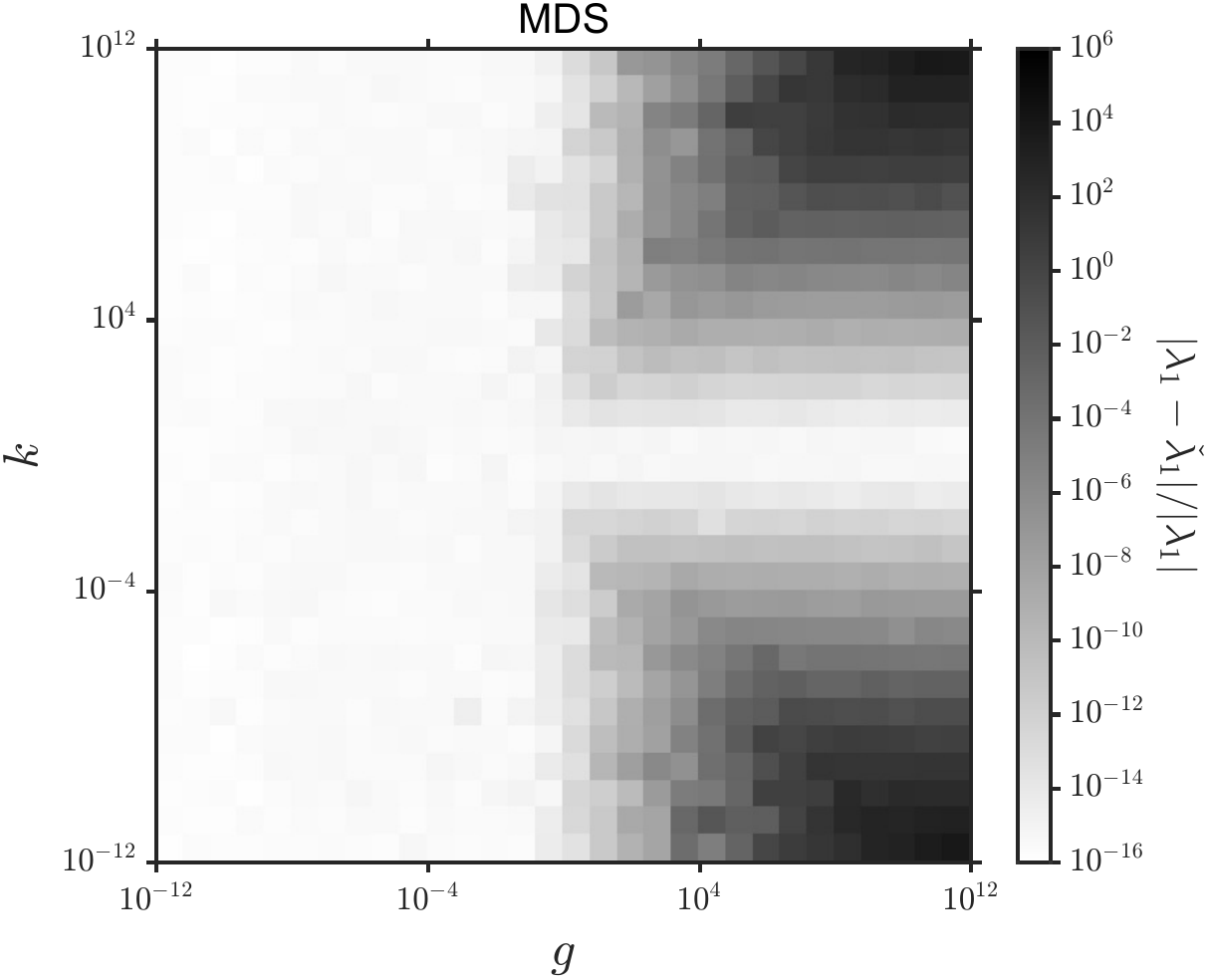}
  \end{subfigure}

  \vspace{1ex}
  \begin{subfigure}[t]{0.45\textwidth}
    \centering
    \includegraphics[width=\linewidth]{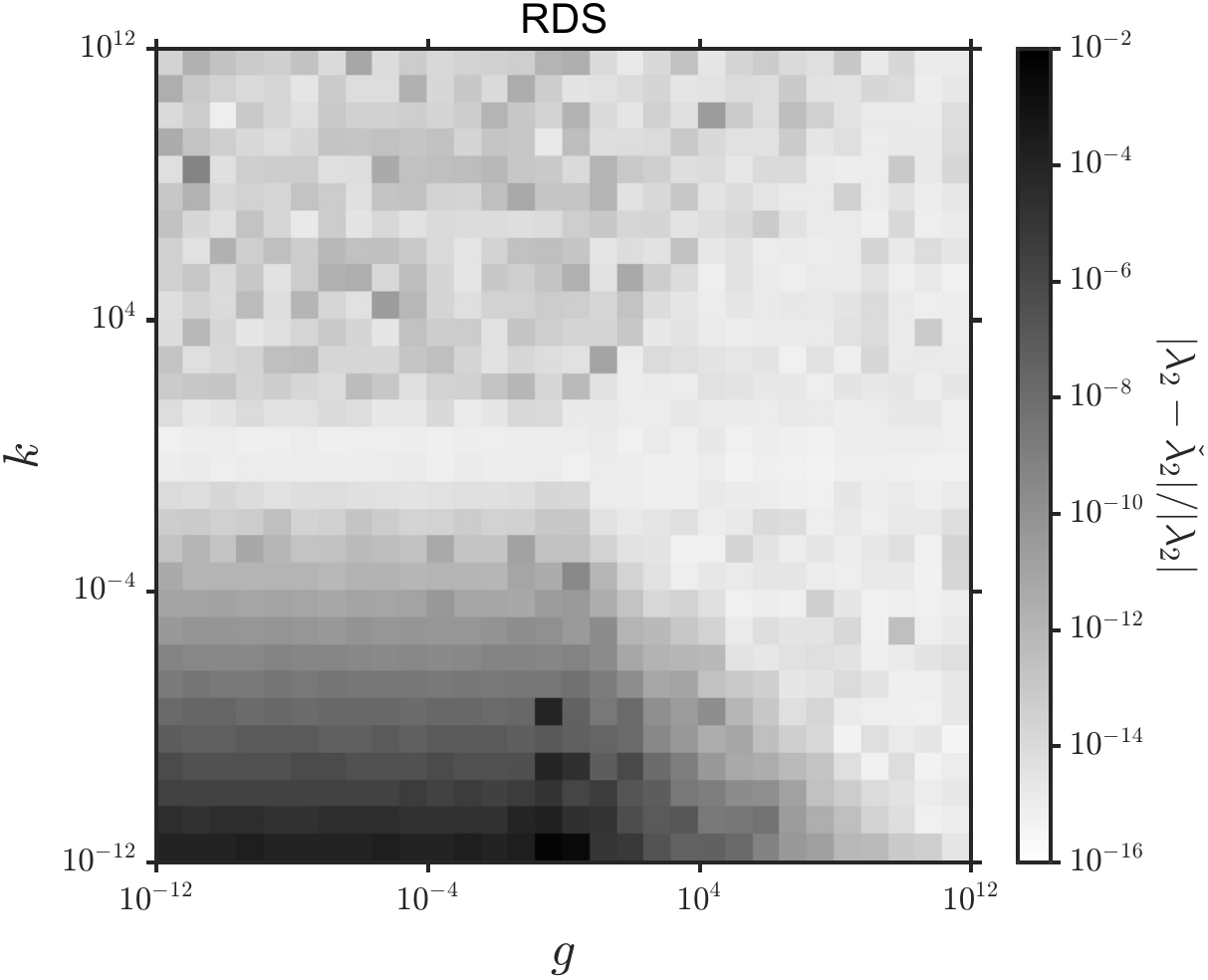}
  \end{subfigure}\hfill
  \begin{subfigure}[t]{0.45\textwidth}
    \centering
    \includegraphics[width=\linewidth]{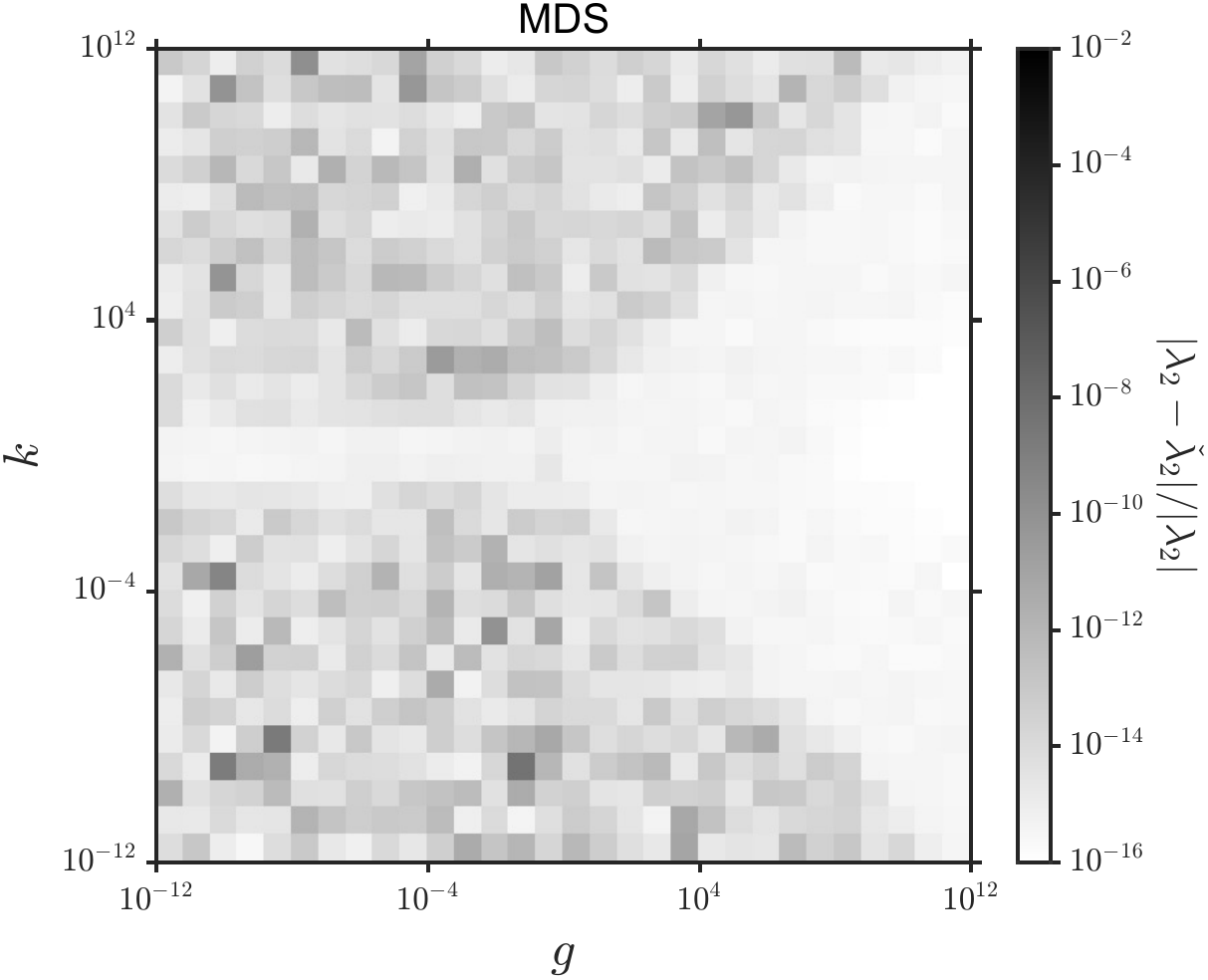}
  \end{subfigure}
  \caption{Relative error of eigenvalues in swapping \( A \) by RDS and MDS. The experimental settings are the same as in Figure~\ref{fig: tol}.}
  \label{fig:eig}
\end{figure}

\section{Strategy for the error analysis} \label{sec: strategy for the error analysis}
In this section we outline the strategy for the error analysis of Algorithm~\ref{alg: modified direct swapping method}. 

Let \( \wh{x} \) be the computed solution to the linear system \( Tx = b \) and \( \wh{X} = \sfvec^{-1}( \wh{x} ) \). Then the QR factorization in the second step of Algorithm~\ref{alg: modified direct swapping method} acts on the perturbed matrix 
\[
\wh{H} = \begin{bmatrix}
    -\wh{X} \\
    I_{2}
\end{bmatrix} 
\]
rather than on the exact \( H \). Additionally, because of rounding errors in the QR factorization of \( \wh{H} \), the resulting factor \( \wh{Q} \) is approximately orthogonal. 

Without loss of generality, we assume the largest entry of \( \lvert \wh{X} \rvert \) is in its first column. It will be shown in Section~\ref{sec: contribution to error from QR factorization} that by the pivoted Givens QR, there exist \( \Delta \wh{X} \in \mathbb{R}^{2 \times 2} \) and orthogonal \( \wt{Q} \in \mathbb{R}^{4 \times 4} \) satisfying 
\begin{equation} \label{bounds for Delta X and Q}
    \lvert \Delta \wh{X} \rvert 
    \le \gamma_{6} \bigg( 5 
    \begin{bmatrix}
       0 & \lvert \mymat{\wh{x}}{1}{2} \rvert \\
       0 & \lvert \mymat{\wh{x}}{2}{2} \rvert 
    \end{bmatrix} + 8 \wh{\sigma}_{2} 
    \begin{bmatrix}
        0 & 1 \\
        0 & 1
    \end{bmatrix} \bigg), \dspace 
    \lVert \wt{Q} - \wh{Q} \rVert 
    \le ( 64 + 45 \sqrt{2} ) \macheps ,  
\end{equation}
such that 
\[ 
\begin{bmatrix}
    -\wh{X} + \Delta \wh{X} \\
    I_{2}
\end{bmatrix} 
= \wt{Q} 
\begin{bmatrix}
    \wt{Z} \\
    0
\end{bmatrix} .
\] 
Let \( \theta = \lVert \wt{Q} - \wh{Q} \rVert \) and \( \wt{X} = \wh{X} - \Delta \wh{X} \). If there is a matrix 
\[
\Delta A 
= \begin{bmatrix}
    \Delta \mymat{A}{1}{1} & \Delta \mymat{A}{1}{2} \\
        & \Delta \mymat{A}{2}{2} 
\end{bmatrix} \in \mathcal{T}_{4} , 
\]
such that  
\begin{equation} \label{eq: sufficient condition for swapping}
    \begin{bmatrix}
        \mymat{A}{1}{1} + \Delta \mymat{A}{1}{1} & \mymat{A}{1}{2} + \Delta \mymat{A}{1}{2} \\
        & \mymat{A}{2}{2} + \Delta \mymat{A}{2}{2} 
    \end{bmatrix}  
    \begin{bmatrix}
        -\wt{X} \\
        I_{2}
    \end{bmatrix} 
    = \begin{bmatrix}
        -\wt{X} \\
        I_{2}
    \end{bmatrix} ( \mymat{A}{2}{2} + \Delta \mymat{A}{2}{2} ) ,
\end{equation}
then \( \wt{Q} \) swaps \( A + \Delta A \) by Theorem~\ref{theo: iff condition for reordering} and \( \forbackerr{A}{\wh{Q}}{\theta} \le \lVert \Delta A \rVert / \lVert A \rVert \) by Definition~\ref{def: error in eigenvalue reordering}.

Clearly, equation (\ref{eq: sufficient condition for swapping}) is equivalent to 
\begin{equation} \label{eq: short equation for sawpping}
( \mymat{A}{1}{1} + \Delta \mymat{A}{1}{1} ) \wt{X} - \wt{X} ( \mymat{A}{2}{2} + \Delta \mymat{A}{2}{2} ) = \mymat{A}{1}{2} + \Delta \mymat{A}{1}{2} .
\end{equation}
For a solution \( \Delta A \) to~(\ref{eq: short equation for sawpping}), set 
\begin{equation} \label{definition for Delta a}
    \Delta a 
    = \begin{bmatrix}
        \sfvec ( \Delta \mymat{A}{1}{1} ) \\
        \sfvec ( \Delta \mymat{A}{2}{2} ) \\
        \sfvec ( \Delta \mymat{A}{1}{2} )
    \end{bmatrix} .
\end{equation}
Since \( \lVert \Delta A \rVert \le \lVert \Delta a \rVert \),
it suffices to study the minimum of \( \lVert \Delta a \rVert \). The following lemma characterizes the minimum of \( \lVert \Delta a \rVert \). Its proof follows directly from the argument in~\cite[pp.~4--5]{hig93}. 
\begin{lemma}[Higham~\cite{hig93}] \label{lemma: minimum of Delta a}
    Let \( \wt{X} = \wt{U}\wt{\Sigma}\wt{V}\trans \in \mathbb{R}^{2 \times 2} \) be the \textup{SVD} of \( \wt{X} \), in which \( \wt{\Sigma} = \diag( \wt{\sigma}_{1}, \wt{\sigma}_{2} ) \), \( \wt{\sigma}_{1} \ge \wt{\sigma}_{2} \ge 0 \). Set  
    \[
    M = [-\wt{X}\trans \otimes I_{2}, I_{2} \otimes \wt{X}, I_{4}] \in \mathbb{R}^{4 \times 12} , \dspace 
    F = \mymat{A}{1}{1} \wt{X} - \wt{X} \mymat{A}{2}{2} - \mymat{A}{1}{2} .  
    \] 
    Let \( M^{\dagger} \) denote the \textup{Moore--Penrose} pseudoinverse of \( M \) and \( f = \sfvec( F ) \). Then for all \( \Delta \mymat{A}{1}{1} \), \( \Delta \mymat{A}{1}{2} \) and \( \Delta \mymat{A}{2}{2} \) satisfying (\ref{eq: short equation for sawpping}), the minimum of \( \lVert \Delta a \rVert \) is attained at \( \Delta a = M^{\dagger} f \) and 
    \[
    \lVert M^{\dagger} f \rVert = \lVert D^{-1} ( \wt{V}\trans \otimes \wt{U}\trans ) f \rVert ,
    \]
    where 
    \[
    D = \diag\big( \tsqrt{ 1 + 2\wt{\sigma}_{1}^{2} }, \tsqrt{ 1 + \wt{\sigma}_{1}^{2} + \wt{\sigma}_{2}^{2} }, \tsqrt{ 1 + \wt{\sigma}_{1}^{2} + \wt{\sigma}_{2}^{2} }, \tsqrt{ 1 + 2\wt{\sigma}_{2}^{2} } \big) .
    \]
\end{lemma} 

From Lemma~\ref{lemma: minimum of Delta a}, the key is to bound \( \lVert D^{-1} ( \wt{V}\trans \otimes \wt{U}\trans ) f \rVert \). Let 
\[
R = \mymat{A}{1}{1} \wh{X} - \wh{X} \mymat{A}{2}{2} - \mymat{A}{1}{2}, \dspace W = -\mymat{A}{1}{1} \Delta \wh{X} + \Delta \wh{X} \mymat{A}{2}{2} .
\] 
Then we have \( F = R + W \)
and 
\[ 
r = \sfvec( R ) = T\wh{x} - b = -\Delta T \cdot \wh{x} , \dspace 
w = \sfvec( W ) = -T \Delta \wh{x} , 
\] 
where \( \Delta \wh{x} = \sfvec( \Delta \wh{X} ) \). Using Lemma~\ref{lemma: norm inequalities} and~(\ref{bounds for Delta X and Q}), it can be shown that 
\begin{align*}
    \lVert r \rVert 
    &\le \lVert \Delta T \rVert \cdot \lVert \wh{x} \rVert
    \le 64 \gamma_{12} \lVert T \rVert \cdot \lVert \wh{x} \rVert
    \le 128\sqrt{2} \, \wh{\sigma}_{1} \cdot \gamma_{12} \lVert A \rVert , \\
    \lVert w \rVert 
    &\le \lVert T \rVert \cdot \lVert \Delta \wh{x} \rVert 
    \le \lVert T \rVert \cdot \gamma_{6} ( 5\wh{\sigma}_{1} + 8\sqrt{2} \, \wh{\sigma}_{2} ) 
    \le ( 10 + 16\sqrt{2} ) \wh{\sigma}_{1} \cdot \gamma_{6} \lVert A \rVert . 
\end{align*} 
Since \( \wt{V}\trans \otimes \wt{U}\trans \) is orthogonal and \( \lvert \wt{\sigma}_{1} - \wh{\sigma}_{1} \rvert \le \lVert \Delta \wh{X} \rVert \ll \wh{\sigma}_{1} \), the first three entries of \( D^{-1} ( \wt{V}\trans \otimes \wt{U}\trans ) f \) are bounded by 
\begin{equation} \label{bound for first three entries of Delta A} 
    \lVert \big( D^{-1} ( \wt{V}\trans \otimes \wt{U}\trans ) f \big)_{1 \tightcolon 3} \rVert  
    \le \frac{ \lVert r \rVert + \lVert w \rVert }{ \tsqrt{ 1 + \wt{\sigma}_{1}^{2} } } 
    \le ( 128 \sqrt{2} \, \gamma_{12} + ( 10 + 16 \sqrt{2} ) \gamma_{6} ) \lVert A \rVert .
\end{equation} 
Thus, we only need to focus on the fourth entry of \( D^{-1} ( \wt{V}\trans \otimes \wt{U}\trans ) f \), i.e., 
\begin{equation} \label{fourth entry of Delta a}
\big( D^{-1} ( \wt{V}\trans \otimes \wt{U}\trans ) f \big)_{4}
= \frac{ \lvert \wt{u}_{2}\trans F\wt{v}_{2} \rvert }{ \tsqrt{ 1 + 2\wt{\sigma}_{2}^{2} } } , 
\end{equation}
where \( \wt{u}_{2} \) and \( \wt{v}_{2} \) are the second columns of \( \wt{U} \) and \( \wt{V} \), respectively. 

When \( \wh{\kappa} := \kappa( \wh{X} ) \le 4 \), notice that 
\[
\wt{\sigma}_{2} \ge \wh{\sigma}_{2} - \lVert \Delta\wh{X} \rVert 
\ge \frac{ \wh{\sigma}_{1} }{ 4 } - \lVert \Delta\wh{X} \rVert 
\approx \frac{ \wh{\sigma}_{1} }{ 4 } .
\]
Then~(\ref{fourth entry of Delta a}) is bounded by 
\begin{equation} \label{bound for the fourth entry when kappa <= 4} 
    \frac{ \lVert r \rVert + \lVert w \rVert }{ \tsqrt{ 1 + 2\wt{\sigma}_{2}^{2} } } 
    \le \frac{ 2\sqrt{2} ( \lVert r \rVert + \lVert w \rVert ) }{ \wh{\sigma}_{1} } 
    \le \big( 512 \gamma_{12} +  ( 64 + 20 \sqrt{2} ) \gamma_{6} \big) \lVert A \rVert 
\end{equation}
and this case is settled. Hence, it suffices to consider the situation \( \wh{\kappa} > 4 \).

Denote 
\[
\wh{\phi}( A ) 
= \frac{ \lvert \wh{u}_{2}\trans R \wh{v}_{2} \rvert }{ \sqrt{ 1 + 2\wh{\sigma}_{2}^{2} } }, \dspace 
\wh{\psi}( A ) 
= \frac{ \lvert \wh{u}_{2}\trans W \wh{v}_{2} \rvert }{ \sqrt{ 1 + 2\wh{\sigma}_{2}^{2} } } ,  
\]
where \( \wh{u}_{2} \) and \( \wh{v}_{2} \) are the second columns of \( \wh{U} \) and \( \wh{V} \), respectively. We will show that when \( \wh{\kappa} > 4 \), 
\begin{align} \label{bound for contribution of residual}
    &\wh{\phi}( A ) 
    \le ( 6272 + 3072\sqrt{2} ) ( 1 + \epsilon_{r} ) \cdot \gamma_{12} \lVert A \rVert, \\
    \label{bound for contribution of QR}
    &\wh{\psi}( A ) 
    \le \big( ( 490 + 240\sqrt{2} ) ( 1 + \epsilon_{r} ) + 32 \big) \cdot \gamma_{6} \lVert A \rVert  
\end{align} 
in Sections~\ref{sec: contribution to error of solving linear system} and~\ref{sec: contribution to error from QR factorization}, respectively. Both a normwise and a componentwise analysis will be used to bound \( \wh{\phi}( A ) \). In particular, (\ref{bound for contribution of residual}) arises from the normwise analysis, in which the unit vectors \( \wh{u}_{2} \) and \( \wh{v}_{2} \) are simply dropped. And it will be shown that the analysis of \( \wh{\phi}( A ) \) leads naturally to~(\ref{bound for contribution of QR}). 

In Section~\ref{sec: contribution to error of solving linear system}, the bounds for \( \wh{\phi}( A ) \) obtained via the componentwise analysis depend on 
\[
\lvert \mymat{\wh{u}}{1}{2} \rvert \wedge \lvert \mymat{\wh{u}}{2}{2} \rvert \le \frac{1}{ \sqrt{ 1 + \eta_{1} } }, \dspace
\lvert \mymat{\wh{v}}{1}{2} \rvert \wedge \lvert \mymat{\wh{v}}{2}{2} \rvert \le \frac{1}{ \sqrt{ 1 + \eta_{2} } } ,
\] 
where \( \eta_{1} \) and \( \eta_{2} \) are quantities that will be formally introduced in~(\ref{def: eta}) and
\[
\eta_{1} \le \frac{ \lvert \mymat{a}{1}{2} \rvert \vee \lvert \mymat{a}{2}{1} \rvert }{ \lvert \mymat{a}{1}{2} \rvert \wedge \lvert \mymat{a}{2}{1} \rvert }, \dspace 
\eta_{2} \le \frac{ \lvert \mymat{a}{3}{4} \rvert \vee \lvert \mymat{a}{4}{3} \rvert }{ \lvert \mymat{a}{3}{4} \rvert \wedge \lvert \mymat{a}{4}{3} \rvert } . 
\]
To establish the bound for~(\ref{fourth entry of Delta a}), we need to relate \( ( \wt{u}_{2}, \wt{v}_{2}, \wt{\sigma}_{2} ) \) to  \( ( \wh{u}_{2}, \wh{v}_{2}, \wh{\sigma}_{2} ) \). Since the relation between \( \wt{v}_{2} \) and \( \wh{v}_{2} \) is analogous to that between \( \wt{u}_{2} \) and \( \wh{u}_{2} \), we only treat the latter. By (\ref{bounds for Delta X and Q}) and \( \wh{\kappa} > 4 \), it follows that 
\[
\lVert \Delta \wh{X} \rVert 
\le \gamma_{6} ( 5\wh{\sigma}_{1} + 8\sqrt{2}\wh{\sigma}_{2} )
\le ( 5 + 2\sqrt{2} )\gamma_{6} \wh{\sigma}_{1} . 
\]
Applying Lemma~\ref{lemma: stability of SVD} to \( \wh{X} \) and \( \Delta \wh{X} \) with \( c = ( 5 + 2\sqrt{2} ) \gamma_{6} \), we obtain  
\[
\lvert \mymat{\wt{u}}{i}{2} \rvert 
\le \lvert \mymat{\wh{u}}{i}{2} \rvert + \frac{ 20 + 8\sqrt{2} }{ 3 } \gamma_{6}, \dspace
i = 1, 2 .
\] 
Notice that as \( c \ll 1 \), we simplify \( d \) to \( c / ( 1 - \wh{\kappa}^{-1} ) \) and bound \( \lvert \mymat{\wt{u}}{i}{2} \rvert \) by \( \lvert \mymat{\wh{u}}{i}{2} \rvert + d \), dropping the \( d^{2} \) term. 
Thus, if \( \eta_{1} \) satisfies 
\begin{align} \label{stable condition for uij}
    \frac{ 20 + 8\sqrt{2} }{ 3 }\gamma_{6} 
    \le \frac{ 0.1 }{ \tsqrt{ 1 + \eta_{1} } } , 
\end{align}
then replacing \( \wh{u}_{2} \) with \( \wt{u}_{2} \) in \( \wh{\phi}( A ) \) and \( \wh{\psi}( A ) \) has a limited effect on the constant factors of bounds obtained via the componentwise analysis. Thus, bounds~(\ref{bound for contribution of residual}) and~(\ref{bound for contribution of QR}) remain unchanged under this situation. Since \( \macheps \) is assumed to be sufficiently small, (\ref{stable condition for uij}) can be further simplified to the assumption for \( \lvert \mymat{a}{1}{2} \rvert \) and \( \lvert \mymat{a}{2}{1} \rvert \) in Theorem~\ref{theo: main result}. 

Finally we turn to \( \wt{\sigma}_{2} \) and \( \wh{\sigma}_{2} \).
Notice that 
\begin{align*}
    \Big\lvert \sqrt{ \frac{ 1 + 2\wt{\sigma}_{2}^{2} }{ 1 + 2\wh{\sigma}_{2}^{2} } } - 1 \Big\rvert
    &= \frac{2( \wh{\sigma}_{2} + \wt{\sigma}_{2} ) \lvert \wh{\sigma}_{2} - \wt{\sigma}_{2} \rvert }{ \sqrt{ 1 + 2\wh{\sigma}_{2}^{2} } ( \sqrt{ 1 + 2\wh{\sigma}_{2}^{2} } + \sqrt{ 1 + 2\wt{\sigma}_{2}^{2} } ) } \\ 
    &\le \frac{ \sqrt{2} \, \lVert \Delta \wh{X} \rVert }{ \sqrt{ 1 + 2\wh{\sigma}_{2}^{2} } } 
    \le ( 4 + 5\sqrt{2} ) \wh{\calK} \gamma_{6} , 
\end{align*} 
Hence if \( \wh{\calK} \gamma_{6} < ( 4 + 5\sqrt{2} )^{-1} \), then 
\[
\sqrt{ \frac{ 1 + 2\wh{\sigma}_{2}^{2} }{ 1 + 2\wt{\sigma}_{2}^{2} } } - 1
\le \frac{ \wh{\calK} \gamma_{6} }{ ( 4 + 5\sqrt{2} )^{-1} - \wh{\calK} \gamma_{6} } \wedge ( \tsqrt{ 1 + 2\wh{\sigma}_{2}^{2} } - 1 ) .   
\] 

Taking \( \gamma_{k} = k\macheps \) for simplicity and putting the above together, from~(\ref{bound for first three entries of Delta A})--(\ref{bound for contribution of QR}) we see that \( \lVert \Delta A \rVert \) is bounded by 
\begin{align*}
    &\lVert \big( D^{-1} ( \wt{V}\trans \otimes \wt{U}\trans ) f \big)_{1 \tightcolon 3} \rVert + ( 1 + \epsilon_{s} ) ( \wh{\phi}( A ) + \wh{\psi}( A ) ) \\
    &\le \big( 78456 + 39936 \sqrt{2} + 78396\epsilon_{s} + 78204\epsilon_{r}( 1 + \epsilon_{s} ) \\
    &\quad + 38304\sqrt{2}( \epsilon_{r} + \epsilon_{s} + \epsilon_{r}\epsilon_{s} ) \big) \macheps \lVert A \rVert 
\end{align*}
if \( A \in \mathcal{S}_{4} \) satisfies the assumption in Theorem~\ref{theo: main result}. Then we conclude the result. 

\section{Contribution to error from solving linear system} \label{sec: contribution to error of solving linear system}

The objective in this section is to prove~(\ref{bound for contribution of residual}). Our analysis begins by imposing convenient conditions on \( A \), using the fact that \( \wh{\phi}( A ) \) is invariant under certain permutations of \( \mymat{A}{i}{j} \). Afterward, we will present a series of propositions characterizing the sizes of \( \lvert \mymat{t}{i}{j} \rvert \). Two cases are identified to require further treatment. For these two cases, the GECP will be analyzed step by step to bound all entries of 
\[
S \coloneqq \wh{P}\trans \lvert \wh{L} \rvert \lvert \wh{U} \rvert \wh{Q} \ge 0 . 
\]
The componentwise bounds for \( S \) finally yield the main result of this section. 

From \( T = I_{2} \otimes \mymat{A}{1}{1} - \mymat[\top]{A}{2}{2} \otimes I_{2} \), we have 
\begin{align} \label{entries of T}
T = \begin{bmatrix}
        \mymat{a}{1}{1} - \mymat{a}{3}{3} & \mymat{a}{1}{2} & -\mymat{a}{4}{3} & 0 \\ 
        \mymat{a}{2}{1} & \mymat{a}{1}{1} - \mymat{a}{3}{3} & 0 & -\mymat{a}{4}{3} \\
        -\mymat{a}{3}{4} & 0 & \mymat{a}{1}{1} - \mymat{a}{3}{3} & \mymat{a}{1}{2} \\
        0 & -\mymat{a}{3}{4} & \mymat{a}{2}{1} & \mymat{a}{1}{1} - \mymat{a}{3}{3}
    \end{bmatrix} .
\end{align} 
The following proposition allows us to impose additional conditions on \( A \) without loss of generality. 
\begin{proposition} \label{prop: simplify A}
    Suppose \( A \in \mathcal{S}_{4} \). Let 
    \[
    B = \begin{bmatrix}
        \mymat[\top]{A}{1}{1} & P\mymat{A}{1}{2} \\
          & \mymat{A}{2}{2}
    \end{bmatrix}, \dspace 
    C = \begin{bmatrix}
        \mymat{A}{1}{1} & \mymat{A}{1}{2}P \\
          & \mymat[\top]{A}{2}{2}
    \end{bmatrix}, \dspace
    D = \begin{bmatrix}
        -\mymat[\top]{A}{2}{2} & \mymat[\top]{A}{1}{2} \\
          & -\mymat[\top]{A}{1}{1} 
    \end{bmatrix} ,
    \]
    where \( P = \begin{bmatrix}
        0 & 1 \\
        1 & 0
    \end{bmatrix} \). For \( A \), \( B \), \( C \), and~\( D  \in \mathcal{S}_{4} \), denote the corresponding computed solutions to linear systems in Algorithm~\ref{alg: modified direct swapping method} step~1 by \( \wh{X} \), \( \wh{Y} \), \( \wh{Z} \), and~\( \wh{W} \), respectively. Then 
    \[ 
    \wh{Y} = P\wh{X}, \dspace \wh{Z} = \wh{X}P, \dspace \wh{W} = \wh{X}\trans 
    \]
    and 
    \[
    \wh{\phi}( A ) = \wh{\phi}( B ) = \wh{\phi}( C ) =  \wh{\phi}( D ) . 
    \] 
\end{proposition}

By Proposition~\ref{prop: simplify A}, for all \( A \in \mathcal{S}_{4} \), one can keep \( \wh{\phi}( A ) \) unchanged while 
\begin{equation} \label{order relation of A11 and A22}
\lvert \mymat{a}{1}{2} \rvert \ge \lvert \mymat{a}{2}{1} \rvert, \dspace \lvert \mymat{a}{4}{3} \rvert \ge \lvert \mymat{a}{3}{4} \rvert 
\end{equation}
by applying suitable permutations to \( \mymat{A}{i}{j} \). Thus, we will insist on (\ref{order relation of A11 and A22}) in the rest of this section. 

From \( \wh{X} = \wh{U}\wh{\Sigma}\wh{V}\trans \), the residual \( R = \mymat{A}{1}{1}\wh{X} - \wh{X}\mymat{A}{2}{2} - \mymat{A}{1}{2} \) can be written as 
\begin{equation} \label{eq: residual eqution}
\wh{U}\trans \mymat{A}{1}{1}\wh{U}\wh{\Sigma} - \wh{\Sigma}\wh{V}\trans \mymat{A}{2}{2}\wh{V} = \wh{U}\trans \mymat{A}{1}{2}\wh{V} + \wh{U}\trans R\wh{V} .
\end{equation}
Set 
\[
\mymat{\wh{A}}{1}{1} = \wh{U}\trans \mymat{A}{1}{1} \wh{U}, \dspace \mymat{\wh{A}}{2}{2} = \wh{V}\trans \mymat{A}{2}{2} \wh{V}, \dspace \mymat{\wh{A}}{1}{2} = \wh{U}\trans ( \mymat{A}{1}{2} + R ) \wh{V} .
\] 
Then (\ref{eq: residual eqution}) can be reduced to 
\begin{equation} \label{eq: reduced residual equation}
    \mymat{\wh{A}}{1}{1}\wh{\Sigma} - \wh{\Sigma}\mymat{\wh{A}}{2}{2} = \mymat{\wh{A}}{1}{2} .
\end{equation}
Denote 
\[
( \mymat{\wh{a}}{i}{j} ) = \begin{bmatrix}
    \mymat{\wh{A}}{1}{1} & \mymat{\wh{A}}{1}{2} \\
                       0 & \mymat{\wh{A}}{2}{2} 
\end{bmatrix} .
\] 
Note that (\ref{order relation of A11 and A22}) does not determine whether \( \lvert \mymat{\wh{a}}{1}{2} \rvert \ge \lvert \mymat{\wh{a}}{2}{1} \rvert \), and the same holds for \( \lvert \mymat{\wh{a}}{3}{4} \rvert \) and \( \lvert \mymat{\wh{a}}{4}{3} \rvert \). However, it is not hard to show that for \( A \) and \( D \) in Proposition~\ref{prop: simplify A}, 
\[
\mymat{\wh{D}}{1}{1} = -\mymat[\top]{\wh{A}}{2}{2}, \dspace \mymat{\wh{D}}{2}{2} = -\mymat[\top]{\wh{A}}{1}{1} .
\]
Moreover, for \( B \) and \( C \) in Proposition~\ref{prop: simplify A}, 
\[ 
\mymat{\wh{B}}{1}{1} = \mymat{\wh{A}}{1}{1}, \dspace 
\mymat{\wh{B}}{2}{2} = \mymat{\wh{A}}{2}{2}, \dspace 
\mymat{\wh{C}}{1}{1} = \mymat{\wh{A}}{1}{1}, \dspace 
\mymat{\wh{C}}{2}{2} = \mymat{\wh{A}}{2}{2} .
\]
Hence, we can further assume for \( A \in \mathcal{S}_{4} \),
\begin{equation} \label{order realation between off-diagonal entries}
    \lvert \mymat{\wh{a}}{1}{2} \rvert \vee \lvert \mymat{\wh{a}}{2}{1} \rvert \ge \lvert \mymat{\wh{a}}{3}{4} \rvert \vee \lvert \mymat{\wh{a}}{4}{3} \rvert .
\end{equation} 

To characterize the sizes of \( \lvert \mymat{\wh{a}}{i}{j} \rvert \) in \( \mymat{\wh{A}}{1}{1} \) and \( \mymat{\wh{A}}{2}{2} \), we introduce two quantities 
\begin{align} \label{def: eta}
\eta_{1} 
= \frac{ \lvert \mymat{\wh{a}}{1}{2} \rvert \vee \lvert \mymat{\wh{a}}{2}{1} \rvert }{ \lvert \mymat{\wh{a}}{1}{2} \rvert \wedge \lvert \mymat{\wh{a}}{2}{1} \rvert }, \dspace 
\eta_{2} 
= \frac{ \lvert \mymat{\wh{a}}{3}{4} \rvert \vee \lvert \mymat{\wh{a}}{4}{3} \rvert }{ \lvert \mymat{\wh{a}}{3}{4} \rvert \wedge \lvert \mymat{\wh{a}}{4}{3} \rvert } .
\end{align}
The following proposition provides an estimate for \( \tau = \max \{ \lvert \mymat{\wh{a}}{1}{2} \rvert, \lvert \mymat{\wh{a}}{2}{1} \rvert, \lvert \mymat{\wh{a}}{3}{4} \rvert, \lvert \mymat{\wh{a}}{4}{3} \rvert \} \). 

\begin{proposition} \label{prop: estimate for tau}
    Suppose \( A \in \mathcal{S}_{4} \) satisfies (\ref{order relation of A11 and A22}), (\ref{order realation between off-diagonal entries}) and \( \wh{\kappa} > 4 \). Then 
    \begin{align*}
        \tau \le 
        \begin{dcases}
        \frac{ 2( 1 + \epsilon_{r} ) }{ \wh{\calK} } \lVert A \rVert \eta_{1}, & \text{if~} \lvert \mymat{\wh{a}}{1}{2} \rvert > \lvert \mymat{\wh{a}}{2}{1} \rvert , \\[0.8ex]
        \frac{ 4( 1 + \epsilon_{r} ) }{ 3\wh{\sigma}_{1} } \lVert A \rVert, & \text{if~} \lvert \mymat{\wh{a}}{1}{2} \rvert \le \lvert \mymat{\wh{a}}{2}{1} \rvert .
        \end{dcases}
    \end{align*}
\end{proposition}

Next we establish upper bounds for the smallest entries in \( \lvert \wh{U} \rvert \) and \( \lvert \wh{V} \rvert \). 
\begin{proposition} \label{prop: estimate for smallest elements in U and V}
    Suppose \( A \in \mathcal{S}_{4} \). Let \( \wh{U} = (\mymat{\wh{u}}{i}{j}) \) and \( \wh{V} = (\mymat{\wh{v}}{i}{j}) \). 
    Then
    \[
    \delta_{1} \coloneqq \min \lvert \mymat{\wh{u}}{i}{j} \rvert \le \frac{ 1 }{ \tsqrt{ 1 + \eta_{1}} } , \dspace 
    \delta_{2} \coloneqq \min \lvert \mymat{\wh{v}}{i}{j} \rvert \le \frac{ 1 }{ \tsqrt{ 1 + \eta_{2}} } .
    \]
\end{proposition}
 
The proof of Proposition~\ref{prop: estimate for smallest elements in U and V} can be found in Appendix~\ref{app: proof of estimate for smallest elements in U and V}. 
Under the condition (\ref{order relation of A11 and A22}), we have 
\begin{align} \label{positions of smaller entries of U and V} 
    \delta_{1}  
    &= \begin{dcases}
        \lvert \mymat{\wh{u}}{1}{2} \rvert 
        = \lvert \mymat{\wh{u}}{2}{1} \rvert, & \text{if~} \lvert \mymat{\wh{a}}{1}{2} \rvert > \lvert \mymat{\wh{a}}{2}{1} \rvert , \\[0.8ex]
        \lvert \mymat{\wh{u}}{1}{1} \rvert
        = \lvert \mymat{\wh{u}}{2}{2} \rvert, & \text{if~} \lvert \mymat{\wh{a}}{1}{2} \rvert \le \lvert \mymat{\wh{a}}{2}{1} \rvert ,
    \end{dcases} \\ 
    \delta_{2} 
    &= \begin{dcases}
        \lvert \mymat{\wh{v}}{1}{2} \rvert 
        = \lvert \mymat{\wh{v}}{2}{1} \rvert, & \text{if~} \lvert \mymat{\wh{a}}{4}{3} \rvert > \lvert \mymat{\wh{a}}{3}{4} \rvert , \\[0.8ex]
        \lvert \mymat{\wh{v}}{1}{1} \rvert
        = \lvert \mymat{\wh{v}}{2}{2} \rvert, & \text{if~} \lvert \mymat{\wh{a}}{4}{3} \rvert \le \lvert \mymat{\wh{a}}{3}{4} \rvert .
    \end{dcases}
\end{align} 
In addition, the following proposition can also be derived from the proof. It characterizes the relation between the off-diagonal entries of \( \mymat{A}{1}{1} \) and \( \mymat{\wh{A}}{1}{1} \), as well as \( \mymat{A}{2}{2} \) and \( \mymat{\wh{A}}{2}{2} \).
\begin{proposition} \label{prop: relation of entries between A11 and hatA11} 
    Suppose \( A \in \mathcal{S}_{4} \) satisfies (\ref{order relation of A11 and A22}). For \( \mymat{A}{1}{1} \) and \( \mymat{\wh{A}}{1}{1} \), 
    \[
    \begin{array}{llll}
        \textup{(1)} & \lvert \mymat{a}{1}{2} - \mymat{a}{2}{1} \rvert 
        = \lvert \mymat{\wh{a}}{1}{2} - \mymat{\wh{a}}{2}{1} \rvert, 
        & \textup{(2)} & \lvert \mymat{a}{2}{1} \rvert 
        \le \lvert \mymat{\wh{a}}{1}{2} \rvert \wedge \lvert \mymat{\wh{a}}{2}{1} \rvert, \\
        \textup{(3)} & \lvert \mymat{a}{1}{2} \rvert 
        \ge \lvert \mymat{\wh{a}}{1}{2} \rvert \vee \lvert \mymat{\wh{a}}{2}{1} \rvert, 
        & \textup{(4)} & \lvert \mymat{a}{1}{2} \rvert 
        \le \lvert \mymat{\wh{a}}{1}{2} \rvert + \lvert \mymat{\wh{a}}{2}{1} \rvert .
    \end{array}
    \]
    Analogous results are also valid for \( \mymat{A}{2}{2} \) and \( \mymat{\wh{A}}{2}{2} \). 
\end{proposition}

Finally, we present a bound for \( \lvert \mymat{a}{1}{1} - \mymat{a}{3}{3} \rvert \), the diagonal entry of \( \lvert T \rvert \). 

\begin{proposition} \label{prop: bound for diagonal entries of |T|}
    Suppose \( A \in \mathcal{S}_{4} \) satisfies (\ref{order relation of A11 and A22}). Then 
    \[
    \lvert \mymat{a}{1}{1} - \mymat{a}{3}{3} \rvert 
    \le \frac{ 1 + \epsilon_{r} }{ \wh{\sigma}_{1} } \lVert A \rVert 
    + \delta_{1} \lvert \mymat{a}{1}{2} \rvert 
    + \delta_{2} \lvert \mymat{a}{4}{3} \rvert .
    \] 
\end{proposition} 

By (\ref{entries of T}), we have \( \mymat{t}{i}{j} = 0 \) for \( i + j = 5 \). Assume \( \lvert \mymat{\wh{a}}{1}{2} \rvert \le \lvert \mymat{\wh{a}}{2}{1} \rvert \). By Propositions~\ref{prop: estimate for smallest elements in U and V}, \ref{prop: relation of entries between A11 and hatA11}~(4), and~\ref{prop: bound for diagonal entries of |T|}, for~\( i + j \ne 5 \), 
\begin{align*} 
    \lvert \mymat{t}{i}{j} \rvert \le 
    \begin{dcases}
        \frac{ 8( 1 + \epsilon_{r} ) }{ 3\wh{\sigma}_{1} } \lVert A \rVert, & i \ne j , \\[0.8ex]
        \frac{ ( 3 + 8\sqrt{2} ) ( 1 + \epsilon_{r} ) }{ 3\wh{\sigma}_{1} } \lVert A \rVert, & i = j .
    \end{dcases} 
\end{align*}  
By the structure of \( T \), it follows that 
\[
\lVert T \rVert_{\infty} 
= \lvert \mymat{a}{1}{1} - \mymat{a}{3}{3} \rvert + \lvert \mymat{a}{1}{2} \rvert + \lvert \mymat{a}{4}{3} \rvert 
\le \frac{ ( 19 + 8\sqrt{2} ) ( 1 + \epsilon_{r} ) }{ 3\wh{\sigma}_{1} } \lVert A \rVert . 
\]
By Lemma~\ref{lemma: norm inequalities} and (\ref{infinite norm bound for residual}), we obtain 
\begin{equation} \label{bound 0 for contribution of residual}
    \wh{\phi}( A ) 
    \le 32 \rho_{4} \gamma_{12} \lVert T \rVert_{\infty} \cdot \wh{\calK} 
    \le \frac{ ( 2432 + 1024\sqrt{2} ) ( 1 + \epsilon_{r} ) }{3} \cdot \gamma_{12} \lVert A \rVert .
\end{equation}
Therefore, \( \wh{\phi}( A ) \) has been bounded in the desired form as long as \( \lvert \mymat{\wh{a}}{1}{2} \rvert \le \lvert \mymat{\wh{a}}{2}{1} \rvert \).

Hence, under the conditions (\ref{order relation of A11 and A22}), (\ref{order realation between off-diagonal entries}) and \( \wh{\kappa} > 4 \), only two cases require further treatment, i.e., 
\[
\text{case 1: } \lvert \mymat{\wh{a}}{1}{2} \rvert > \lvert \mymat{\wh{a}}{2}{1} \rvert, \ \lvert \mymat{\wh{a}}{3}{4} \rvert \ge \lvert \mymat{\wh{a}}{4}{3} \rvert \dspace \text{and} \dspace 
\text{case 2: } \lvert \mymat{\wh{a}}{1}{2} \rvert > \lvert \mymat{\wh{a}}{2}{1} \rvert, \ \lvert \mymat{\wh{a}}{3}{4} \rvert < \lvert \mymat{\wh{a}}{4}{3} \rvert .
\]
We now discuss case 1 and case 2 separately. 

\subsection{Case 1} \label{subsec: case 1}
In this subsection, we make the following two assumptions:
\begin{align} 
    \label{case 1: assumption 1}
    &\eta_{1} 
    \ge 3, \\ 
    \label{case 1: assumption 2}
    &\lvert \mymat{a}{1}{2} \rvert 
    \ge \frac{ ( 4 + 8\sqrt{2} ) ( 1 + \epsilon_{r} ) }{ \wh{\calK} } \lVert A \rVert . 
\end{align} 
These assumptions follow from the fact that if any one of them fails, 
\begin{equation} \label{bound for phi(A) in case 1}
    \wh{\phi} ( A ) 
    \le 32 \rho_{4} \gamma_{12} \lVert T \rVert_{\infty} \cdot \wh{\calK} 
    \le ( 2176 + 1536\sqrt{2} ) ( 1 + \epsilon_{r} ) \cdot \gamma_{12} \lVert A \rVert   
\end{equation}
and the claim in~(\ref{bound for contribution of residual}) follows. See Appendix~\ref{app: proof of componentwise bounds of |T| for case 1}. 

\begin{proposition} \label{prop: componentwise bounds of |T| for case 1}
    Suppose \( A \in \mathcal{S}_{4} \) satisfies conditions (\ref{order relation of A11 and A22}), (\ref{order realation between off-diagonal entries}) and \( \wh{\kappa} > 4 \). For case 1, under the assumptions~(\ref{case 1: assumption 1})--(\ref{case 1: assumption 2}), entries in \( \lvert T \rvert \) admit the following bounds.
    \begin{itemize}
        \item Relative bounds 
        \begin{align}
            \label{case 1: relative bounds for entries of |T| -- 1}
            &\lvert \mymat{a}{2}{1} \rvert 
            \le \frac{1}{3} \lvert \mymat{a}{1}{2} \rvert, \dspace 
            \lvert \mymat{a}{3}{4} \rvert 
            \le \lvert \mymat{a}{4}{3} \rvert 
            \le \frac{ 2\sqrt{2} - 1 }{7} \lvert \mymat{a}{1}{2} \rvert, \\
            \label{case 1: relative bounds for entries of |T| -- 2}
            &\lvert \mymat{a}{1}{1} - \mymat{a}{3}{3} \rvert 
            \le \frac{3}{4} \lvert \mymat{a}{1}{2} \rvert . 
        \end{align}

        \item Absolute bounds 
        \begin{align}
            \label{case 1: bound for |a12|}
            &\lvert \mymat{a}{1}{2} \rvert 
            \le \frac{ 2 ( 1 + \epsilon_{r} ) }{\wh{\calK}} \lVert A \rVert \cdot ( \eta_{1} + 1 ), \\ 
            \label{case 1: bound for |a21|}
            &\lvert \mymat{a}{2}{1} \rvert 
            \le \frac{ 2 ( 1 + \epsilon_{r} ) }{\wh{\calK}} \lVert A \rVert, \\ 
            \label{case 1: bound for |a34| and |a43|}
            &\lvert \mymat{a}{3}{4} \rvert \le \lvert \mymat{a}{4}{3} \rvert 
            \le \frac{ 4 ( 1 + \epsilon_{r} ) }{\wh{\calK}} \lVert A \rVert, \\
            \label{case 1: bound for (a11-a33)^2/|a12|}
            &\frac{ ( \mymat{a}{1}{1} - \mymat{a}{3}{3} )^{2} }{ \lvert \mymat{a}{1}{2} \rvert } 
            \le \frac{ ( 9 + 2\sqrt{2} ) ( 1 + \epsilon_{r} ) }{ 2\wh{\calK} } \lVert A \rVert .
        \end{align}
    \end{itemize}
\end{proposition}

\textbf{The first step of GECP.}
Notice that the entry \( \mymat{a}{1}{2} \) appears at both positions \( ( 1, 2 ) \) and \( ( 3, 4 ) \), so either can be chosen as the first pivot. Indeed, we will show that \( \mymat{a}{1}{2} \) in \( ( 1, 2 ) \) and \( ( 3, 4 ) \) must be chosen as pivots in the first two steps of GECP. Without loss of generality, assume \( \mymat{a}{1}{2} \) in \( ( 1, 2 ) \) is chosen as the first pivot. Then after the first pivoting, 
\[
\wh{P}_{1} T \wh{Q}_{1}\trans = 
\begin{bNiceMatrix}[first-row, first-col, margin=4pt] 
    & \scriptstyle 2 & \scriptstyle 1 & \scriptstyle 3 & \scriptstyle 4 \\
    \scriptstyle 1 & \mymat{a}{1}{2} & \mymat{a}{1}{1} - \mymat{a}{3}{3} & -\mymat{a}{4}{3} & 0 \\
    \scriptstyle 2 & \mymat{a}{1}{1} - \mymat{a}{3}{3} & \mymat{a}{2}{1} & 0 & -\mymat{a}{4}{3} \\
    \scriptstyle 3 & 0 & -\mymat{a}{3}{4} & \mymat{a}{1}{1} - \mymat{a}{3}{3} & \mymat{a}{1}{2} \\
    \scriptstyle 4 & -\mymat{a}{3}{4} & 0 & \mymat{a}{2}{1} & \mymat{a}{1}{1} - \mymat{a}{3}{3}
\end{bNiceMatrix} ,
\]
in which \( \wh{P}_{1} = I_{4} \) and \( \wh{Q}_{1}\trans \) swaps columns \( 1 \) and \( 2 \).
Applying \( L_{1} = I_{4} - l_{1} e_{1}\trans \) to \( \wh{P}_{1}T \wh{Q}_{1}\trans \), where \( e_{i} \) is the \( i \)th column of \( I_{4} \) and 
\[
l_{1} = [0, ( \mymat{a}{1}{1} - \mymat{a}{3}{3} ) \mymat[-1]{a}{1}{2}, 0, -\mymat{a}{3}{4}\mymat[-1]{a}{1}{2}]\trans , 
\]
we obtain \( \wh{T}^{(1)} = ( \mymat[(1)]{\wh{t}}{i}{j} ) = \sffl( L_{1}\wh{P}_{1}T\wh{Q}_{1}\trans ) \). 

We track the growth of the entry initially at position \( ( i, j ) \) in \( T \) as it is moved by pivoting. Due to zero entries in \( (1, 4) \) and \( (3, 2) \), only entries in \( (2, 1) \), \( (2, 3) \), \( (4, 1) \) and \( (4, 3) \) need to be updated. By Proposition~\ref{prop: componentwise bounds of |T| for case 1}, \( \lvert \mymat[(1)]{\wh{t}}{i}{j} \rvert \) in these positions can be bounded as follows. Since  \( \gamma_{k} \ll 1 \) in these settings, factors of the form \( ( 1 + \gamma_{k} ) \cdot C \) are all replaced by \( C \) for simplicity. 
\begin{align*}
    ( 2, 1 ) \colon \lvert \mymat[(1)]{\wh{t}}{2}{2} \rvert 
    &\le ( 1 + \gamma_{5} ) \cdot \big( \lvert \mymat{a}{2}{1} \rvert + ( \mymat{a}{1}{1} - \mymat{a}{3}{3} )^{2} \lvert \mymat[-1]{a}{1}{2} \rvert \big) 
    \le \frac{ ( 13 + 2\sqrt{2} ) ( 1 + \epsilon_{r} ) }{ 2\wh{\calK} } \lVert A \rVert, \\
    ( 2, 3 ) \colon \lvert \mymat[(1)]{\wh{t}}{2}{3} \rvert 
    &\le ( 1 + \gamma_{3} ) \cdot \lvert ( ( \mymat{a}{1}{1} - \mymat{a}{3}{3}) \mymat[-1]{a}{1}{2} ) \mymat{a}{4}{3} \rvert 
    \le \frac{ 3 ( 1 + \epsilon_{r} ) }{ \wh{\calK} } \lVert A \rVert, \\
    ( 4, 1 ) \colon \lvert \mymat[(1)]{\wh{t}}{4}{2} \rvert 
    &\le ( 1 + \gamma_{3} ) \cdot \lvert \mymat{a}{3}{4} ( \mymat[-1]{a}{1}{2} (\mymat{a}{1}{1} - \mymat{a}{3}{3}) ) \rvert 
    \le \frac{ 3 ( 1 + \epsilon_{r} ) }{ \wh{\calK} } \lVert A \rVert, \\
    ( 4, 3 ) \colon \lvert \mymat[(1)]{\wh{t}}{4}{3} \rvert 
    &\le ( 1 + \gamma_{3}) \cdot \big( \lvert \mymat{a}{2}{1} \rvert + \lvert \mymat{a}{3}{4} ( \mymat[-1]{a}{1}{2} \mymat{a}{4}{3} ) \rvert \big)
    \le \frac{ ( 10 + 8\sqrt{2} ) ( 1 + \epsilon_{r} ) }{ 7\wh{\calK} } \lVert A \rVert .
\end{align*}  
Meanwhile, it is not hard to show that \( \lvert \mymat{a}{1}{2} \rvert \) in \( (3, 4) \) will be the largest entry in \( \lvert \wh{T}^{(1)}( 2 \tightcolon 4, 2 \tightcolon 4 ) \rvert \) by referring to~(\ref{case 1: relative bounds for entries of |T| -- 1})--(\ref{case 1: relative bounds for entries of |T| -- 2}). For brevity, we omit details.

\textbf{The second step of GECP.} After the second pivoting, we have 
\[
\wh{P}_{2} \wh{T}^{(1)} \wh{Q}_{2}\trans = 
\begin{bNiceMatrix}[first-row, first-col, margin=4pt] 
    & \scriptstyle 2 & \scriptstyle 4 & \scriptstyle 3 & \scriptstyle 1 \\
    \scriptstyle 1 & \mymat{a}{1}{2} & 0 & -\mymat{a}{4}{3} & \mymat{a}{1}{1} - \mymat{a}{3}{3} \\
    \scriptstyle 3 & 0 & \mymat{a}{1}{2} & \mymat{a}{1}{1} - \mymat{a}{3}{3} & -\mymat{a}{3}{4} \\
    \scriptstyle 2 & 0 & -\mymat{a}{4}{3} & \mymat[(1)]{\wh{t}}{2}{3} & \mymat[(1)]{\wh{t}}{2}{2} \\
    \scriptstyle 4 & 0 & \mymat{a}{1}{1} - \mymat{a}{3}{3} & \mymat[(1)]{\wh{t}}{4}{3} & \mymat[(1)]{\wh{t}}{4}{2}
\end{bNiceMatrix} ,
\] 
in which \( \wh{P}_{2} \) swaps rows \( 2 \) and \( 3 \), \( \wh{Q}_{2}\trans \) swaps columns \( 2 \) and \( 4 \). 
Then applying \( L_{2} = I_{4} - l_{2}e_{2}\trans \) to \( \wh{P}_{2}\wh{T}^{(1)}\wh{Q}_{2}\trans \), where 
\[
l_{2}
= [0, 0, -\mymat{a}{4}{3}\mymat[-1]{a}{1}{2}, (\mymat{a}{1}{1} - \mymat{a}{3}{3})\mymat[-1]{a}{1}{2}]\trans ,
\] 
we obtain \( \wh{T}^{(2)} = (\mymat[(2)]{\wh{t}}{i}{j} ) = \sffl( L_{2}\wh{P}_{2}\wh{T}^{(1)}\wh{Q}_{2}\trans ) \).

Entries in \( (2, 1) \), \( (2, 3) \), \( (4, 1) \) and \( (4, 3) \) will be updated again. Bounds for \( \lvert \mymat[(2)]{\wh{t}}{i}{j} \rvert \) in these positions are as follows:
\begin{align*}
    (2, 1) \colon \lvert \mymat[(2)]{\wh{t}}{3}{4} \rvert
    &\le ( 1 + \gamma_{3} ) \cdot \big( \lvert \mymat[(1)]{\wh{t}}{2}{2} \rvert + \lvert ( \mymat{a}{4}{3}\mymat[-1]{a}{1}{2} ) \mymat{a}{3}{4} \rvert \big)
    \le \frac{ ( 83 + 30\sqrt{2} ) ( 1 + \epsilon_{r} ) }{ 14\wh{\calK} } \lVert A \rVert, \\
    (2, 3) \colon \lvert \mymat[(2)]{\wh{t}}{3}{3} \rvert 
    &\le ( 1 + \gamma_{4} ) \cdot \big( \lvert \mymat[(1)]{\wh{t}}{2}{3} \rvert + \lvert \mymat{a}{4}{3} ( \mymat[-1]{a}{1}{2} ( \mymat{a}{1}{1} - \mymat{a}{3}{3} ) ) \rvert \big)
    \le \frac{ 6 ( 1 + \epsilon_{r} ) }{\wh{\calK}} \lVert A \rVert, \\
    (4, 1) \colon \lvert \mymat[(2)]{\wh{t}}{4}{4} \rvert 
    &\le ( 1 + \gamma_{4} ) \cdot \big( \lvert \mymat[(1)]{\wh{t}}{4}{2} \rvert + \lvert ( ( \mymat{a}{1}{1} - \mymat{a}{3}{3} ) \mymat[-1]{a}{1}{2} ) \mymat{a}{3}{4} \rvert \big)
    \le \frac{ 6 ( 1 + \epsilon_{r} ) }{ \wh{\calK} } \lVert A \rVert, \\
    (4, 3) \colon \lvert \mymat[(2)]{\wh{t}}{4}{3} \rvert
    &\le ( 1 + \gamma_{5} ) \cdot \big( \lvert \mymat[(1)]{\wh{t}}{4}{3} \rvert + ( \mymat{a}{1}{1} - \mymat{a}{3}{3} )^{2} \lvert \mymat[-1]{a}{1}{2} \rvert \big)
    \le \frac{ ( 83 + 30\sqrt{2} ) ( 1 + \epsilon_{r} ) }{ 14 \wh{\calK} } \lVert A \rVert . 
\end{align*}

\textbf{The third step of GECP.}
Now it is not clear which is the largest entry of \( \lvert \wh{T}^{(2)}( 3 \tightcolon 4, 3 \tightcolon 4 ) \rvert \). 
However, by complete pivoting, the resulting update is bounded by the third largest entry of \( \lvert \wh{T}^{(2)}( 3 \tightcolon 4, 3 \tightcolon 4 ) \rvert \). (We ignore the factor \( 1 + \gamma_{k} \) as before.)
Thus, in the third step of GECP, the update is at most 
\begin{align*} 
    \frac{ 6 ( 1 + \epsilon_{r} ) }{ \wh{\calK} } \lVert A \rVert .
\end{align*} 

We summarize our results as follows. Notice that by the statement \( \mymat{s}{i}{j} = \lvert \mymat{t}{i}{j} \rvert \) we mean that the relative error between \( \mymat{s}{i}{j}  \) and \( \lvert \mymat{t}{i}{j} \rvert \) is at most \( \gamma_{3} \), a number that is small enough to be negligible in the analysis.
\begin{itemize}
    \item For \( (i, j) \notin \{ ( 2, 1 ), ( 2, 3 ), ( 4, 1 ), ( 4, 3 ) \} \), \( \mymat{s}{i}{j} = \lvert \mymat{t}{i}{j} \rvert \).

    \item For \( (i, j) \in \{ (2, 1), (4, 3) \} \), 
    \begin{align*} 
        \notag \mymat{s}{i}{j} 
        &\le \lvert \mymat{a}{2}{1} \rvert + 2 \Big( ( \mymat{a}{1}{1} - \mymat{a}{3}{3} )^{2} \lvert \mymat[-1]{a}{1}{2} \rvert + \lvert \mymat{a}{3}{4} \mymat[-1]{a}{1}{2} \mymat{a}{4}{3} \rvert + \frac{ 6( 1 + \epsilon_{r} ) } { \wh{\calK} } \lVert A \rVert \Big) \\
        &\le \frac{ ( 153 + 30\sqrt{2} ) ( 1 + \epsilon_{r} ) }{ 7\wh{\calK} } \lVert A \rVert . 
    \end{align*}

   \item For \( (i, j) \in \{ (2, 3), (4, 1) \} \), 
    \begin{align*} 
        \notag \mymat{s}{i}{j} 
        &\le 2\Big( \lvert (\mymat{a}{1}{1} - \mymat{a}{3}{3}) \mymat[-1]{a}{1}{2} \mymat{a}{4}{3} \rvert + \frac{ 6 ( 1 + \epsilon_{r} ) }{ \wh{\calK} } \lVert A \rVert \Big) \\
        &\le \frac{ 18 ( 1 + \epsilon_{r} ) }{ \wh{\calK} } \lVert A \rVert . 
    \end{align*}
\end{itemize}

We are now prepared to bound \( \wh{\phi}( A ) \). Notice that \( \wh{X} \) can be split into two parts:
\[
\wh{X} 
= \wh{U}\wh{\Sigma}\wh{V}\trans
= \wh{U}\begin{bmatrix}
    \wh{\sigma}_{1} & 0 \\
                 0  & 0
\end{bmatrix} \wh{V}\trans
+ \wh{U}\begin{bmatrix}
    0  & 0 \\
    0  & \wh{\sigma}_{2}
\end{bmatrix}\wh{V}\trans \coloneqq \wh{X}_{1} + \wh{X}_{2} .
\]
Using Lemma~\ref{lemma: norm inequalities}, the contribution of \( \wh{X}_{2} \) to the bound of \( \wh{\phi}( A ) \) is bounded by  
\begin{equation} \label{contribution of X_2}
    64 \gamma_{12} \lVert T \rVert \cdot \frac{ \lVert \sfvec( \wh{X}_{2} ) \rVert }{ \sqrt{ 1 + 2\wh{\sigma}_{2}^{2} } } 
    \le 64\sqrt{2} \, \gamma_{12} \lVert A \rVert .
\end{equation} 

Next we bound the contribution of \( \wh{X}_{1} \) to \( \wh{\phi}( A ) \). For convenience, we denote \( \wh{X}_{1} \) by \( \wh{X} \). Then
\[
\lvert \mymat{\wh{X}}{i}{j} \rvert 
= \Big\lvert ( \wh{U}\trans e_{i} )\trans 
\begin{bmatrix}
    \wh{\sigma}_{1} & 0 \\
    0 & 0
\end{bmatrix} ( \wh{V} e_{j} ) \Big\rvert 
\le \lvert \mymat{\wh{u}}{i}{1} \mymat{\wh{v}}{1}{j} \rvert \wh{\sigma}_{1}, \dspace 1 \le i, j \le 2 . 
\]
From (\ref{positions of smaller entries of U and V}), it follows that 
\begin{equation} \label{case 1: component bound for x}
    \lvert \wh{x} \rvert 
    \le [ \delta_{2} \wh{\sigma}_{1}, \delta_{1} \delta_{2} \wh{\sigma}_{1}, \wh{\sigma}_{1}, \delta_{1} \wh{\sigma}_{1} ]\trans .  
\end{equation}
For \( r = \sfvec(R) = [ r_{1}, r_{2}, r_{3}, r_{4} ]\trans \), we have
\begin{equation} \label{case 1: component bound for uRv}
    \lvert \wh{u}_{2}\trans R\wh{v}_{2} \rvert 
    \le \delta_{1} \lvert r_{1} \rvert + \lvert r_{2} \rvert + \delta_{1} \delta_{2} \lvert r_{3} \rvert + \delta_{2} \lvert r_{4} \rvert .
\end{equation} 
Remember that by Proposition~\ref{prop: estimate for smallest elements in U and V},  
\[
\delta_{1} \le \frac{1}{ \sqrt{ 1 + \eta_{1} } }, \dspace 
\delta_{2} \le \frac{1}{ \sqrt{ 1 + \eta_{2} } } , 
\]
in which \( \eta_{1} \ge 3 \) and \( \eta_{2} \ge 1 \). Combining these with~(\ref{componentwise bound for residual}), (\ref{case 1: component bound for x}) and~(\ref{case 1: component bound for uRv}), the contribution of each \( \mymat{s}{i}{j} \) to \( \wh{\phi}( A ) \) is bounded as follows: 
\[
\begin{bNiceMatrix}[first-row, first-col, margin=4pt] 
    & \scriptstyle \delta_{2} \wh{\sigma}_{1} & \scriptstyle \delta_{1} \delta_{2} \wh{\sigma}_{1} & \scriptstyle \wh{\sigma}_{1} & \scriptstyle \delta_{1}\wh{\sigma}_{1} \\
    \scriptstyle \delta_{1} & \mymat{s}{1}{1} & \mymat{s}{1}{2} & \mymat{s}{1}{3} & \mymat{s}{1}{4} \\
    \scriptstyle 1 & \mymat{s}{2}{1} & \mymat{s}{2}{2} & \mymat{s}{2}{3} & \mymat{s}{2}{4} \\
    \scriptstyle \delta_{1} \delta_{2} & \mymat{s}{3}{1} & \mymat{s}{3}{2} & \mymat{s}{3}{3} & \mymat{s}{3}{4} \\
    \scriptstyle \delta_{2} & \mymat{s}{4}{1} & \mymat{s}{4}{2} & \mymat{s}{4}{3} & \mymat{s}{4}{4}
\end{bNiceMatrix} 
\]
\begin{itemize} 
    \item For \( (i, j) \in \{ (1, 2), (3, 4) \} \), \( \mymat{s}{i}{j} = \lvert \mymat{a}{1}{2} \rvert \). Then the contribution of two \( \mymat{s}{i}{j} \) is bounded by
    \begin{equation} \label{case 1: contribution of a12}
        2 \gamma_{12} \delta_{1}^{2} \delta_{2} \lvert \mymat{a}{1}{2} \rvert  \cdot \wh{\calK}
        \le 2 \sqrt{2} ( 1 + \epsilon_{r} ) \cdot \gamma_{12} \lVert A \rVert. 
    \end{equation}
    
    \item For \( (i, j) \in \{ (1, 1), (2, 2), (3, 3), (4, 4) \} \), \( \mymat{s}{i}{j} = \lvert \mymat{a}{1}{1} - \mymat{a}{3}{3} \rvert \). By Proposition~\ref{prop: bound for diagonal entries of |T|}, the contribution of four \( \mymat{s}{i}{j} \) is bounded by 
    \begin{align} \label{case 1: contribution of a11 - a33}
        \notag 4 \gamma_{12} \delta_{1} \delta_{2} \lvert \mymat{a}{1}{1} - \mymat{a}{3}{3} \rvert \cdot \wh{\calK}
        &\le 4 \gamma_{12} \Big( \frac{ 1 + \epsilon_{r} }{ \wh{\sigma}_{1} }  \delta_{1} \delta_{2} \lVert A \rVert + \delta_{1}^{2} \delta_{2} \lvert \mymat{a}{1}{2} \rvert + \delta_{1} \delta_{2}^{2} \lvert \mymat{a}{4}{3} \rvert \Big) \cdot \calK \\
        &\le ( 4 + 5\sqrt{2} ) ( 1 + \epsilon_{r} ) \cdot \gamma_{12} \lVert A \rVert .
    \end{align}

    \item For \( ( i, j ) \in \{ ( 1, 3 ), ( 2, 4 ) \} \), \( \mymat{s}{i}{j} = \lvert \mymat{a}{4}{3} \rvert \). Then the contribution of two \( \mymat{s}{i}{j} \) is bounded by 
    \begin{equation} \label{case 1: contribution of a43}
        2\gamma_{12} \delta_{1} \lvert \mymat{a}{4}{3} \rvert \cdot \wh{\calK} 
        \le 4( 1 + \epsilon_{r} ) \cdot \gamma_{12} \lVert A \rVert . 
    \end{equation}

    \item For \( ( i, j ) \in \{ ( 3, 1 ), ( 4, 2 ) \} \), \( \mymat{s}{i}{j} = \lvert \mymat{a}{3}{4} \rvert \). Then the contribution of two \( \mymat{s}{i}{j} \) is bounded by 
    \begin{equation} \label{case 1: contribution of a34}
        2\gamma_{12} \delta_{1} \delta_{2}^{2} \lvert \mymat{a}{3}{4} \rvert \cdot \wh{\calK} 
        \le 2( 1 + \epsilon_{r} ) \cdot \gamma_{12} \lVert A \rVert .
    \end{equation}

    \item For \( ( i, j ) \in \{ ( 2, 1 ), ( 4, 3 ) \} \), the contribution of two \( \mymat{s}{i}{j} \) is bounded by 
    \begin{equation} \label{case 1: contribution of s21 and s43}
        2\gamma_{12} \delta_{2} \mymat{s}{i}{j} \cdot \wh{\calK} 
        \le \frac{ ( 60 + 153\sqrt{2} ) ( 1 + \epsilon_{r} ) }{7} \cdot \gamma_{12} \lVert A \rVert . 
    \end{equation}

    \item For \( ( i, j ) \in \{ ( 2, 3 ), ( 4, 1 ) \} \), the contribution of two \( \mymat{s}{i}{j} \) is bounded by 
    \begin{equation} \label{case 1: contribution of s23 and s41}
        \gamma_{12} ( \delta_{2}^{2} + 1 ) \mymat{s}{i}{j} \cdot \wh{\calK} 
        \le 27( 1 + \epsilon_{r} ) \cdot \gamma_{12} \lVert A \rVert . 
    \end{equation}

    \item The remaining \( \mymat{s}{i}{j} = 0 \). 
\end{itemize} 
By summing the bounds~(\ref{contribution of X_2}) and~(\ref{case 1: contribution of a12})--(\ref{case 1: contribution of s23 and s41}), we obtain  
\begin{equation} \label{smaller bound for phi(A) in case 1}
    \wh{\phi} ( A ) 
    \le \Big( \frac{ 319 + 202\sqrt{2} }{ 7 } ( 1 + \epsilon_{r} ) + 64 \sqrt{2} \Big) \cdot \gamma_{12} \lVert A \rVert  
\end{equation}
under assumptions~(\ref{case 1: assumption 1})--(\ref{case 1: assumption 2}). This bound for \( \wh{\phi}( A ) \) is smaller than the bound given in (\ref{bound for phi(A) in case 1}). Thus, we conclude that (\ref{bound for phi(A) in case 1}) is always valid for case~1. 

\subsection{Case 2} \label{subsec: case 2} 
In this subsection, we make the following two assumptions:
\begin{align} 
    \label{case 2: assumption 1}
    &\eta_{1} \ge 8, \\ 
    \label{case 2: assumption 2}
    &\lvert \mymat{a}{1}{2} \rvert \ge \frac{ 16 ( 1 + \epsilon_{r} ) }{ \wh{\calK} } \lVert A \rVert . 
\end{align} 
And if any one of them fails, 
\begin{equation} \label{bound for phi(A) in case 2}
    \wh{\phi} ( A ) 
    \le 32 \rho_{4} \gamma_{12} \lVert T \rVert_{\infty} \cdot \wh{\calK} 
    \le ( 6272 + 3072\sqrt{2} ) ( 1 + \epsilon_{r} ) \cdot \gamma_{12} \lVert A \rVert . 
\end{equation} 
See Appendix~\ref{app: proof of componentwise bounds of |T| for case 2} for a brief proof of~(\ref{bound for phi(A) in case 2}).

\begin{proposition} \label{prop: componentwise bounds of |T| for case 2} 
    Suppose \( A \in \mathcal{S}_{4} \) satisfies conditions (\ref{order relation of A11 and A22}), (\ref{order realation between off-diagonal entries}) and \( \wh{\kappa} > 4 \). For case 2, under the assumptions~(\ref{case 2: assumption 1})--(\ref{case 2: assumption 2}), entries in \( \lvert T \rvert \) admit the following bounds.
    \begin{itemize}
        \item Relative bounds 
        \begin{align} \label{case 2: relative bound for |T| -- 1}
            &\lvert \mymat{a}{2}{1} \rvert 
            \le \frac{1}{8} \lvert \mymat{a}{1}{2} \rvert, \dspace 
            \lvert \mymat{a}{3}{4} \rvert 
            \le \frac{1}{8} \lvert \mymat{a}{1}{2} \rvert, \\ 
            \label{case 2: relative bound for |T| -- 2}
            &\lvert \mymat{a}{4}{3} \rvert 
            \le \frac{9}{8} \lvert \mymat{a}{1}{2} \rvert, \dspace 
            \lvert \mymat{a}{1}{1} - \mymat{a}{3}{3} \rvert 
            \le \frac{37}{48} \lvert \mymat{a}{1}{2} \rvert . 
        \end{align}

        \item Absolute bounds 
        \begin{align} 
            \label{case 2: bound for |a12|} 
            &\lvert \mymat{a}{1}{2} \rvert 
            \le \frac{ 2 ( 1 + \epsilon_{r} ) }{\wh{\calK}} \lVert A \rVert \cdot ( \eta_{1} + 1 ), \\
            \label{case 2: bound for |a21|}
            &\lvert \mymat{a}{2}{1} \rvert 
            \le \frac{ 2 ( 1 + \epsilon_{r} ) }{\wh{\calK}} \lVert A \rVert, \\ 
            \label{case 2: bound for |a43|}
            &\lvert \mymat{a}{4}{3} \rvert 
            \le \frac{ 2 ( 1 + \epsilon_{r} ) }{ \wh{\calK} } \lVert A \rVert \cdot ( \eta_{2} + 1 ), \\
            \label{case 2: bound for |a34|} 
            &\lvert \mymat{a}{3}{4} \rvert 
            \le \frac{ 2 ( 1 + \epsilon_{r} ) }{ \wh{\calK} } \lVert A \rVert, \\
            \label{case 2: bound for |a11-a33|^2/|a12|} 
            &\frac{ ( \mymat{a}{1}{1} - \mymat{a}{3}{3} )^{2} }{ \lvert \mymat{a}{1}{2} \rvert } 
            \le \frac{ 207( 1 + \epsilon_{r} ) }{ 16\wh{\calK} } \lVert A \rVert .
        \end{align}
    \end{itemize}
\end{proposition}

Compared with case 1, in case 2 the absolute bound for \( \lvert \mymat{a}{4}{3} \rvert \) includes an extra \( \eta_{2} + 1 \) multiplier term and \( \lvert \mymat{a}{4}{3} \rvert \) could be larger than \( \lvert \mymat{a}{1}{2} \rvert \). To avoid excessive repetition with case 1, we only consider \( \lvert \mymat{a}{4}{3} \rvert > \lvert \mymat{a}{1}{2} \rvert \) in detail. Then the results in Proposition~\ref{prop: componentwise bounds of |T| for case 2} still hold with \( \mymat{a}{1}{2} \) replaced by \( \mymat{a}{4}{3} \) in~(\ref{case 2: relative bound for |T| -- 1})--(\ref{case 2: relative bound for |T| -- 2}) and~(\ref{case 2: bound for |a11-a33|^2/|a12|}). Moreover, by (\ref{case 2: bound for |a43|}) and assumption~(\ref{case 2: assumption 2}), it is easy to show that \( \eta_{2} \ge 7 \). For a brief discussion of the situation when \( \lvert \mymat{a}{4}{3} \rvert \le \lvert \mymat{a}{1}{2} \rvert \), see Remark~\ref{remark: |a43| <= |a12| in case 2}. We continue to use the notation introduced in Subsection~\ref{subsec: case 1} with the same meaning.

\textbf{The first step of GECP.} 
Without loss of generality, we assume \( -\mymat{a}{4}{3} \) in \( ( 1, 3 ) \) is chosen as the first pivot. Then \( \wh{P}_{1} = I_{4} \) and \( \wh{Q}\trans_{1} \) swaps columns \( 1 \) and \( 3 \). Thus
\[
\wh{P}_{1} T \wh{Q}_{1}\trans = 
\begin{bNiceMatrix}[first-row, first-col, margin=4pt] 
    & \scriptstyle 3 & \scriptstyle 2 & \scriptstyle 1 & \scriptstyle 4 \\
    \scriptstyle 1 & -\mymat{a}{4}{3} & \mymat{a}{1}{2} & \mymat{a}{1}{1} - \mymat{a}{3}{3} & 0 \\
    \scriptstyle 2 & 0 & \mymat{a}{1}{1} - \mymat{a}{3}{3} & \mymat{a}{2}{1} & -\mymat{a}{4}{3} \\
    \scriptstyle 3 & \mymat{a}{1}{1} - \mymat{a}{3}{3} & 0 & -\mymat{a}{3}{4} & \mymat{a}{1}{2} \\
    \scriptstyle 4 & \mymat{a}{2}{1} & -\mymat{a}{3}{4} & 0 & \mymat{a}{1}{1} - \mymat{a}{3}{3}
\end{bNiceMatrix} .
\] 
Due to zero entries, only entries in \( (3, 1) \), \( (3, 2) \), \( (4, 1) \) and \( (4, 2) \) are updated. By Proposition~\ref{prop: componentwise bounds of |T| for case 2}, 
\begin{align*} 
    (3, 1) \colon \lvert \mymat[(1)]{\wh{t}}{3}{3} \rvert  
    &\le ( 1 + \gamma_{5} ) \cdot \big( \lvert \mymat{a}{3}{4} \rvert + ( \mymat{a}{1}{1} - \mymat{a}{3}{3})^{2} \lvert \mymat[-1]{a}{4}{3} \rvert \big)
    \le \frac{ 239( 1 + \epsilon_{r} ) }{ 16\wh{\calK} } \lVert A \rVert, \\
    (3, 2) \colon \lvert \mymat[(1)]{\wh{t}}{3}{2} \rvert
    &\le ( 1 + \gamma_{3} ) \cdot \lvert ( ( \mymat{a}{1}{1} - \mymat{a}{3}{3}) \mymat[-1]{a}{4}{3} ) \mymat{a}{1}{2} \rvert
    \le \frac{ 37 ( 1 + \epsilon_{r} ) }{ 24 \wh{\calK} } \lVert A \rVert \cdot ( \eta_{1} + 1 ), \\
    (4, 1) \colon \lvert \mymat[(1)]{\wh{t}}{4}{3} \rvert 
    &\le ( 1 + \gamma_{3} ) \cdot \lvert \mymat{a}{2}{1} ( \mymat[-1]{a}{4}{3} ( \mymat{a}{1}{1} - \mymat{a}{3}{3} ) ) \rvert 
    \le \frac{ 37 ( 1 + \epsilon_{r} ) }{ 24 \wh{\calK} } \lVert A \rVert, \\
    (4, 2) \colon \lvert \mymat[(1)]{\wh{t}}{4}{2} \rvert
    &\le ( 1 + \gamma_{3} ) \cdot \big( \lvert \mymat{a}{3}{4} \rvert + \lvert \mymat{a}{2}{1} ( \mymat[-1]{a}{4}{3} \mymat{a}{1}{2} ) \rvert \big)
    \le \frac{ 4( 1 + \epsilon_{r} ) }{ \wh{\calK} } \lVert A \rVert .
\end{align*}
Compared with case 1, now the bound for \( \lvert \mymat[(1)]{\wh{t}}{3}{2} \rvert \) in \( (3, 2) \) includes an extra \( ( \eta_{1} + 1 ) \) multiplier term. Moreover, it can be shown that \( \lvert \mymat{a}{4}{3} \rvert \) is still the largest entry in \( \lvert \wh{T}^{(1)}( 2 \tightcolon 4, 2 \tightcolon 4 ) \rvert \) by referring to~(\ref{case 2: relative bound for |T| -- 1})--(\ref{case 2: relative bound for |T| -- 2}). 

\textbf{The second step of GECP.} 
Now \( -\mymat{a}{4}{3} \) in \( (2, 4) \) is chosen as the second pivot. Thus 
\[
\wh{P}_{2}\wh{T}^{(1)}\wh{Q}_{2}\trans = 
\begin{bNiceMatrix}[first-row, first-col, margin=4pt] 
    & \scriptstyle 3 & \scriptstyle 4 & \scriptstyle 1 & \scriptstyle 2 \\
    \scriptstyle 1 & -\mymat{a}{4}{3} & 0 & \mymat{a}{1}{1} - \mymat{a}{3}{3} & \mymat{a}{1}{2} \\
    \scriptstyle 2 & 0 & -\mymat{a}{4}{3} & \mymat{a}{2}{1} & \mymat{a}{1}{1} - \mymat{a}{3}{3} \\
    \scriptstyle 3 & 0 & \mymat{a}{1}{2} & \mymat[(1)]{\wh{t}}{3}{3}  & \mymat[(1)]{\wh{t}}{3}{2} \\
    \scriptstyle 4 & 0 & \mymat{a}{1}{1} - \mymat{a}{3}{3} & \mymat[(1)]{\wh{t}}{4}{3} & \mymat[(1)]{\wh{t}}{4}{2}
\end{bNiceMatrix} ,
\]
where \( \wh{P}_{2} = I_{4} \), \( \wh{Q}_{2}\trans \) swaps columns \( 2 \) and \( 4 \).
After the update, bounds for \( \lvert \mymat[(2)]{\wh{t}}{i}{j} \rvert \) in \( (3, 1) \), \( (3, 2) \), \( (4, 1) \) and~\( (4, 2) \) are as follows: 
\begin{align*} 
    (3, 1) \colon \lvert \mymat[(2)]{\wh{t}}{3}{3} \rvert  
    & \le ( 1 + \gamma_{3} ) \cdot \big( \lvert \mymat[(1)]{\wh{t}}{3}{3} \rvert + \lvert ( \mymat{a}{1}{2} \mymat[-1]{a}{4}{3} ) \mymat{a}{2}{1} \rvert \big)
    \le \frac{ 271( 1 + \epsilon_{r} ) }{ 16\wh{\calK} } \lVert A \rVert , \\ 
    (3, 2) \colon \lvert \mymat[(2)]{\wh{t}}{3}{4} \rvert  
    & \le ( 1 + \gamma_{4} ) \cdot \big( \lvert \mymat[(1)]{\wh{t}}{3}{2} \rvert + \lvert \mymat{a}{1}{2} ( \mymat[-1]{a}{4}{3} ( \mymat{a}{1}{1} - \mymat{a}{3}{3} ) ) \rvert \big)
    \le \frac{ 37 ( 1 + \epsilon_{r} ) }{ 12 \wh{\calK} } \lVert A \rVert \cdot ( \eta_{1} + 1 ) , \\
    (4, 1) \colon \lvert \mymat[(2)]{\wh{t}}{4}{3} \rvert  
    &\le ( 1 + \gamma_{4} ) \cdot \big( \lvert \mymat[(1)]{\wh{t}}{4}{3} \rvert + \lvert ( \mymat{a}{1}{1} - \mymat{a}{3}{3} ) \mymat[-1]{a}{4}{3}\mymat{a}{2}{1} \rvert \big)
    \le \frac{ 37 ( 1 + \epsilon_{r} ) }{ 12 \wh{\calK} } \lVert A \rVert , \\
    (4, 2) \colon \lvert \mymat[(2)]{\wh{t}}{4}{4} \rvert
    &\le ( 1 + \gamma_{5} ) \cdot \big( \lvert \mymat[(1)]{\wh{t}}{4}{2} \rvert + ( \mymat{a}{1}{1} - \mymat{a}{3}{3} )^{2} \lvert \mymat[-1]{a}{4}{3} \rvert \big)
    \le \frac{ 271 ( 1 + \epsilon_{r} ) }{ 16 \wh{\calK} } \lVert A \rVert .
\end{align*}  

\textbf{The third step of GECP.} 
The bounds we obtain for \( \lvert \mymat[(2)]{\wh{t}}{i}{j} \rvert \) imply that the third largest entry of \( \lvert \wh{T}^{(2)}( 3 \tightcolon 4, 3 \tightcolon 4 ) \rvert \) is at most 
\[
\frac{ 271 ( 1 + \epsilon_{r} ) }{ 16 \wh{\calK} } \lVert A \rVert  .
\]
As explained in case 1, the same bound also applies to the update in the third step of GECP.

In summary, the bounds for \( \mymat{s}{i}{j} \) are as follows:
\begin{itemize}
    \item For \( (i, j) \notin \{ (3, 1), (3, 2), (4, 1), (4, 2) \} \), \( \mymat{s}{i}{j} = \lvert \mymat{t}{i}{j} \rvert \).

    \item For \( (i, j) \in \{ (3, 1), (4, 2) \} \), 
    \begin{align*}
        \mymat{s}{i}{j} 
        &\le \lvert \mymat{a}{3}{4} \rvert + 2 \Big( ( \mymat{a}{1}{1} - \mymat{a}{3}{3} )^{2} \lvert \mymat[-1]{a}{4}{3} \rvert + \lvert \mymat{a}{1}{2} \mymat[-1]{a}{4}{3} \mymat{a}{2}{1} \rvert + \frac{ 271( 1 + \epsilon_{r} ) } { 16 \wh{\calK} } \lVert A \rVert \Big) \\
        &\le \frac{ 263 ( 1 + \epsilon_{r} ) }{ 4 \wh{\calK} } \lVert A \rVert . 
    \end{align*}

    \item For \( \mymat{s}{3}{2} \),  
    \begin{align*}
        \mymat{s}{3}{2} 
        &\le 2 \Big( \lvert ( \mymat{a}{1}{1} - \mymat{a}{3}{3} ) \mymat[-1]{a}{4}{3} \mymat{a}{1}{2} \rvert + \frac{ 271( 1 + \epsilon_{r} ) } { 16 \wh{\calK} } \lVert A \rVert \Big) \\
        &\le \Big( \frac{271}{8} + \frac{ 37 ( \eta_{1} + 1 ) }{ 12 } \Big) \cdot  \frac{ 1 + \epsilon_{r} }{ \wh{\calK} } \lVert A \rVert . 
    \end{align*} 

    \item For \( \mymat{s}{4}{1} \),  
    \[
    \mymat{s}{4}{1} 
    \le 2\Big( \lvert \mymat{a}{2}{1} \mymat[-1]{a}{4}{3} (  \mymat{a}{1}{1} - \mymat{a}{3}{3} ) \rvert + \frac{ 271( 1 + \epsilon_{r} ) } { 16 \wh{\calK} } \lVert A \rVert \Big) \le \frac{ 887 ( 1 + \epsilon_{r} ) }{ 24 \wh{\calK} } \lVert A \rVert . 
    \]
    
\end{itemize} 

We still denote \( \wh{X}_{1}  \) by \( \wh{X} \). Remember that \( \eta_{1} \ge 8 \) and \( \eta_{2} \ge 7 \). Following the argument in Subsection~\ref{subsec: case 1}, the contribution of each \( \mymat{s}{i}{j} \) to the bound of \( \wh{\phi}( A ) \) is bounded as follows: 
\[
\begin{bNiceMatrix}[first-row, first-col, margin=4pt] 
    & \scriptstyle \wh{\sigma}_{1} & \scriptstyle \delta_{1}\wh{\sigma}_{1} & \scriptstyle \delta_{2}\wh{\sigma}_{1} & \scriptstyle \delta_{1}\delta_{2}\wh{\sigma}_{1} \\
    \scriptstyle \delta_{1}\delta_{2} & \mymat{s}{1}{1} & \mymat{s}{1}{2} & \mymat{s}{1}{3} & \mymat{s}{1}{4} \\
    \scriptstyle \delta_{2} & \mymat{s}{2}{1} & \mymat{s}{2}{2} & \mymat{s}{2}{3} & \mymat{s}{2}{4} \\
    \scriptstyle \delta_{1} & \mymat{s}{3}{1} & \mymat{s}{3}{2} & \mymat{s}{3}{3} & \mymat{s}{3}{4} \\
    \scriptstyle 1 & \mymat{s}{4}{1} & \mymat{s}{4}{2} & \mymat{s}{4}{3} & \mymat{s}{4}{4}
\end{bNiceMatrix} 
\]
\begin{itemize} 
    \item For \( ( i, j)  \in \{ ( 1, 2 ), ( 3, 4 ) \} \), \( \mymat{s}{i}{j} = \lvert \mymat{a}{1}{2} \rvert \). Then the contribution of two \( \mymat{s}{i}{j} \) is bounded by 
    \begin{equation} \label{case 2: contribution of a12}
        2\gamma_{12} \delta_{1}^{2} \delta_{2} \lvert \mymat{a}{1}{2} \rvert \cdot \wh{\calK} 
        \le \sqrt{2} ( 1 + \epsilon_{r} ) \cdot \gamma_{12} \lVert A \rVert . 
    \end{equation}

    \item For \( ( i, j ) \in \{ ( 2, 1 ), ( 4, 3 ) \} \), \( \mymat{s}{i}{j} = \lvert \mymat{a}{2}{1} \rvert \). Then the contribution of two \( \mymat{s}{i}{j} \) is bounded by 
    \begin{equation} \label{case 2: contribution of a21}
        2\gamma_{12} \delta_{2} \lvert \mymat{a}{2}{1} \rvert \cdot \wh{\calK} 
        \le \sqrt{2} ( 1 + \epsilon_{r} ) \cdot \gamma_{12} \lVert A \rVert . 
    \end{equation}

    \item For \( (i, j) \in \{ (1, 3), (2, 4) \} \), \( \mymat{s}{i}{j} = \lvert \mymat{a}{4}{3} \rvert \). Then the contribution of two \( \mymat{s}{i}{j} \) is bounded by 
    \begin{equation} \label{case 2: contribution of a43}
        2\gamma_{12} \delta_{1} \delta_{2}^{2} \lvert \mymat{a}{4}{3} \rvert \cdot \wh{\calK} 
        \le \frac{ 4( 1 + \epsilon_{r} ) }{3} \cdot \gamma_{12} \lVert A \rVert . 
    \end{equation}
    
    \item For \( (i, j) \in \{ (1, 1), (2, 2), (3, 3), (4, 4) \} \), \( \mymat{s}{i}{j} = \lvert \mymat{a}{1}{1} - \mymat{a}{3}{3} \rvert \). Then the contribution of four \( \mymat{s}{i}{j} \) is bounded by 
    \begin{align} \label{case 2: contribution of a11 - a33}
        \notag 4\gamma_{12} \delta_{1}\delta_{2} \lvert \mymat{a}{1}{1} - \mymat{a}{3}{3} \rvert \cdot \wh{\calK} 
        &\le 4\gamma_{12} \Big( \frac{ 1 + \epsilon_{r} }{ \wh{\sigma}_{1} } \delta_{1}\delta_{2} \lVert A \rVert + \delta_{1}^{2} \delta_{2} \lvert \mymat{a}{1}{2} \rvert + \delta_{1} \delta_{2}^{2} \lvert \mymat{a}{4}{3} \rvert \Big) \cdot \wh{\calK} \\ 
        &\le \frac{ ( 8 + 7\sqrt{2} ) ( 1 + \epsilon_{r} ) }{3} \cdot \gamma_{12} \lVert A \rVert .
    \end{align}

    \item  For \( (i, j) \in \{ ( 3, 1 ), ( 4, 2 ) \} \), the contribution of two \( \mymat{s}{i}{j} \) is bounded by 
    \begin{equation} \label{case 2: contribution of s31 and s42}
        2\gamma_{12} \delta_{1} \mymat{s}{i}{j} \cdot \wh{\calK} 
        \le \frac{ 263( 1 + \epsilon_{r} ) }{6} \cdot \gamma_{12} \lVert A \rVert . 
    \end{equation}

    \item The contribution of \( \mymat{s}{3}{2} \) is bounded by 
    \begin{equation} \label{case 2: contribution of s32}
        \gamma_{12} \delta_{1}^{2}\mymat{s}{3}{2} \cdot \wh{\calK} 
        \le \frac{ 493 ( 1 + \epsilon_{r} ) }{ 72 } \cdot \gamma_{12} \lVert A \rVert .
    \end{equation}

    \item The contribution of \( \mymat{s}{4}{1} \) is bounded by 
    \begin{equation} \label{case 2: contribution of s41}
        \gamma_{12} \mymat{s}{4}{1} \cdot \wh{\calK} 
        \le \frac{ 887 ( 1 + \epsilon_{r} ) }{ 24 } \cdot \gamma_{12} \lVert A \rVert .
    \end{equation}

    \item The remaining \( \mymat{s}{i}{j} = 0 \).
\end{itemize}

Summing bounds (\ref{contribution of X_2}) and (\ref{case 2: contribution of a12})--(\ref{case 2: contribution of s41}) gives 
\begin{equation} \label{smaller bound for phi(A) in case 2}
\wh{\phi}( A ) 
\le \Big( \frac{ 3299 + 156\sqrt{2} }{ 36 } ( 1 + \epsilon_{r} ) + 64\sqrt{2} \Big) \cdot \gamma_{12} \lVert A \rVert .
\end{equation}  
Comparing (\ref{bound for phi(A) in case 2}) with (\ref{smaller bound for phi(A) in case 2}) and (\ref{bound for phi(A) in case 1}), it is clear that for both cases,  
\[
\wh{\phi}( A ) 
\le ( 6272 + 3072\sqrt{2} ) ( 1 + \epsilon_{r} ) \cdot  \gamma_{12} \lVert A \rVert .
\]

\begin{remark} \label{remark: |a43| <= |a12| in case 2} 
    If \( \lvert \mymat{a}{1}{2} \rvert \ge \lvert \mymat{a}{4}{3} \rvert \), one can follow the above analysis to show that the bound (\ref{bound for phi(A) in case 2}) still holds. In particular, when \( \lvert \mymat{a}{1}{2} \rvert = \lvert \mymat{a}{4}{3} \rvert \), a somewhat tricky pivoting choice is that \( -\mymat{a}{4}{3} \) is the first pivot while \( \mymat{a}{1}{2} \) is the second, or vice versa. In this case, it is not difficult to show that \( \mymat{s}{2}{1} \), \( \mymat{s}{3}{1} \), and~\( \mymat{s}{4}{1} \) can be bounded by \( O( ( 1 + \epsilon_{r} ) \macheps \lVert A \rVert ) \). The other \( \mymat{s}{i}{j} \) either satisfy \( \mymat{s}{i}{j} = \lvert \mymat{t}{i}{j} \rvert \) or are bounded by 
    \[
    O\Big( \frac{ 1 + \epsilon_{r} }{ \wh{\calK} } \macheps \lVert A \rVert + \lvert \mymat{a}{1}{1} - \mymat{a}{3}{3} \rvert \Big) . 
    \] 
    Notice that \( \lvert \mymat{a}{1}{2} \rvert = \lvert \mymat{a}{4}{3} \rvert \) implies  
    \[
    \lvert \mymat{a}{1}{2} \rvert = \lvert \mymat{a}{4}{3} \rvert 
    \le \frac{ 2( 1 + \epsilon_{r} ) }{\wh{\calK}} \lVert A \rVert \cdot ( ( \eta_{1} \wedge \eta_{2} ) + 1  )  
    \] 
    and \( \eta_{2} \ge 7 \). Thus, for \( k = 1, 2 \) we have 
    \[
    \delta_{k} \lvert \mymat{a}{1}{1} - \mymat{a}{3}{3} \rvert 
    \le \frac{ 1 + \epsilon_{r} }{ \wh{\sigma}_{1} } \delta_{k} \lVert A \rVert + \delta_{1} \delta_{k} \lvert \mymat{a}{1}{2} \rvert + \delta_{2} \delta_{k} \lvert \mymat{a}{4}{3} \rvert 
    \le \frac{ ( 16 + \sqrt{2} )( 1 + \epsilon_{r} ) }{ 4\wh{\calK} } \lVert A \rVert .
    \]
    Using this fact, we can establish a bound of the form \( O( ( 1 + \epsilon_{r} ) \macheps \lVert A \rVert ) \) for \( \wh{\phi} ( A ) \) when  \( \lvert \mymat{a}{1}{2} \rvert = \lvert \mymat{a}{4}{3} \rvert \).  
\end{remark}

\section{Contribution to error from QR factorization} \label{sec: contribution to error from QR factorization}

The objective in this section is to show that there exist \( \Delta \wh{X} \) and \( \wt{Q} \) satisfying~(\ref{bounds for Delta X and Q}) and~(\ref{bound for contribution of QR}). Based on this, we will present a brief proof of~(\ref{bound for contribution of QR}).

Recall that the largest entry of \( \lvert \wh{X} \rvert \) is assumed to be in its first column. Let \( \wt{G}_{1} = \mygiv{-\mymat{\wh{x}}{1}{1}}{-\mymat{\wh{x}}{2}{1}} \) and \( \wh{G}_{1} \) be the computed \( \wt{G}_{1} \). By Lemma~\ref{lemma: componentwise backward error result on GR}, we obtain
\[
\sffl\Big( 
\wh{G}_{1}\trans 
\begin{bmatrix}
    -\mymat{\wh{x}}{1}{2}  \\
    -\mymat{\wh{x}}{2}{2}
\end{bmatrix} 
\Big) 
= \wt{G}_{1}\trans
\begin{bmatrix}
    -\mymat{\wh{x}}{1}{2} + \Delta \mymat{\wh{x}}{1}{2} \\
    -\mymat{\wh{x}}{2}{2} + \Delta \mymat{\wh{x}}{2}{2}
\end{bmatrix} , 
\]
where
\[ 
\lvert \Delta \mymat{\wh{x}}{1}{2} \rvert 
\le \gamma_{6} ( 5\lvert \mymat{\wh{x}}{1}{2} \rvert + 8 \wh{\sigma}_{2} ) , \quad
\lvert \Delta \mymat{\wh{x}}{2}{2} \rvert 
\le \gamma_{6} ( 5\lvert \mymat{\wh{x}}{2}{2} \rvert + 8 \wh{\sigma}_{2} ) .
\]
Therefore, the \( 2 \times 2 \) matrix 
\[
\begin{bmatrix}
    0 & \Delta \mymat{\wh{x}}{1}{2} \\
    0 & \Delta \mymat{\wh{x}}{2}{2}
\end{bmatrix}
\]
satisfies the requirement for \( \Delta \wh{X} \) in~(\ref{bounds for Delta X and Q}). It will be shown that this is the desired \( \Delta \wh{X} \). 

Let 
\[
\wh{H} 
= \begin{bmatrix}
    -\wh{X} \\
    I_{2}
\end{bmatrix}
= \begin{bmatrix}
    -\mymat{\wh{x}}{1}{1} & -\mymat{\wh{x}}{1}{2} \\
    -\mymat{\wh{x}}{2}{1} & -\mymat{\wh{x}}{2}{2} \\
    1 & 0 \\
    0 & 1
\end{bmatrix}, \dspace  
\wt{H} 
= \begin{bmatrix}
    -\mymat{\wh{x}}{1}{1} & -\mymat{\wh{x}}{1}{2} + \Delta \mymat{\wh{x}}{1}{2} \\
    -\mymat{\wh{x}}{2}{1} & -\mymat{\wh{x}}{2}{2} + \Delta \mymat{\wh{x}}{2}{2} \\
    1 & 0 \\
    0 & 1
\end{bmatrix} .
\]
Consider the Givens QR of \( \wh{H} \) in finite precision arithmetic and of \( \wt{H} \) in exact arithmetic, with the ordering of Givens update specified in (\ref{order of rotations if col1>col2}). For \( 1 \le i \le 4 \), let \( \wh{G}_{i} \) and \( \wt{G}_{i} \) be the \( i \)th computed and exact Givens rotations applied to \( \wh{H} \) and \( \wt{H} \), respectively.
The entries of \( \wh{G}_{i} \) and \( \wt{G}_{i} \) are denoted by 
\[ 
\wh{G}_{i} = \begin{bmatrix}
    \wh{c}_{i} & \wh{s}_{i} \\
   -\wh{s}_{i} & \wh{c}_{i}  
\end{bmatrix}, \dspace
\wt{G}_{i} = \begin{bmatrix}
    \wt{c}_{i} & \wt{s}_{i} \\
   -\wt{s}_{i} & \wt{c}_{i}  
\end{bmatrix} . 
\]

\textbf{The first Givens update.} 
Rotations \( \wh{G}_{1} \) and \( \wt{G}_{1} \) are specified at the beginning of this section. Denote 
\[
\mymat[(1)]{\wh{x}}{1}{1} 
= \sffl\big( \sqrt{ \mymat[2]{\wh{x}}{1}{1} + \mymat[2]{\wh{x}}{2}{2} } \big), \dspace
\mymat[(1)]{\wt{x}}{1}{1} 
= \sqrt{ \mymat[2]{\wh{x}}{1}{1} + \mymat[2]{\wh{x}}{2}{2} } 
\]
and 
\[
\begin{bmatrix}
    \mymat[(1)]{\wt{x}}{1}{2} \\
    \mymat[(1)]{\wt{x}}{2}{2}
\end{bmatrix} 
\coloneqq \wt{G}_{1}\trans
\begin{bmatrix}
    -\mymat{\wh{x}}{1}{2} + \Delta \mymat{\wh{x}}{1}{2} \\
    -\mymat{\wh{x}}{2}{2} + \Delta \mymat{\wh{x}}{2}{2}
\end{bmatrix} . 
\]
Applying Lemma~\ref{lemma: perturbations of GR} with \( \gamma_{k} = 0 \), it follows that
\[
\lvert \wh{c}_{1} - \wt{c}_{1} \rvert 
\le \gamma_{4} \lvert \wt{c}_{1} \rvert, \dspace 
\lvert \wh{s}_{1} - \wt{s}_{1} \rvert 
\le \gamma_{4} \lvert \wt{s}_{1} \rvert, \dspace 
\lvert \mymat[(1)]{\wh{x}}{1}{1} - \mymat[(1)]{\wt{x}}{1}{1} \rvert 
\le \gamma_{3} \lvert \mymat[(1)]{\wt{x}}{1}{1} \rvert .
\] 
 
\textbf{The second Givens update.}
The rotation \( \wh{G}_{2} \) is the computed \( \mygiv{\mymat[(1)]{\wh{x}}{1}{1}}{1} \) and \( \wt{G}_{2} = \mygiv{\mymat[(1)]{\wt{x}}{1}{1}}{1} \). Applying Lemma~\ref{lemma: perturbations of GR} with \( \gamma_{k} = \gamma_{3} \), we have 
\[
\lvert \wh{c}_{2} - \wt{c}_{2} \rvert \le \gamma_{10} \lvert \wt{c}_{2} \rvert , \dspace 
\lvert \wh{s}_{2} - \wt{s}_{2} \rvert \le \gamma_{10} \lvert \wt{s}_{2} \rvert . 
\]
Notice that 
\[
\sffl\Big( \wh{G}_{2}\trans 
\begin{bmatrix}
    \mymat[(1)]{\wt{x}}{1}{2} \\
    0
\end{bmatrix} \Big) 
= \sffl\Big( 
\begin{bmatrix}
    \times \\
    \wh{s}_{2} \mymat[(1)]{\wt{x}}{1}{2}
\end{bmatrix} \Big), \dspace 
\wt{G}_{2}\trans \begin{bmatrix}
    \mymat[(1)]{\wt{x}}{1}{2} \\
    0
\end{bmatrix} 
= \begin{bmatrix}
    \times \\
    \wt{s}_{2} \mymat[(1)]{\wt{x}}{1}{2}
\end{bmatrix} .
\]
Denote \( \mymat[(2)]{\wh{x}}{3}{2} = \sffl (\wh{s}_{2} \mymat[(1)]{\wt{x}}{1}{2}) \) and \( \mymat[(2)]{\wt{x}}{3}{2} = \wt{s}_{2} \mymat[(1)]{\wt{x}}{1}{2} \). By \( \lvert \wh{s}_{2} - \wt{s}_{2} \rvert \le \gamma_{10} \lvert \wt{s}_{2} \rvert \), we have
\[
\lvert \mymat[(2)]{\wh{x}}{3}{2} - \mymat[(2)]{\wt{x}}{3}{2} \rvert 
\le \gamma_{11} \lvert \mymat[(2)]{\wt{x}}{3}{2} \rvert .
\]

\textbf{The third Givens update.} The rotation \( \wh{G}_{3} \) is the computed \( \mygiv{\mymat[(1)]{\wt{x}}{2}{2}}{1} \) and \( \wt{G}_{3} = \mygiv{\mymat[(1)]{\wt{x}}{2}{2}}{1} \). Denote 
\[ 
\mymat[(3)]{\wh{x}}{2}{2}
= \sffl\big( \tsqrt{ (\mymat[(1)]{\wt{x}}{2}{2})^{2} + 1 } \big), \dspace 
\mymat[(3)]{\wt{x}}{2}{2} 
= \tsqrt{ (\mymat[(1)]{\wt{x}}{2}{2})^{2} + 1 } . 
\]
Applying Lemma~\ref{lemma: perturbations of GR} with \( \gamma_{k} = 0 \), we obtain 
\[
\lvert \wh{c}_{3} - \wt{c}_{3} \rvert \le \gamma_{4} \lvert \wt{c}_{3} \rvert, \dspace
\lvert \wh{s}_{3} - \wt{s}_{3} \rvert \le \gamma_{4} \lvert \wt{s}_{3} \rvert, \dspace 
\lvert \mymat[(3)]{\wt{x}}{2}{2} - \mymat[(3)]{\wh{x}}{2}{2} \rvert 
\le \gamma_{3} \lvert \mymat[(3)]{\wt{x}}{2}{2} \rvert .  
\]

\textbf{The fourth Givens update.} The rotation \( \wh{G}_{4} \) is the computed \( \mygiv{\mymat[(3)]{\wh{x}}{2}{2}}{\mymat[(2)]{\wh{x}}{3}{2}} \) and \( \wt{G}_{4} = \mygiv{\mymat[(3)]{\wt{x}}{2}{2}}{\mymat[(2)]{\wt{x}}{3}{2}} \). Applying Lemma~\ref{lemma: perturbations of GR} with \( \gamma_{k} = \gamma_{11} \), we have
\[ 
\lvert \wh{c}_{4} - \wt{c}_{4} \rvert \le \gamma_{26} \lvert \wt{c}_{4} \rvert, \dspace  
\lvert \wh{s}_{4} - \wt{s}_{4} \rvert \le \gamma_{26} \lvert \wt{s}_{4} \rvert .
\]

Now we reconstruct \( \wh{G}_{i} \) and \( \wt{G}_{i} \) to their corresponding \( 4 \times 4 \) representations. From the above arguments, it can be seen that 
\begin{align*}
    &\lVert \wh{G}_{1} - \wt{G}_{1} \rVert 
    \le \sqrt{2} \, \gamma_{4}, \dspace 
    \lVert \wh{G}_{2} - \wt{G}_{2} \rVert 
    \le \sqrt{2} \, \gamma_{11}, \\ 
    &\lVert \wh{G}_{3} - \wt{G}_{3} \rVert 
    \le \sqrt{2} \, \gamma_{4}, \dspace 
    \lVert \wh{G}_{4} - \wt{G}_{4} \rVert 
    \le \sqrt{2} \, \gamma_{26} .
\end{align*}
By \cite[Lemma 3.6 and Problem 3.10]{hig02}, \( \wh{Q} = \sffl( \wh{G}_{1} \wh{G}_{2} \wh{G}_{3} \wh{G}_{4} ) \) and the orthogonal matrix \( \wt{Q} = \wt{G}_{1} \wt{G}_{2} \wt{G}_{3} \wt{G}_{4} \) satisfy 
\begin{align*}
    \lVert \wt{Q} - \wh{Q} \rVert
    &\le \lVert \wt{Q} - \wh{G}_{1} \wh{G}_{2} \wh{G}_{3} \wh{G}_{4} \rVert + \lVert \wh{G}_{1} \wh{G}_{2} \wh{G}_{3} \wh{G}_{4} - \wh{Q} \rVert \\
    &\le ( 1 + \sqrt{2} \gamma_{4} )( 1 + \sqrt{2} \gamma_{4} )( 1 + \sqrt{2} \gamma_{11} )( 1 + \sqrt{2} \gamma_{26} ) - 1 \\
    &\quad + ( 64 \macheps + O( \macheps^{2} ) ) \lVert \wh{G}_{1} \rVert \lVert \wh{G}_{2} \rVert \lVert \wh{G}_{3} \rVert \lVert \wh{G}_{4} \rVert .
\end{align*}
Ignoring the terms that can be bounded by \( O( \macheps^{2} ) \), we obtain  
\[
\lVert \wt{Q} - \wh{Q} \rVert 
\le ( 64 + 45\sqrt{2} ) \macheps 
\]
Now we have verified that the obtained \( \Delta \wh{X} \) and \( \wt{Q} \) satisfy (\ref{bounds for Delta X and Q}).  

By~(\ref{bounds for Delta X and Q}), we have 
\[
\lvert w \rvert 
= \lvert T  \Delta \wh{x} \rvert 
\le 5 \gamma_{6} \lvert T \rvert \lvert \wh{x} \rvert + 8 \wh{\sigma}_{2} \gamma_{6} \lvert T \rvert [ 0, 0, 1, 1 ]\trans . 
\]
For convenience, we denote \( 5 \gamma_{6} \lvert T \rvert \lvert \wh{x} \rvert \) and \( 8 \wh{\sigma}_{2} \gamma_{6} \lvert T \rvert [ 0, 0, 1, 1 ]\trans \) by \( w_{1} \) and \( w_{2} \), respectively. By Lemma~\ref{lemma: norm inequalities}, the contribution of \( w_{2} \) to \( \wh{\psi}( A ) \) is bounded by  
\begin{equation} \label{contribution of w2} 
    \frac{ 2 \lVert w_{2} \rVert_{\infty} }{ \sqrt{ 1 + 2\wh{\sigma}_{2}^{2} } } 
    \le \frac{ 16 \wh{\sigma}_{2} }{ \sqrt{1 + 2 \wh{\sigma}_{2}^{2}} } \cdot \gamma_{6} \lVert T \rVert_{\infty} 
    \le 32 \gamma_{6} \lVert A \rVert . 
\end{equation} 

As for \( w_{1} \), recall that the bounds for \( \wh{\phi}(A) \) in Section~\ref{sec: contribution to error of solving linear system} are built on two estimates for \( r \): the componentwise bound~(\ref{componentwise bound for residual}) (including (\ref{case 1: contribution of a12})--(\ref{case 1: contribution of s23 and s41}), (\ref{case 2: contribution of a12})--(\ref{case 2: contribution of s41})), and the normwise bound~(\ref{infinite norm bound for residual}) (including (\ref{bound 0 for contribution of residual}), (\ref{bound for phi(A) in case 1}), (\ref{contribution of X_2}) and (\ref{bound for phi(A) in case 2})). Moreover, the bounds for \( \mymat{s}{i}{j} \) derived in Subsections~\ref{subsec: case 1} and~\ref{subsec: case 2}, under their respective assumptions, carry over directly to \( |\mymat{t}{i}{j}| \). By replacing \( \gamma_{12} \) or \( 16\rho_{4} \gamma_{12} \) with \( 5\gamma_{6} \) in corresponding situations, it follows that the contribution of \( w_{1} \) to \( \wh{\psi}( A ) \) is bounded by  
\begin{equation} \label{contribution of w1}
    ( 490 + 240\sqrt{2} ) ( 1 + \epsilon_{r} ) \cdot \gamma_{6} \lVert A \rVert . 
\end{equation}
This is from replacing \( 16 \rho_{4} \gamma_{12} \) in~(\ref{bound for phi(A) in case 2}) with \( 5\gamma_{6} \).

Summing the bounds~(\ref{contribution of w2}) and~(\ref{contribution of w1}), we conclude that for \( \wh{\kappa} > 4 \), \[
\wh{\psi}( A ) 
\le \big( ( 490 + 240\sqrt{2} ) ( 1 + \epsilon_{r} ) + 32 \big) \cdot \gamma_{6} \lVert A \rVert . 
\]

\begin{remark} 
    From the above analysis, it can be seen that the ordering of the first two Givens updates is critical for ensuring the existence of a \( \Delta \wh{X} \) satisfying (\ref{bounds for Delta X and Q}). In contrast, different orderings of Givens updates in another column merely result in different bound of the form \( O( \macheps ) \) for \( \lVert \wt{Q} - \wh{Q} \rVert \). In particular, the update orderings specified in~(\ref{order of rotations if col1>col2}) and~(\ref{order of rotations if col1<=col2}) are chosen to keep the constant factor in the bound for \( \lVert \wt{Q} - \wh{Q} \rVert \) as small as we can. 
\end{remark}

\section{Conclusion} \label{sec: conclusion}
In this paper, we propose the pivoted Givens QR as a substitute for the Householder QR employed in LAPACK subroutine to swap adjacent \( 2 \times 2 \) diagonal blocks in the real Schur form. Our error analysis suggests that, regarding the error measured by \( \tol \), this modified direct swapping algorithm can be used confidently in practice. We provide empirical evidence to illustrate the advantage of the proposed algorithm relative to existing methods from \cite{Bai1993,Camps19}. 

We conclude this paper with a few suggestions for future work. First, our experiments show that for all three algorithms we test, the relative errors of eigenvalues may still be large even when \( \tol \) is on the order of~\( \macheps \lVert A \rVert_{\max} \). These difficulties may be due to the fragility of eigenvalues: a normwise small perturbation of \( A \) can cause a large change in the eigenvalues of \( \mymat{A}{i}{i} \). Therefore, it is a challenging task to develop swapping algorithms with small relative errors of eigenvalues. Second, compared with the real Schur form, reordering eigenvalues in the generalized real Schur form poses a greater challenge in terms of numerical stability. Thus, designing a provably stable algorithm for the generalized real Schur form remains an important subject. 

\section*{Acknowledgments}
The authors thank Weiguo Gao and Linxi Sai for helpful discussions.

\section*{Declaration of AI usage}
The authors used DeepSeek-V4-Flash exclusively for language refinement and proofreading purposes.
The authors are responsible for the scientific contributions, including algorithm development and error analysis.

\bibliographystyle{siamplain}
\bibliography{ref}

@article{Bai1993,
    title = {On swapping diagonal blocks in real {S}chur form},
    journal = {Linear Algebra Appl.},
    volume = {186},
    pages = {75--95},
    year = {1993},
    author = {Bai, Zhaojun and Demmel, James W.},
    doi    = {10.1016/0024-3795(93)90286-W},
}

@article{Braman02,
    author = {Braman, Karen and Byers, Ralph and Mathias, Roy},
    title = {The Multishift {QR} Algorithm. {Part~II}: {A}ggressive Early Deflation},
    journal = {SIAM J. Matrix Anal. Appl.},
    volume = {23},
    number = {4},
    pages = {948--973},
    year = {2002},
    doi = {10.1137/S0895479801384585},
}

@article{Camps19,
    title = {Swapping 2 {$\times$} 2 blocks in the {S}chur and generalized {S}chur form},
    journal = {J. Comput. Appl. Math.},
    volume = {373},
    pages = {112274},
    year = {2020},
    author = {Camps, Daan and Mastronardi, Nicola and Vandebril, Raf and {Van Dooren}, Paul},
    doi = {10.1016/j.cam.2019.05.022},
}

@article{Cao81,
    author = {Cao, Z. and Zhang, F.},
    title = {Direct methods for ordering eigenvalues of a real matrix},
    journal = {Chinese Uniu. J. Comput. Math.},
    volume = {1},
    pages = {27--36},
    year = {1981},
    note = {(In Chinese)},
}

@article{Cryer1968,
    author  = {Cryer, C. W.},
    title   = {Pivot size in {G}aussian elimination},
    journal = {Numer. Math.},
    year    = {1968},
    volume  = {12},
    number  = {4},
    pages   = {335--345},
    doi     = {10.1007/BF02162514},
}

@article{Dooren81,
    author = {Van Dooren, P.},
    title = {A Generalized Eigenvalue Approach for Solving {R}iccati Equations},
    year = {1981},
    volume = {2},
    number = {2},
    journal = {SIAM J. Sci. Statist. Comput.},
    pages = {121--135},
    doi = {10.1137/0902010},
}

@article{Dooren82,
    author = {Van Dooren, P.},
    title = {Algorithm 590: {DSUBSP} and {EXCHQZ}: {F}ORTRAN Subroutines for Computing Deflating Subspaces with Specified Spectrum},
    year = {1982},
    volume = {8},
    number = {4},
    journal = {ACM Trans. Math. Softw.},
    pages = {376--382},
    doi = {10.1145/356012.356017},
}

@article{Fokkema98,
    author = {Fokkema, Diederik R. and Sleijpen, Gerard L. G. and Van der Vorst, Henk A.},
    title = {{J}acobi--{D}avidson Style {QR} and {QZ} Algorithms for the Reduction of Matrix Pencils},
    journal = {SIAM J. Sci. Comput.},
    volume = {20},
    number = {1},
    pages = {94--125},
    year = {1998},
    doi = {10.1137/S1064827596300073},
}

@book{Golub13,
    author = {Golub, Gene H. and Van Loan, Charles F.},
    title = {Matrix Computations},
    publisher = {Johns Hopkins University Press},
    year = {2013},
    address = {Philadelphia, PA, USA},
    edition   = {4th},
}

@article{hig93,
    author = {Higham, Nicholas J.},
    title = {Perturbation theory and backward error for {$AX-XB=C$}},
    journal = {BIT},
    year = {1993},
    volume = {33},
    pages = {124--136},
    doi   = {10.1007/BF01990348},
}

@book{hig02,
    author = {Higham, Nicholas J.},
    title = {Accuracy and Stability of Numerical Algorithms},
    publisher = {SIAM},
    address   = {Philadelphia, PA, USA},
    year = {2002},
    edition = {Second},
    doi = {10.1137/1.9780898718027},
}

@Inbook{Kagstrom93,
    author = {K{\aa}gstr{\"o}m, B.},
    title = {A Direct Method for Reordering Eigenvalues in the Generalized Real Schur form of a Regular Matrix Pair {$(A, B)$}},
    bookTitle= {Linear Algebra for Large Scale and Real-Time Applications},
    year = {1993},
    publisher = {Springer Netherlands},
    address = {Dordrecht},
    pages = {195--218},
    doi = {10.1007/978-94-015-8196-7\_11},
}

@article{Kagstrom96,
    author = {K{\aa}gstr{\"o}m, B. and Poromaa, P.},
    title = {Computing eigenspaces with specified eigenvalues of a regular matrix pair {$(A, B)$} and condition estimation: theory, algorithms and software},
    journal = {Numer. Algorithms},
    year = {1996},
    volume = {12},
    pages = {369--407},
    doi = {10.1007/BF02142813},
}

@article{Kragstrom07,
    author = { K{\r a}gstr{\"o}m, Bo and  Kressner, Daniel},
    title = {Multishift Variants of the {QZ} Algorithm with Aggressive Early Deflation},
    journal = {SIAM J. Matrix Anal. Appl.},
    volume = {29},
    number = {1},
    pages = {199--227},
    year = {2007},
    doi = {10.1137/05064521X},
}

@article{Laub79,
    author = {Laub, Alan J.},
    title = {A {S}chur method for solving algebraic {R}iccati equations},
    year = {1979},
    volume = {24},
    number = {6},
    journal = {IEEE Trans.\ Autom.\ Control},
    pages = {913--921},
    doi = {10.1109/TAC.1979.1102178},
}

@techreport{Parlett88,
    author    = {Parlett, Beresford N. and Ng, Kwee Chye},
    title     = {Development of an Accurate Algorithm for {EXP($Bt$)}, {Part~I}: {P}rograms to Swap Diagonal Blocks},
    institution = {Univ.\ of California, Berkeley},
    year      = {1988},
    type      = {Internal Report},
    number    = {CPAM-294},
}

@mastersthesis{Sai2022,
  author  = {Sai, Linxi},
  title   = {Swapping Eigenvalues in Real {S}chur Form and Its Rounding Error Analysis},
  school  = {Fudan University},
  year    = {2022},
  note    = {(In Chinese)},
}

@article{Steel21,
    author = {Steel, Thijs and Camps, Daan and Meerbergen, Karl and Vandebril, Raf},
    title = {A Multishift, Multipole Rational {QZ} Method with Aggressive Early Deflation},
    journal = {SIAM J. Matrix Anal. Appl.},
    volume = {42},
    number = {2},
    pages = {753--774},
    year = {2021},
    doi = {10.1137/19M1249631},
}

@article{Stewart76,
    author = {Stewart, G. W.},
    title = {Algorithm 506: {HQR}3 and {EXCHNG}: {F}ortran Subroutines for Calculating and Ordering the Eigenvalues of a Real Upper {H}essenberg Matrix [{F}2]},
    year = {1976},
    volume = {2},
    number = {3},
    journal = {ACM Trans. Math. Softw.},
    pages = {275--280},
    doi = {10.1145/355694.355700},
}

@article{Stewart02,
    author = {Stewart, G. W.},
    title = {A {K}rylov--{S}chur Algorithm for Large Eigenproblems},
    journal = {SIAM J. Matrix Anal. Appl.},
    volume = {23},
    number = {3},
    pages = {601--614},
    year = {2002},
    doi = {10.1137/S0895479800371529}
}

@article{Stewart02_2,
    author = {Stewart, G. W.},
    title = {Addendum to ``{A} {K}rylov--{S}chur Algorithm for Large Eigenproblems''},
    journal = {SIAM J. Matrix Anal. Appl.},
    volume = {24},
    number = {2},
    pages = {599--601},
    year = {2002},
    doi = {10.1137/S0895479802403150},
}

@article{Wilkinson61,
    author = {Wilkinson, J. H.},
    title = {Error analysis of direct methods of matrix inversion},
    year = {1961},
    volume = {8},
    number = {3},
    journal = {J. ACM},
    pages = {281--330},
    doi = {10.1145/321075.321076},
}

@article{Wedin72,
    author = {Wedin, Per-\AA{ke}.},
    title = {Perturbation bounds in connection with singular value decomposition},
    journal = {BIT},
    year = {1972},
    volume = {12},
    number = {1},
    pages = {99--111},
    doi   = {10.1007/BF01932678},
}

\begin{appendices} 
\section{Appendix} \label{appendix}

\subsection{Proof of Lemma~\ref{lemma: perturbations of GR}}
\label{app: Proof of Lemma perturbations of GR}
\begin{proof}
If \( v_{1} = 0 \), then \( \wt{v}_{1} = 0 \) and the result is trivial. Assume \( v_{1} \ne 0 \). Let 
\[ 
\mygiv{\wt{v}_{1}}{\wt{v}_{2}}
 = \begin{bmatrix}
     \wt{c} & \wt{s} \\
    -\wt{s} & \wt{c}
 \end{bmatrix} .
\]
By~\cite[Lemma 19.7]{hig02}, we have
\[
\lvert \wh{c} - \wt{c} \rvert \le \gamma_{4} \lvert \wt{c} \rvert, \dspace 
\lvert \wh{s} - \wt{s} \rvert \le \gamma_{4} \lvert \wt{s} \rvert .
\]
Since  
\[
\Big\lvert \frac{ \wt{c} }{ c } - 1 \Big\rvert 
= \Big\lvert \frac{ \wt{v}_{1} }{ v_{1} } \cdot \sqrt{ \frac{ v_{1}^{2} + v_{2}^{2} }{ \wt{v}_{1}^{2} + \wt{v}_{2}^{2} } } - 1 \Big\rvert 
\le \gamma_{2k} ,
\] 
it immediately follows that \( \lvert \wh{c} - c \rvert \le \gamma_{2k+4} \lvert c \rvert \). In an analogous way, we can obtain \( \lvert \wh{s} - s \rvert \le \gamma_{2k+4} \lvert s \rvert \).

The second statement follows directly from the standard floating point arithmetic error model and we omit details.
\end{proof} 

\subsection{Proof of Lemma~\ref{lemma: componentwise backward error result on GR}} \label{app: proof of Lemma componentwise backward error result on GR}  
\begin{proof} 
Without loss of generality, we assume \( \lvert \mymat{x}{1}{1} \rvert \) is the largest entry in \( \lvert X \rvert \). If \( \mymat{x}{2}{1} = 0 \), then \( G = I_{2} \) and the result is trivial. Suppose \( \mymat{x}{2}{1} \ne 0 \). Denote  
\[
G = \begin{bmatrix}
    c & s \\
    -s & c
\end{bmatrix}, \dspace 
\wh{G} = \begin{bmatrix}
    \wh{c} & \wh{s} \\
    -\wh{s} & \wh{c} 
\end{bmatrix} .
\]
By~\cite[Lemma 19.7]{hig02}, we have
\[
\lvert \wh{c} - c \rvert \le \gamma_{4} \lvert c \rvert, \dspace 
\lvert \wh{s} - s \rvert \le \gamma_{4} \lvert s \rvert .
\]
Suppose 
\begin{equation} \label{eq: Givens backward error}
    \sffl \Big(
    \begin{bmatrix}
        \wh{c} & -\wh{s} \\
        \wh{s} &  \wh{c} 
    \end{bmatrix}
    \begin{bmatrix}
        \mymat{x}{1}{2}  \\
        \mymat{x}{2}{2} 
    \end{bmatrix}
    \Big) = 
    \begin{bmatrix}
        c & -s \\
        s &  c
    \end{bmatrix}
    \begin{bmatrix}
    \mymat{x}{1}{2} + \Delta \mymat{x}{1}{2} \\
    \mymat{x}{2}{2} + \Delta \mymat{x}{2}{2}
    \end{bmatrix} .
\end{equation}
Then there exist \( \nu_{6} \), \( \nu_{6}^{\prime} \) such that \( \lvert \nu_{6} \rvert \), \( \lvert \nu_{6}^{\prime} \rvert \le \gamma_{6} \) and 
\[
\sffl( \wh{c}\mymat{x}{1}{2} - \wh{s}\mymat{x}{2}{2} ) = c\mymat{x}{1}{2} - s\mymat{x}{2}{2} + \nu_{6}c\mymat{x}{1}{2} + \nu_{6}^{\prime}s\mymat{x}{2}{2} 
\coloneqq c\mymat{x}{1}{2} - s\mymat{x}{2}{2} + \Delta_{1} .
\]
Similarly, we have 
\[
\sffl( \wh{s}\mymat{x}{1}{2} + \wh{c}\mymat{x}{2}{2} ) = s\mymat{x}{1}{2} + c\mymat{x}{2}{2} + \Delta_{2} ,
\]
where \( \Delta_{2} = \mu_{6}s \mymat{x}{1}{2} + \mu_{6}^{\prime}c \mymat{x}{2}{2} \) with \( \lvert \mu_{6} \rvert \), \( \lvert \mu_{6}^{\prime} \rvert \le \gamma_{6} \). Comparing both sides of (\ref{eq: Givens backward error}), we obtain 
\[
\Delta \mymat{x}{1}{2} = c \Delta_{1} + s \Delta_{2}, \dspace 
\Delta \mymat{x}{2}{2} = -s \Delta_{1} + c \Delta_{2} .
\]
Substituting the expressions of \( \Delta_{1} \) and \( \Delta_{2} \) into the above equations, we obtain 
\begin{align*}
    & \lvert \Delta \mymat{x}{1}{2} \rvert
    = \lvert \nu_{6} c^{2} \mymat{x}{1}{2} + \nu_{6}^{\prime}cs \mymat{x}{2}{2} + ( \mu_{6} s^{2} \mymat{x}{1}{2} + \mu_{6}^{\prime}cs \mymat{x}{2}{2} ) \rvert
    \le \gamma_{6} ( \lvert \mymat{x}{1}{2} \rvert + 2 \lvert cs \rvert \cdot \lvert \mymat{x}{2}{2} \rvert ) , \\
    & \lvert \Delta \mymat{x}{2}{2} \rvert 
    = \lvert \nu_{6} cs \mymat{x}{1}{2} + \nu_{6}^{\prime}s^{2} \mymat{x}{2}{2} - ( \mu_{6} cs \mymat{x}{1}{2} + \mu_{6}^{\prime} c^{2} \mymat{x}{2}{2} ) \rvert
    \le \gamma_{6} ( \lvert \mymat{x}{2}{2} \rvert + 2 \lvert cs \rvert \cdot \lvert \mymat{x}{1}{2} \rvert ) .
\end{align*}

If \( \sigma_{2}( X ) = 0 \), the \( 2 \times 2 \) matrix \( X \) satisfies \( \mymat{x}{1}{1} \mymat{x}{2}{2} = \mymat{x}{1}{2} \mymat{x}{2}{1} \). Then from  
\[
cs = -\frac{ \mymat{x}{1}{1}\mymat{x}{2}{1} }{ \mymat[2]{x}{1}{1} + \mymat[2]{x}{2}{1} } 
= -\frac{ \mymat{x}{1}{2}\mymat{x}{2}{2} } { \mymat[2]{x}{1}{2} + \mymat[2]{x}{2}{2} } ,
\]
it follows that \( \lvert cs \rvert \cdot \lvert \mymat{x}{2}{2} \rvert \le \lvert \mymat{x}{1}{2} \rvert \), \( \lvert cs \rvert \cdot \lvert \mymat{x}{1}{2} \rvert \le \lvert \mymat{x}{2}{2} \rvert \), and the result holds. 

Now assume \( \sigma_{2}( X ) > 0 \). By a simple calculation, we obtain  
\[
X^{-1} = \frac{ 1 }{ \mymat{x}{1}{1}\mymat{x}{2}{2} - \mymat{x}{1}{2}\mymat{x}{2}{1} }
\begin{bmatrix}
    \mymat{x}{2}{2} & -\mymat{x}{1}{2} \\
   -\mymat{x}{2}{1} &  \mymat{x}{1}{1}
\end{bmatrix} .
\]
From \( \lVert X^{-1} \rVert \le 2 \lVert X^{-1} \rVert_{\max} \) it follows 
\[
\sigma_{2}( X ) = \frac{ 1 }{ \lVert X^{-1} \rVert } 
\ge \frac{ \lvert \mymat{x}{1}{1}\mymat{x}{2}{2} - \mymat{x}{1}{2}\mymat{x}{2}{1} \rvert }{ 2\lvert \mymat{x}{1}{1} \rvert }
= \frac{1}{2} \Big\lvert \mymat{x}{2}{2} - \frac{ \mymat{x}{1}{2}\mymat{x}{2}{1} }{ \mymat{x}{1}{1} } \Big\rvert . 
\]
Next we discuss the cases \( \lvert \mymat{x}{2}{2} \rvert > \lvert \mymat{x}{1}{2} \rvert \) and \( \lvert \mymat{x}{2}{2} \rvert \le \lvert \mymat{x}{1}{2} \rvert \), respectively.
    \begin{itemize}
        \item When \( \lvert \mymat{x}{2}{2} \rvert > \lvert \mymat{x}{1}{2} \rvert \), we only need to focus on \( \lvert cs \rvert \cdot \lvert \mymat{x}{2}{2} \rvert \). If \( \lvert \mymat{x}{2}{2} \rvert \le 2 \lvert \mymat{x}{1}{2} \rvert \), then from \( 2\lvert cs \rvert \le 1 \) it follows \( \lvert cs \rvert \cdot \lvert \mymat{x}{2}{2} \rvert \le \lvert \mymat{x}{1}{2} \rvert \) and the result holds. Now assume \( \lvert \mymat{x}{2}{2} \rvert > 2 \lvert \mymat{x}{1}{2} \rvert \). Since \( \lvert \mymat{x}{1}{1} \rvert \) is the largest entry of \( \lvert X \rvert \), 
        \[
        \Big\lvert \mymat{x}{2}{2} - \frac{ \mymat{x}{1}{2}\mymat{x}{2}{1} }{ \mymat{x}{1}{1} } \Big\rvert 
        > \frac{ \lvert \mymat{x}{2}{2} \rvert  }{ 2 } 
        \dspace \Longrightarrow \dspace 
        \lvert \mymat{x}{2}{2} \rvert 
        < 4\sigma_{2}( X ) , 
        \]
        which implies the result.
        \item When \( \lvert \mymat{x}{2}{2} \rvert \le \lvert \mymat{x}{1}{2} \rvert \), we only need to focus on \( \lvert cs \rvert \cdot \lvert \mymat{x}{1}{2} \rvert \). If \( \lvert \mymat{x}{1}{2}\mymat{x}{2}{1} \rvert \le 2 \lvert \mymat{x}{1}{1}\mymat{x}{2}{2} \rvert \), then
        \[
        \lvert cs \rvert \cdot \lvert \mymat{x}{1}{2} \rvert 
        \le \Big\lvert \frac{ \mymat{x}{1}{1}\mymat{x}{2}{1} }{ \mymat[2]{x}{1}{1} + \mymat[2]{x}{2}{1} } \cdot \frac{ 2\mymat{x}{1}{1}\mymat{x}{2}{2} }{ \mymat{x}{2}{1} } \Big\rvert
        \le 2 \lvert \mymat{x}{2}{2} \rvert . 
        \]
        If \( \lvert \mymat{x}{1}{2}\mymat{x}{2}{1} \rvert > 2 \lvert \mymat{x}{1}{1}\mymat{x}{2}{2} \rvert \), then
        \[
        \Big\lvert \mymat{x}{2}{2} - \frac{ \mymat{x}{1}{2}\mymat{x}{2}{1} }{ \mymat{x}{1}{1} } \Big\rvert 
        > \frac{ \lvert \mymat{x}{1}{2}\mymat{x}{2}{1} \rvert }{ 2 \lvert \mymat{x}{1}{1} \rvert }
        \dspace \Longrightarrow \dspace 
        \lvert \mymat{x}{1}{2} \rvert 
        < \frac{ 4 \lvert \mymat{x}{1}{1} \rvert }{ \lvert \mymat{x}{2}{1} \rvert } \sigma_{2}( X ) .
        \]
        Thus,
        \[
        \lvert cs \rvert \cdot \lvert \mymat{x}{1}{2} \rvert 
        \le \Big\lvert \frac{ \mymat{x}{1}{1}\mymat{x}{2}{1} }{ \mymat[2]{x}{1}{1} + \mymat[2]{x}{2}{1} } \cdot \frac{ 4\mymat{x}{1}{1} }{ \mymat{x}{2}{1} } \Big\rvert \sigma_{2}( X ) 
        \le 4 \sigma_{2}( X ) 
        \]
        and the result still holds. 
\end{itemize} 
Hence in either case we have verified the result.
\end{proof}

\subsection{Proof of Lemma~\ref{lemma: stability of SVD}} \label{app: proof of stability of SVD}
\begin{proof}
We need the following theorem to complete the proof.
\begin{theorem}[Wedin~\cite{Wedin72}] \label{theo: Wedin SVD}
    Suppose \( X \), \( \Delta X \in \mathbb{R}^{n \times n} \). Let 
    \[
    X = [ U_{1},~ U_{2} ] \begin{bmatrix}
    \Sigma_{1} & 0 \\
    0          & \Sigma_{2}
    \end{bmatrix} \begin{bmatrix}
    V_{1}\trans \\
    V_{2}\trans
    \end{bmatrix}, \dspace
    \wt{X} = [ \wt{U}_{1},~\wt{U}_{2} ] \begin{bmatrix}
    \wt{\Sigma}_{1} & 0 \\
                 0  & \wt{\Sigma}_{2}
    \end{bmatrix} \begin{bmatrix}
    \wt{V}_{1}\trans \\
    \wt{V}_{2}\trans
    \end{bmatrix}
    \]
    be the \textup{SVD} of \( X \) and \( \wt{X} = X + \Delta X \), where \( \Sigma_{1} \), \( \wt{\Sigma}_{1} \in \mathbb{R}^{r \times r} \), \( U_{1} \), \( \wt{U}_{1} \in \mathbb{R}^{n \times r} \), \( V_{1} \), \( \wt{V}_{1} \in \mathbb{R}^{n \times r} \). Let \( \wt{\sigma}_{r} \) be the smallest diagonal entry of \( \wt{\Sigma}_{1} \) and \( \sigma_{r+1} \) be the largest diagonal entry of \( \Sigma_{2} \). If there exist \( \alpha \ge 0 \), \( \delta > 0 \) such that 
    \( \wt{\sigma}_{r} \ge \alpha + \delta \) and \( \sigma_{r+1} \le \alpha \),
    then set \( \epsilon = \lVert ( \Delta X ) \wt{V}_{1} \rVert \vee \lVert ( \Delta X )\trans\wt{U}_{1} \rVert \) and we have
    \[
    \lVert \sin \theta \big( \wt{U}_{1}, U_{1} \big) \rVert 
    \le \frac{ \epsilon }{ \delta }, \dspace
    \lVert \sin \theta \big( \wt{V}_{1}, V_{1} \big) \rVert 
    \le \frac{ \epsilon }{ \delta }.
    \]
\end{theorem}

Now we continue the proof of Lemma~\ref{lemma: stability of SVD}. The assumption \( \lVert \Delta X \rVert \le c \lVert X \rVert \) implies  
\[
\lVert \wt{X} \rVert
\ge \lVert X \rVert - \lVert \Delta X \rVert
\ge ( 1 - c ) \lVert X \rVert .
\]
Hence, we can take \( \alpha = \sigma_{2}( X ) \) and
\[
\delta 
= \lVert \wt{X} \rVert - \alpha  
\ge \big( 1 - c - \kappa^{-1}( X ) \big) \lVert X \rVert .
\]
Since \( \epsilon = \lVert ( \Delta X ) \wt{v}_{1} \rVert \vee \lVert ( \Delta X )\trans \wt{u}_{1} \rVert \le \lVert \Delta X \rVert \), by Theorem~\ref{theo: Wedin SVD} we obtain 
\begin{align*}
\lVert \wt{u}_{1} - ( \wt{u}_{1}\trans u_{1} ) u_{1} \rVert 
&= \lvert \sin \theta ( \wt{u}_{1}, u_{1} ) \rvert  
\le \frac{ \lVert \Delta X \rVert }{ \big( 1 - c - \kappa^{-1}( X ) \big) \lVert X \rVert } \\
&\le \frac{ c }{ 1 - c - \kappa^{-1}( X ) } .
\end{align*}
Notice that the first equality follows from 
\begin{align*}
\lVert \wt{u}_{1} - ( \wt{u}_{1}\trans u_{1} ) u_{1} \rVert^{2} 
&= \langle \wt{u}_{1} - ( \wt{u}_{1}\trans u_{1} ) u_{1}, \wt{u}_{1} - ( \wt{u}_{1}\trans u_{1} ) u_{1} \rangle \\ 
&= 1 - 2\big( \cos \theta( \wt{u}_{1}, u_{1} ) \big)^{2} + \big( \cos \theta( \wt{u}_{1}, u_{1} ) \big)^{2} 
= \big( \sin \theta( \wt{u}_{1}, u_{1} ) \big)^{2} .
\end{align*} 

Without loss of generality, we assume the angle between \( \wt{u}_{1} \) and \( u_{1}  \) is acute. Otherwise, replace \( u_{1} \) with \( -u_{1} \) in the following argument. Then 
\[
\lVert u_{1} - ( \wt{u}_{1}\trans u_{1} ) u_{1} \rVert 
= 1 - \cos \theta( \wt{u}_{1}, u_{1} ) 
\le \big( \sin \theta( \wt{u}_{1}, u_{1} ) \big)^{2} 
\le \frac{ c^{2} }{ \big( 1 - c - \kappa^{-1}( X ) \big)^{2} } .
\]
Using the above results, we obtain that for \( i = 1, 2 \),
\[
\lvert \mymat{\wt{u}}{i}{1} \rvert - \lvert \mymat{u}{i}{1} \rvert 
\le \lVert u_{1} - \wt{u}_{1} \rVert \le \lVert u_{1} - ( \wt{u}_{1}\trans u_{1} ) u_{1} \rVert + \lVert \wt{u}_{1} - ( \wt{u}_{1}\trans u_{1} ) u_{1} \rVert 
\le d + d^{2} , 
\]
where \( d = c / ( 1 - c - \kappa^{-1}( X ) ) \). 
Our result for \( \mymat{\wt{u}}{i}{j} \) then follows from the fact that \( \lvert \mymat{u}{1}{1} \rvert = \lvert \mymat{u}{2}{2} \rvert \), \( \lvert \mymat{u}{2}{1} \rvert = \lvert \mymat{u}{1}{2} \rvert \) for any \( 2 \times 2 \) orthogonal matrix \( U \). Analogously, we can prove the result for \( \mymat{\wt{v}}{i}{j} \). 
\end{proof}

\subsection{Proof of Proposition~\ref{prop: simplify A}}
\begin{proof}
    This proof stems from the following fact. Let \( P_{1} \), \( P_{2} \in \mathbb{R}^{n \times n} \) be two permutations. Then for \( T \in \mathbb{R}^{n \times n} \) and \( b \in \mathbb{R}^{n} \), the computed solutions \( \wh{x} \) and \( \wh{y} \) obtained by GECP to linear systems \( Tx = b \) and \( P_{1}TP_{2}\trans y = P_{1}b \) satisfy \( \wh{y} = P_{2} \wh{x} \) and \( r_{\wh{y}} = P_{1}( TP_{2}\trans \wh{y} - b ) = P_{1} r_{\wh{x}} \).

    Notice that the diagonal entries of \( \mymat{A}{1}{1} \) are the same. Thus, for \( B \) we have 
    \begin{align*}
        T_{B} = I_{2} \otimes \mymat[\top]{A}{1}{1} - \mymat[\top]{A}{2}{2} \otimes I_{2} 
        &= ( I_{2} \otimes P )( I_{2} \otimes \mymat{A}{1}{1} - \mymat[\top]{A}{2}{2} \otimes I_{2} ) ( I_{2} \otimes P ) \\
        &= ( I_{2} \otimes P ) T_{A} ( I_{2} \otimes P )
    \end{align*}
    and \( \sfvec( P\mymat{A}{1}{2} ) = ( I_{2} \otimes P ) \sfvec( \mymat{A}{1}{2} ) \), which imply \( \wh{Y} = P\wh{X} \) and \( R_{\wh{Y}} = P R_{\wh{X}} \). Then it is not hard to see the results for \( B \) hold. In a similar manner, we can prove the results for \( C \). 
    
    For \( D \), a direct computation yields 
    \begin{align*}
        T_{D} &= -I_{2} \otimes \mymat[\top]{A}{2}{2} + \mymat{A}{1}{1} \otimes I_{2} \\
        &= \begin{bmatrix}
            \mymat{a}{1}{1} - \mymat{a}{3}{3} & -\mymat{a}{4}{3} & \mymat{a}{1}{2} & 0 \\
            -\mymat{a}{3}{4} & \mymat{a}{1}{1} - \mymat{a}{3}{3} & 0 & \mymat{a}{1}{2} \\
            \mymat{a}{2}{1} & 0 & \mymat{a}{1}{1} - \mymat{a}{3}{3} & -\mymat{a}{4}{3} \\
            0 & \mymat{a}{2}{1} & -\mymat{a}{3}{4} & \mymat{a}{1}{1} - \mymat{a}{3}{3}  
        \end{bmatrix}
    \end{align*}
    Thus, \( T_{D} = QTQ \), where \( Q \) is the permutation obtained by swapping rows 2 and 3 of \( I_{4} \). Combining with \( \sfvec( \mymat[\top]{A}{1}{2} ) = Q\sfvec( \mymat{A}{1}{2} ) \), we have \( \wh{W} = \wh{X}\trans \) and \( R_{\wh{W}} = R_{\wh{X}}\trans \). Then the assertions for \( D \) hold.
\end{proof}

\subsection{Proof of Proposition~\ref{prop: estimate for tau}}
\begin{proof} 
    From (\ref{eq: reduced residual equation}), it follows that 
    \begin{equation} \label{eq: componentwise equations for hatA12}
        \begin{bmatrix}
            \wh{\sigma}_{1}\mymat{\wh{a}}{1}{1} - \wh{\sigma}_{1}\mymat{\wh{a}}{3}{3} & \wh{\sigma}_{2}\mymat{\wh{a}}{1}{2} - \wh{\sigma}_{1}\mymat{\wh{a}}{3}{4}  \\
            \wh{\sigma}_{1}\mymat{\wh{a}}{2}{1} - \wh{\sigma}_{2}\mymat{\wh{a}}{4}{3} & \wh{\sigma}_{2}\mymat{\wh{a}}{2}{2} - \wh{\sigma}_{2}\mymat{\wh{a}}{4}{4}  
        \end{bmatrix}
    =   \begin{bmatrix}
            \mymat{\wh{a}}{1}{3} & \mymat{\wh{a}}{1}{4} \\
            \mymat{\wh{a}}{2}{3} & \mymat{\wh{a}}{2}{4}
        \end{bmatrix}. 
    \end{equation}
    Since \( \lVert R \rVert = \epsilon_{r} \lVert A \rVert \), for all \( \mymat{\wh{a}}{i}{j} \) in \( \mymat{\wh{A}}{1}{2} \), we have
    \[
    \lvert \mymat{\wh{a}}{i}{j} \rvert 
    \le \lVert \mymat{\wh{A}}{1}{2} \rVert 
    \le \lVert \mymat{A}{1}{2} \rVert + \lVert R \rVert  
    \le ( 1 + \epsilon_{r} ) \lVert A \rVert .
    \]
    This fact will be used multiple times in the subsequent proof.

    If \( \lvert \mymat{\wh{a}}{1}{2} \rvert > \lvert \mymat{\wh{a}}{2}{1} \rvert \), notice that 
    \[
    \frac{ \lvert \mymat{\wh{a}}{2}{3} \rvert }{ \lvert \mymat{\wh{a}}{1}{2} \rvert } 
    = \frac{ \lvert \wh{\sigma}_{1} \mymat{\wh{a}}{2}{1} - \wh{\sigma}_{2} \mymat{\wh{a}}{4}{3} \rvert }{ \lvert \mymat{\wh{a}}{1}{2} \rvert } 
    \ge \Big( \frac{ \lvert \mymat{\wh{a}}{2}{1} \rvert }{ \lvert \mymat{\wh{a}}{1}{2} \rvert } - \frac{1}{ \wh{\kappa} } \cdot \frac{ \lvert \mymat{\wh{a}}{4}{3} \rvert }{ \lvert \mymat{\wh{a}}{1}{2} \rvert } \Big) \wh{\sigma}_{1} 
    \]
    By~(\ref{order realation between off-diagonal entries}) and~(\ref{def: eta}), we obtain 
    \begin{equation} \label{hat a23 / hat a12}
        \frac{ \lvert \mymat{\wh{a}}{2}{3} \rvert }{ \lvert \mymat{\wh{a}}{1}{2} \rvert } \ge \Big( \frac{1}{\eta_{1}} - \frac{1}{\wh{\kappa}} \Big) \wh{\sigma}_{1} .
    \end{equation} 
    When \( \wh{\kappa} \le 2\eta_{1} \), it follows that 
    \[
    \frac{ \lvert \mymat{\wh{a}}{1}{2} \rvert }{ \lVert A \rVert } 
    \le \frac{ \lvert \mymat{\wh{a}}{1}{2} \rvert }{ \lVert \mymat{A}{1}{1} \rVert }
    = \frac{ \lvert \mymat{\wh{a}}{1}{2} \rvert }{ \lVert \mymat{\wh{A}}{1}{1} \rVert } 
    \le 1 \le \frac{2}{ \wh{\kappa} } \eta_{1} . 
    \]
    When \( \wh{\kappa} > 2\eta_{1} \), by~(\ref{hat a23 / hat a12}) we have 
    \[
    \frac{ \lvert \mymat{\wh{a}}{1}{2} \rvert }{ \lVert A \rVert } 
    \le ( 1 + \epsilon_{r} ) \frac{ \lvert \mymat{\wh{a}}{1}{2} \rvert }{ \lvert \mymat{\wh{a}}{2}{3} \rvert } 
    \le \Big( \frac{1}{\eta_{1}} - \frac{1}{\wh{\kappa}} \Big)^{-1} \cdot \frac{ 1 + \epsilon_{r} }{ \wh{\sigma}_{1} } 
    < \frac{ 2( 1 + \epsilon_{r} ) }{ \wh{\sigma}_{1} } \eta_{1} .  
    \]
    Thus, for \( \tau = \lvert \mymat{\wh{a}}{1}{2} \rvert \), in either case we have 
    \[
    \lvert \mymat{\wh{a}}{1}{2} \rvert \le \frac{ 2( 1 + \epsilon_{r} ) }{ \wh{\calK} } \lVert A \rVert \eta_{1} .
    \]

    If \( \lvert \mymat{\wh{a}}{1}{2} \rvert \le \lvert \mymat{\wh{a}}{2}{1} \rvert \), notice that by~(\ref{order realation between off-diagonal entries}) and \( \wh{\kappa} > 4 \), 
    \[
    \frac{ \lvert \mymat{\wh{a}}{2}{3} \rvert }{ \lvert \mymat{\wh{a}}{2}{1} \rvert } 
    = \frac{ \lvert \mymat{\wh{a}}{2}{1}\wh{\sigma}_{1} - \mymat{\wh{a}}{4}{3}\wh{\sigma}_{2} \rvert }{ \lvert \mymat{\wh{a}}{2}{1} \rvert} 
    \ge \Big( 1 - \frac{1}{ \wh{\kappa} } \Big)\wh{\sigma}_{1}
    \ge \frac{ 3 }{ 4 } \wh{\sigma}_{1} .
    \]
    Thus, for \( \tau = \lvert \mymat{\wh{a}}{2}{1} \rvert \), we have 
    \[
    \lvert \mymat{\wh{a}}{2}{1} \rvert \le \frac{ 4( 1 + \epsilon_{r} ) }{ 3\wh{\sigma}_{1} } \lVert A \rVert .
    \]
    Now we finish the proof.  
\end{proof}

\subsection{Proof of Proposition~\ref{prop: estimate for smallest elements in U and V}} \label{app: proof of estimate for smallest elements in U and V} 

\begin{proof}
Let us take \( \wh{U} \) as an example. When \( \eta_{1} = 1 \), the result is trivial. Now suppose \( \eta_{1} > 1 \). Without loss of generality, assume \( \wh{U} \) takes the form 
\[
\wh{U} = \begin{bmatrix}
    \wh{c} & \wh{s} \\
   -\wh{s} & \wh{c} 
\end{bmatrix}, \dspace \wh{c}^{2} + \wh{s}^{2} = 1 .
\] 
Then 
\begin{equation} \label{eq: U^T A11 U}
    \mymat{\wh{A}}{1}{1} = \wh{U}\trans \mymat{A}{1}{1} \wh{U} 
    = \begin{bmatrix}
    \mymat{a}{1}{1} - \wh{c}\wh{s} (\mymat{a}{1}{2} + \mymat{a}{2}{1}) & \wh{c}^{2}\mymat{a}{1}{2} - \wh{s}^{2}\mymat{a}{2}{1} \\
    \wh{c}^{2}\mymat{a}{2}{1} - \wh{s}^{2}\mymat{a}{1}{2} & \mymat{a}{1}{1} + \wh{c}\wh{s}(\mymat{a}{1}{2} + \mymat{a}{2}{1})
    \end{bmatrix} .
\end{equation}
Thus \( \eta_{1} \ne 1 \) implies \( \wh{c}^{2} \ne \wh{s}^{2} \). By (\ref{eq: U^T A11 U}), it can be seen that 
\[
\mymat{\wh{a}}{1}{2} = \wh{c}^{2}\mymat{a}{1}{2} - \wh{s}^{2} \mymat{a}{2}{1}, \dspace  
\mymat{\wh{a}}{2}{1} = \wh{c}^{2}\mymat{a}{2}{1} - \wh{s}^{2} \mymat{a}{1}{2} .
\]
By some algebra, it yields 
 \[
\mymat{a}{1}{2} = \frac{ \wh{c}^{2}\mymat{\wh{a}}{1}{2} + \wh{s}^{2}\mymat{\wh{a}}{2}{1} } { \wh{c}^{2} - \wh{s}^{2} }, \dspace
\mymat{a}{2}{1} = \frac{ \wh{c}^{2}\mymat{\wh{a}}{2}{1} + \wh{s}^{2}\mymat{\wh{a}}{1}{2} } { \wh{c}^{2} - \wh{s}^{2} } .
\]
Since \( \mymat{a}{1}{2}\mymat{a}{2}{1} < 0 \), we have 
\[
( \wh{c}^{2}\mymat{\wh{a}}{1}{2} + \wh{s}^{2}\mymat{\wh{a}}{2}{1} ) ( \wh{c}^{2}\mymat{\wh{a}}{2}{1} + \wh{s}^{2}\mymat{\wh{a}}{1}{2} ) < 0 .
\]
Substituting \( \mymat{\wh{a}}{2}{1} = -\mymat{\wh{a}}{1}{2}\eta_{1} \) or \( \mymat{\wh{a}}{1}{2} = -\mymat{\wh{a}}{2}{1}\eta_{1} \) gives
\[
\lvert \wh{s} \rvert > \sqrt{\eta_{1}} \, \lvert \wh{c} \rvert \dspace \text{or} \dspace 
\lvert \wh{c} \rvert > \sqrt{\eta_{1}} \, \lvert \wh{s} \rvert .
\]
Combining with \( \wh{c}^{2} + \wh{s}^{2} = 1 \), we obtain the result. 
\end{proof}

\subsection{Proof of Proposition~\ref{prop: bound for diagonal entries of |T|}} \label{sec : proof of bound for diagonal entries of |T|}
\begin{proof} 
Since \( \trace( \mymat{A}{1}{1} ) = \trace( \mymat{\wh{A}}{1}{1} ) \) and \( \trace( \mymat{A}{2}{2} ) = \trace( \mymat{\wh{A}}{2}{2} ) \), we have 
\[
\lvert \mymat{a}{1}{1} - \mymat{a}{3}{3} \rvert 
= \frac{ \lvert ( \mymat{\wh{a}}{1}{1} + \mymat{\wh{a}}{2}{2} ) - ( \mymat{\wh{a}}{3}{3} + \mymat{\wh{a}}{4}{4} ) \rvert }{ 2 } 
\le \lvert \mymat{\wh{a}}{1}{1} - \mymat{\wh{a}}{3}{3} \rvert + \frac{ \lvert \mymat{\wh{a}}{1}{1} - \mymat{\wh{a}}{2}{2} \rvert }{ 2 } + \frac{ \lvert \mymat{\wh{a}}{3}{3} - \mymat{\wh{a}}{4}{4} \rvert }{ 2 } .
\]
For \( \lvert \mymat{\wh{a}}{1}{1} - \mymat{\wh{a}}{3}{3} \rvert \), recall (\ref{eq: componentwise equations for hatA12}) implies \( \mymat{\wh{a}}{1}{3} = ( \mymat{\wh{a}}{1}{1} - \mymat{\wh{a}}{3}{3} ) \wh{\sigma}_{1} \). Thus, 
\[
\frac{ \lvert \mymat{\wh{a}}{1}{1} - \mymat{\wh{a}}{3}{3} \rvert }{ \lVert A \rVert }  
\le ( 1 + \epsilon_{r} ) \frac{ \lvert \mymat{\wh{a}}{1}{1} - \mymat{\wh{a}}{3}{3} \rvert }{ \lvert \mymat{\wh{a}}{1}{3} \rvert }
= \frac{ 1 + \epsilon_{r} }{ \wh{\sigma}_{1} } .
\]
For \( \lvert \mymat{\wh{a}}{1}{1} - \mymat{\wh{a}}{2}{2} \rvert \), by (\ref{eq: U^T A11 U}) we have 
\[
\lvert \mymat{\wh{a}}{1}{1} - \mymat{\wh{a}}{2}{2} \rvert 
= 2 \lvert \wh{c}\wh{s} \rvert \cdot \lvert \mymat{a}{1}{2} + \mymat{a}{2}{1} \rvert 
\le 2\delta_{1} \lvert \mymat{a}{1}{2} \rvert . 
\]
In a similar manner, it can be shown that  
\[
\lvert \mymat{ \wh{a}}{3}{3} - \mymat{\wh{a}}{4}{4} \rvert 
\le 2\delta_{2} \lvert \mymat{a}{4}{3} \rvert . 
\]
Then we complete the proof of this proposition. 
\end{proof}

\subsection{Proof of Proposition~\ref{prop: componentwise bounds of |T| for case 1}} \label{app: proof of componentwise bounds of |T| for case 1} 

\begin{proof}

    \textbf{Bounds for \( \lvert \mymat{a}{1}{2} \rvert \) and \( \lvert \mymat{a}{2}{1} \rvert \).} By Propositions~\ref{prop: estimate for tau} and~\ref{prop: relation of entries between A11 and hatA11}~(4), we have  
    \begin{align} \label{case 1: upper bounds of a12 and a21}
        \notag&\lvert \mymat{a}{2}{1} \rvert \le \lvert \mymat{\wh{a}}{2}{1} \rvert
        = \eta_{1}^{-1}  \lvert \mymat{\wh{a}}{1}{2} \rvert
        < \frac{ 2( 1 + \epsilon_{r} ) }{ \wh{\calK} } \lVert A \rVert , \\
        &\lvert \mymat{a}{1}{2} \rvert 
        \le \lvert \mymat{\wh{a}}{1}{2} \rvert + \lvert \mymat{\wh{a}}{2}{1} \rvert
        < \frac{ 2( 1 + \epsilon_{r} ) }{ \wh{\calK} } \lVert A \rVert \cdot ( \eta_{1} + 1 ) .
    \end{align}
    Moreover, \( 3\lvert \mymat{a}{2}{1} \rvert \le \lvert \mymat{a}{1}{2} \rvert \) is a direct consequence of Proposition~\ref{prop: relation of entries between A11 and hatA11}~(2) and assumption~(\ref{case 1: assumption 1}).

    \textbf{Bounds for \( \lvert \mymat{a}{3}{4} \rvert \) and \( \lvert \mymat{a}{4}{3} \rvert \).} 
    If \( \wh{\sigma}_{1} \lvert \mymat{\wh{a}}{3}{4} \rvert \le 2\wh{\sigma}_{2} \lvert \mymat{\wh{a}}{1}{2} \rvert \), then 
    \[
    \frac{ \lvert \mymat{a}{4}{3} \rvert }{ \lVert A \rVert } 
    \le \frac{ \lvert \mymat{\wh{a}}{3}{4} \rvert + \lvert \mymat{\wh{a}}{4}{3} \rvert }{ \lVert A \rVert } 
    \le 2 \frac{ \lvert \mymat{\wh{a}}{3}{4} \rvert }{ \lvert \mymat{\wh{a}}{1}{2} \rvert } 
    \le \frac{4}{ \wh{\kappa} } . 
    \]
    The first inequality is from Proposition~\ref{prop: relation of entries between A11 and hatA11}~(4).

If \( \wh{\sigma}_{1} \lvert \mymat{\wh{a}}{3}{4} \rvert > 2\wh{\sigma}_{2} \lvert \mymat{\wh{a}}{1}{2} \rvert \), then 
\[
\frac{ \lvert \mymat{\wh{a}}{3}{4} \rvert }{ \lvert \mymat{\wh{a}}{1}{4} \rvert }
= \frac{ \lvert \mymat{\wh{a}}{3}{4} \rvert }{ \lvert \wh{\sigma}_{2}\mymat{\wh{a}}{1}{2} - \wh{\sigma}_{1}\mymat{\wh{a}}{3}{4} \rvert } 
< \frac{2}{ \wh{\sigma}_{1} } .
\]
Thus, 
\[
\frac{ \lvert \mymat{a}{4}{3} \rvert }{ \lVert A \rVert } 
\le 2\frac{ \lvert \mymat{\wh{a}}{3}{4} \rvert }{ \lVert A \rVert } 
\le 2( 1 + \epsilon_{r} ) \frac{ \lvert \mymat{\wh{a}}{3}{4} \rvert }{ \lvert \mymat{\wh{a}}{1}{4} \rvert }
< \frac{ 4( 1 + \epsilon_{r} ) }{ \wh{\sigma}_{1} } .
\]
Hence in either case, we have
\[ 
\lvert \mymat{a}{3}{4} \rvert \le \lvert \mymat{a}{4}{3} \rvert 
< \frac{ 4( 1 + \epsilon_{r} ) }{ \wh{\calK} } \lVert A \rVert .
\]
Then the relative bound for \( \lvert \mymat{a}{3}{4} \rvert \) and \( \lvert \mymat{a}{4}{3} \rvert \) in (\ref{case 1: relative bounds for entries of |T| -- 1}) follows naturally from assumption~(\ref{case 1: assumption 2}).  

\textbf{Bounds for \( \lvert \mymat{a}{1}{1} - \mymat{a}{3}{3} \rvert \).} By Propositions~\ref{prop: estimate for smallest elements in U and V} and~\ref{prop: bound for diagonal entries of |T|}, 
    \begin{align} \label{case 1: upper bound for |a11-a33|}
        \lvert \mymat{a}{1}{1} - \mymat{a}{3}{3} \rvert
        &\le \frac{ 1 + \epsilon_{r} }{ \wh{\sigma}_{1} } \lVert A \rVert + \delta_{1} \lvert \mymat{a}{1}{2} \rvert + \delta_{2} \lvert \mymat{a}{4}{3} \rvert \\
        \notag &\le \frac{ 1 + \epsilon_{r} }{ \wh{\sigma}_{1} } \lVert A \rVert + \delta_{1} \lvert \mymat{a}{1}{2} \rvert + \frac{ \sqrt{2} }{2} \lvert \mymat{a}{4}{3} \rvert \\
        \notag &< \frac{ ( 1 + 2\sqrt{2} ) ( 1 + \epsilon_{r} ) }{ \wh{\calK} } \lVert A \rVert + \delta_{1} \lvert \mymat{a}{1}{2} \rvert 
        \le \frac{3}{4} \lvert \mymat{a}{1}{2} \rvert . 
    \end{align}
    Here assumptions~(\ref{case 1: assumption 1}) and~(\ref{case 1: assumption 2}) are used in the last inequality. Moreover, by the basic inequality \( ( a + b )^{2} \le 2 ( a^{2} + b^{2} ) \),
    \begin{align*}
        \frac{ \lvert \mymat{a}{1}{1} - \mymat{a}{3}{3} \rvert^{2} }{ \lvert \mymat{a}{1}{2} \rvert }
        &\le \frac{2}{ \lvert \mymat{a}{1}{2} \rvert }\Big( \frac{ ( 1 + 2\sqrt{2} )^{2} ( 1 + \epsilon_{r} )^{2} }{ \wh{\calK}^{2} } \lVert A \rVert^{2} + \delta_{1}^{2} \lvert \mymat{a}{1}{2} \rvert^{2} \Big) \\
        &\le \frac{ ( 1 + 2\sqrt{2} ) ( 1 + \epsilon_{r} ) }{ 2\wh{\calK} } \lVert A \rVert + 2\delta_{1}^{2} \lvert \mymat{a}{1}{2} \rvert 
        \le \frac{ ( 9 + 2\sqrt{2} ) ( 1 + \epsilon_{r} ) }{     2\wh{\calK} } \lVert A \rVert . 
    \end{align*}
    Here assumption~(\ref{case 1: assumption 2}) is used in the second inequality, Proposition~\ref{prop: estimate for smallest elements in U and V} and (\ref{case 1: upper bounds of a12 and a21}) are used in the last inequality. 

    \textbf{Proof of (\ref{bound for phi(A) in case 1}).}
    If either (\ref{case 1: assumption 1}) or (\ref{case 1: assumption 2}) fails, then by (\ref{case 1: upper bounds of a12 and a21}), 
    \[
    \lvert \mymat{a}{1}{2} \rvert 
    \le \frac{ ( 4 + 8\sqrt{2} ) ( 1 + \epsilon_{r} ) }{ \wh{\calK} } \lVert A \rVert .
    \]
    In this case, by (\ref{case 1: upper bound for |a11-a33|}), 
    \begin{align*}
        \lvert \mymat{a}{1}{1} - \mymat{a}{3}{3} \rvert 
        &\le \frac{ 1 + \epsilon_{r} }{ \wh{\sigma}_{1} } \lVert A \rVert + \frac{ \sqrt{2} }{2} \cdot \frac{ ( 4 + 8\sqrt{2} ) ( 1 + \epsilon_{r} ) }{ \wh{\calK} } \lVert A \rVert + \frac{ \sqrt{2} }{2} \cdot \frac{ 4( 1 + \epsilon_{r} ) }{ \wh{\calK} } \lVert A \rVert \\
        &\le \frac{ ( 9 + 4\sqrt{2} ) ( 1 + \epsilon_{r} ) }{ \wh{\calK} } \lVert A \rVert .
    \end{align*} 
    Combining with the bound of \( \lvert \mymat{a}{4}{3} \rvert \), we obtain  
    \[
    \lVert T \rVert_{\infty} 
    = \lvert \mymat{a}{1}{2} \rvert + \lvert \mymat{a}{4}{3} \rvert + \lvert \mymat{a}{1}{1} - \mymat{a}{3}{3} \rvert
    \le \frac{ ( 17 + 12\sqrt{2} ) ( 1 + \epsilon_{r} ) }{ \wh{K} } \lVert A \rVert . 
    \]
    Mimicking (\ref{bound 0 for contribution of residual}), it can be shown that 
    \[
    \wh{\phi}( A ) 
    \le 32 \rho_{4} \gamma_{12} \lVert T \rVert_{\infty} \cdot \wh{\calK} 
    \le ( 2176 + 1536\sqrt{2} ) ( 1 + \epsilon_{r} ) \cdot \gamma_{12} \lVert A \rVert . 
    \]
\end{proof}

\subsection{Proof of Proposition~\ref{prop: componentwise bounds of |T| for case 2}} \label{app: proof of componentwise bounds of |T| for case 2} 

\begin{proof}
    Absolute bounds for \( \lvert \mymat{a}{1}{2} \rvert \) and \( \lvert \mymat{a}{2}{1} \rvert \) are identical to those in case~1. The relative bound for \( \lvert \mymat{a}{2}{1} \rvert \) is a direct consequence of Proposition~\ref{prop: relation of entries between A11 and hatA11}~(2) and assumption~(\ref{case 2: assumption 1}). 

    \textbf{Bounds for \( \lvert \mymat{a}{3}{4} \rvert \) and \( \lvert \mymat{a}{4}{3} \rvert \).} 
    If \( \wh{\sigma}_{1} \lvert \mymat{\wh{a}}{3}{4} \rvert \le 2 \wh{\sigma}_{2} \lvert \mymat{\wh{a}}{1}{2} \rvert \), then  
    \[
    \frac{ \lvert \mymat{a}{4}{3} \rvert }{ \lVert A \rVert } 
    \le \frac{ \lvert \mymat{\wh{a}}{3}{4} \rvert + \lvert \mymat{\wh{a}}{4}{3} \rvert }{ \lvert \mymat{\wh{a}}{1}{2} \rvert }
    = ( \eta_{2} + 1 ) \frac{ \lvert \mymat{\wh{a}}{3}{4} \rvert }{ \lvert \mymat{\wh{a}}{1}{2} \rvert } 
    \le \frac{ 2( \eta_{2} + 1 ) }{ \wh{\kappa} } . 
    \]
    
    If \( \wh{\sigma}_{1} \lvert \mymat{\wh{a}}{3}{4} \rvert > 2\wh{\sigma}_{2} \lvert \mymat{\wh{a}}{1}{2} \rvert \), then
    \[
    \lvert \mymat{\wh{a}}{1}{4} \rvert 
    = \lvert \wh{\sigma}_{2} \mymat{\wh{a}}{1}{2} - \wh{\sigma}_{1} \mymat{\wh{a}}{3}{4} \rvert 
    > \frac{ \wh{\sigma}_{1} }{2} \lvert \mymat{\wh{a}}{3}{4} \rvert 
    = \frac{ \wh{\sigma}_{1} }{ 2\eta_{2} } \lvert \mymat{\wh{a}}{4}{3} \rvert .
    \]
    Arguing as in Appendix~\ref{app: proof of componentwise bounds of |T| for case 1}, we have
    \[
    \frac{ \lvert \mymat{a}{4}{3} \rvert }{ \lVert A \rVert } 
    \le ( 1 + \epsilon_{r} ) \frac{ \lvert \mymat{\wh{a}}{3}{4} \rvert + \lvert \mymat{\wh{a}}{4}{3} \rvert }{ \lvert \mymat{\wh{a}}{1}{4} \rvert } 
    < \frac{ 2( 1 + \epsilon_{r} ) }{ \wh{\sigma}_{1} } \cdot ( \eta_{2} + 1 ) .
    \]
    Thus in either case, we have 
    \begin{align*} 
        \lvert \mymat{a}{4}{3} \rvert 
        &\le \frac{ 2( 1 + \epsilon_{r} ) }{ \wh{\calK} } \lVert A \rVert \cdot ( \eta_{2} + 1 ), \\
        \lvert \mymat{a}{3}{4} \rvert 
        &\le \lvert \mymat{\wh{a}}{3}{4} \rvert 
        \le \frac{ 2( 1 + \epsilon_{r} ) }{ \wh{\calK} } \lVert A \rVert .
    \end{align*}

    To prove the relative bound for \( \lvert \mymat{a}{4}{3} \rvert \), we need to exploit assumption~(\ref{case 2: assumption 2}). By (\ref{case 2: bound for |a43|}) and assumption~(\ref{case 2: assumption 2}), 
    \[
    \frac{ \lvert \mymat{a}{4}{3} \rvert }{ \lvert \mymat{a}{1}{2} \rvert } 
    \le \frac{ \eta_{2} + 1 }{ 8 } . 
    \]
    On the other hand, by (\ref{order realation between off-diagonal entries}) and  Proposition~\ref{prop: relation of entries between A11 and hatA11}~(4), 
    \[
    \frac{ \lvert \mymat{a}{4}{3} \rvert }{ \lvert \mymat{a}{1}{2} \rvert } 
    \le \frac{ \lvert \mymat{\wh{a}}{1}{2} \rvert }{ ( 1 + \eta_{2}^{-1} ) \lvert \mymat{\wh{a}}{4}{3} \rvert } 
    \le \frac{ \eta_{2} + 1 }{ \eta_{2} } . 
    \] 
    Comparing two bounds, we obtain 
    \[
    \frac{ \lvert \mymat{a}{4}{3} \rvert }{ \lvert \mymat{a}{1}{2} \rvert } 
    \le \Big( \frac{ \eta_{2} + 1 }{ \eta_{2} } \wedge \frac{ \eta_{2} + 1 }{8} \Big) 
    \le \frac{9}{8} . 
    \]

    \textbf{Bounds for \( \lvert a_{11} - a_{33} \rvert \).} 
    By Proposition~\ref{prop: bound for diagonal entries of |T|}, we have 
    \begin{align*} 
        \lvert \mymat{a}{1}{1} - \mymat{a}{3}{3} \rvert 
        &\le \frac{ 1 + \epsilon_{r} }{ \wh{\sigma}_{1} } \lVert A \rVert + \delta_{1} \lvert \mymat{a}{1}{2} \rvert + \delta_{2} \lvert \mymat{a}{4}{3} \rvert \\ 
        &\le \frac{1}{16} \lvert \mymat{a}{1}{2} \rvert + \frac{1}{3} \lvert \mymat{a}{1}{2} \rvert + \frac{1}{ \sqrt{1 + \eta_{2}} } \cdot \Big( \frac{ \eta_{2} + 1 }{ \eta_{2} } \wedge \frac{ \eta_{2} + 1 }{8} \Big) \lvert \mymat{a}{1}{2} \rvert \\  
        &\le \frac{1}{16} \lvert \mymat{a}{1}{2} \rvert + \frac{1}{3} \lvert \mymat{a}{1}{2} \rvert + \frac{3}{8} \lvert \mymat{a}{1}{2} \rvert  
        = \frac{37}{48} \lvert \mymat{a}{1}{2} \rvert. 
    \end{align*} 
    Moreover, by \( ( a + b + c )^{2} \le 3( a^{2} + b^{2} + c^{2} ) \) we have
        \begin{align*} 
            \frac{ ( \mymat{a}{1}{1} - \mymat{a}{3}{3} )^{2} }{ \lvert \mymat{a}{1}{2} \rvert }
            &\le \frac{ 3 } { \lvert \mymat{a}{1}{2} \rvert } \Big( \frac{ ( 1 + \epsilon_{r} )^{2} }{ \wh{\sigma}_{1}^{2} } \lVert A \rVert^{2} + \delta_{1}^{2} \lvert \mymat{a}{1}{2} \rvert^{2} + \delta_{2}^{2} \lvert \mymat{a}{4}{3} \rvert^{2} \Big) \\ 
            &\le \frac{ 3( 1 + \epsilon_{r} ) }{ 16\wh{\calK} } \lVert A \rVert + 3 \delta_{1}^{2} \lvert \mymat{a}{1}{2} \rvert +  \frac{27}{8} \delta_{2}^{2} \lvert \mymat{a}{4}{3} \rvert \\
            &\le \Big( \frac{ 3( 1 + \epsilon_{r} ) }{ 16\wh{\calK} } + \frac{ 6( 1 + \epsilon_{r} ) }{ \wh{\calK} } + \frac{ 27( 1 + \epsilon_{r} ) }{ 4\wh{\calK} } \Big) \lVert A \rVert 
            =  \frac{ 207( 1 + \epsilon_{r} ) }{ 16\wh{\calK} } \lVert A \rVert .
    \end{align*} 
    Here assumption~(\ref{case 2: assumption 2}) is used in the second inequality.  

    \textbf{Proof of (\ref{bound for phi(A) in case 2}).} If either~(\ref{case 2: assumption 1}) or~(\ref{case 2: assumption 2}) fails, we have 
    \begin{align*}
        &\lvert \mymat{a}{1}{2} \rvert 
        \le \frac{ 16( 1 + \epsilon_{r} ) }{ \wh{\calK} } \lVert A \rVert, \dspace 
        \lvert \mymat{a}{4}{3} \rvert 
        \le 2 \lvert \mymat{a}{1}{2} \rvert 
        \le \frac{ 32( 1 + \epsilon_{r} ) }{ \wh{\calK} } \lVert A \rVert, \\  
        &\lvert \mymat{a}{1}{1} - \mymat{a}{3}{3} \rvert 
        \le \frac{ 1 + \epsilon_{r} }{ \wh{\sigma}_{1} } \lVert A \rVert + \frac{ \sqrt{2} }{2} \lvert \mymat{a}{1}{2} \rvert + \dfrac{ \sqrt{2} }{2} \lvert \mymat{a}{4}{3} \rvert 
        \le \frac{ ( 1 + 24\sqrt{2} ) ( 1 + \epsilon_{r} ) }{ \wh{\calK} } \lVert A \rVert .
    \end{align*}
    Then 
    \[
    \lVert T \rVert_{\infty} 
    \le  \frac{ ( 49 + 24\sqrt{2} ) ( 1 + \epsilon_{r} ) }{ \wh{\calK} } \lVert A \rVert . 
    \]
    Mimicking (\ref{bound 0 for contribution of residual}), it can be shown that 
    \[
    \wh{\phi}( A ) 
    \le 32 \rho_{4} \gamma_{12} \lVert T \rVert_{\infty} \cdot \wh{\calK} 
    \le ( 6272 + 3072\sqrt{2} ) ( 1 + \epsilon_{r} ) \cdot \gamma_{12} \lVert A \rVert . 
    \] 
\end{proof}

\end{appendices}

\end{document}